\documentclass[12pt,twoside]{amsart}
\usepackage[margin=3cm]{geometry}
\usepackage[colorlinks=false]{hyperref}
\usepackage[english]{babel}
\usepackage{graphicx,titling}
\usepackage{float}
\usepackage{amsmath,amsfonts,amssymb,amsthm}
\usepackage{lipsum}
\usepackage[T1]{fontenc}
\usepackage{fourier}
\usepackage{color}
\usepackage[latin1]{inputenc}
\usepackage{esint}
\usepackage{caption}
\usepackage{ccicons}

\makeatletter
\def\blfootnote{\xdef\@thefnmark{}\@footnotetext}
\makeatother

\newcommand\ccnote{
    \blfootnote{\copyright\,\, Nicola Gigli}
    \blfootnote{\ccLogo\, \ccAttribution\,\, Licensed under a \href{https://creativecommons.org/licenses/by/4.0/}{Creative Commons Attribution License (CC-BY)}.}
}

\usepackage[export]{adjustbox}
\numberwithin{equation}{section}
\usepackage{setspace}

\renewcommand{\leq}{\leqslant}

\renewcommand{\geq}{\geqslant}
\renewcommand{\mathbb}{\varmathbb}
\usepackage{fancyhdr}
\newtheorem{theorem}{Theorem}[section]
\newtheorem{lemma}[theorem]{Lemma}

\newtheorem{proposition}[theorem]{Proposition}
\newtheorem{definition}[theorem]{Definition}
\newtheorem{remark}[theorem]{Remark}
\usepackage{amsmath}
\usepackage{amsfonts,amssymb}
\usepackage{mathrsfs}
\usepackage{bm}
\usepackage{esint}
\usepackage{amsthm}
\usepackage{enumitem}
\usepackage{tikz-cd}

\usepackage{geometry}
\usepackage{a4wide}
\usepackage{hyperref}

\newcommand{\N}{\mathbb{N}}
\newcommand{\Q}{\mathbb{Q}}
\newcommand{\R}{\mathbb{R}}

\newcommand{\ppi}{{\mbox{\boldmath$\pi$}}}

\newcommand{\sppi}{{\mbox{\scriptsize\boldmath$\pi$}}}
\newcommand{\restr}[1]{\lower3pt\hbox{$|_{#1}$}}
\newcommand{\pr}{\mathscr P}

\newcommand{\e}{{\rm e}}

\newcommand{\X}{{\rm X}}
\newcommand{\Y}{{\rm Y}}

\newcommand{\sfd}{{\sf d}}
\newcommand{\mm}{{\mathfrak m}}

\renewcommand{\S}{{\rm S}}

\newcommand{\CD}{{\sf CD}}
\newcommand{\RCD}{{\sf RCD}}

\renewcommand{\d}{{\rm d}}

\newcommand{\lip}{{\rm lip}}

\newcommand{\Lip}{{\rm Lip}}

\newcommand{\weakto}{\rightharpoonup}
\newcommand{\nchi}{{\raise.3ex\hbox{$\chi$}}}
\newcommand{\supp}{{\rm supp}}
\newcommand{\lims}{\varlimsup}
\newcommand{\limi}{\varliminf}
\newcommand{\eps}{\varepsilon}
\newcommand{\Id}{{\rm Id}}
\newcommand{\loc}{{\rm loc}}

\newcommand{\osc}{{\sf Osc}}
\newcommand{\HS}{{\sf HS}}

\newcommand{\st}[1]{\noindent{\bf  #1}}  

\newcommand{\h}{{\sf h}} 

\newcommand{\E}{{\sf E^{KS}}} 

\newcommand{\fr}{\penalty-20\null\hfill$\blacksquare$}         

\DeclareMathOperator*{\esssup}{\rm ess-sup}
\DeclareMathOperator*{\essinf}{\rm ess-inf}

\newcommand{\Fl}{{\sf Fl}}
\renewcommand{\div}{{\rm div}}

\newcommand{\KS}{{\sf KS}^2}
\newcommand{\ks}{{\sf KS}}

\newcommand{\kse}{{\sf ks}}

\newcommand{\la}{\langle}
\newcommand{\ra}{\rangle}
\newcommand{\CAT}{{\sf CAT}}

\newcommand{\bd}{{\bf \Delta}}

\newcommand{\testi}{{\rm  Test}^\infty(\X)}
\newcommand{\cont}{{\sf Cont}}

\newcommand{\bs}{{\sf bs}}
\renewcommand{\b}{{\sf b}}

\newcommand{\acm}{{\mathcal AC}^-}
\renewcommand{\o}{{\sf o}}
\newcommand{\tilt}{{\sf tilt}}
\newcommand{\norm}{{\sf Norm}}
\newcommand{\tvar}{\widetilde{\rm Test}}

\address{Nicola Gigli, SISSA, via Bonomea 265, Trieste, Italy}
\email{ngigli@sissa.it}

\begin{document}

\thispagestyle{empty}

\begin{minipage}{0.28\textwidth}
\begin{figure}[H]
\includegraphics[width=2.5cm,height=2.5cm,left]{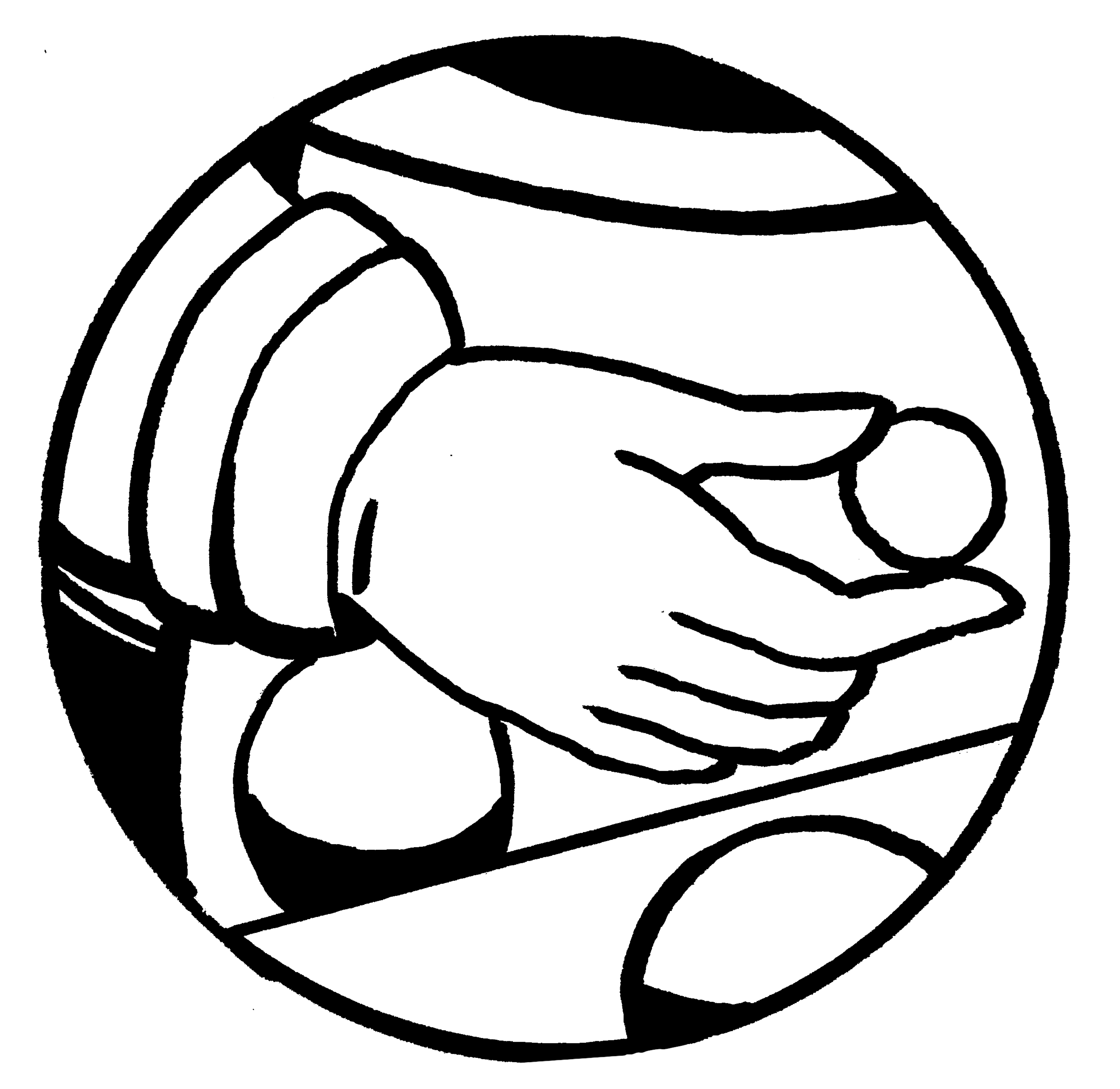}
\end{figure}
\end{minipage}
\begin{minipage}{0.7\textwidth} 
\begin{flushright}
Ars Inveniendi Analytica (2023), Paper No. 5, 55 pp.
\\
DOI 10.15781/sf2c-1y90
\\
ISSN: 2769-8505
\end{flushright}
\end{minipage}

\ccnote

\vspace{1cm}


\begin{center}
\begin{huge}
\textit{On the  regularity of harmonic maps from ${\RCD}(K,N)$ to ${\sf CAT}(0)$ spaces and related results}

\end{huge}
\end{center}

\vspace{0.4cm}


\begin{center}
{\large{\bf{Nicola Gigli}}} \\
\vskip0.15cm
\footnotesize{SISSA}
\end{center}

\vspace{0.4cm}


\begin{center}
\noindent \em{Communicated by Francesco Maggi}
\end{center}
\vspace{1cm}


\noindent \textbf{Abstract.} \textit{ 
For an harmonic map $u$ from a domain $U\subset\X$ in an $\RCD(K,N)$ space $\X$ to a $\CAT(0)$ space $\Y$ we prove the Lipschitz estimate}
\[
\Lip(u\restr B)\leq C(K^-R^2,N)\,r^{-1}\,\inf_{\o\in\Y}\,\sqrt{\fint_{2B}\sfd_\Y^2(u(\cdot),\o)\,\d\mm},\qquad\forall 2B\subset U,
\]
\textit{where $r\in(0,R)$ is the radius of $B$. This is obtained by combining classical Moser's iteration, a Bochner-type inequality that we derive (guided by recent works of Zhang-Zhu) together with a reverse Poincar\'e inequality that is also established here.  A direct consequence of our estimate   is a Lioville-Yau type theorem in the case $K=0$. Among the ingredients of the proof, a variational principle in general $\RCD$ spaces is particularly relevant. It can be roughly stated as: if $(\X,\sfd,\mm)$ is $\RCD(K,\infty)$ and $f\in C_b(\X)$ is so that $\Delta f\leq C$ for some constant $C>0$, then for every $t>0$ and $\mm$-a.e.\ $x\in\X$ there is a unique minimizer $F_t(x)$ for
$
y\ \mapsto\  f(y)+\frac{\sfd^2(x,y)}{2t}
$
and the map $F_t$ satisfies}
\[
(F_t)_*\mm\leq e^{t(C+2K^-\osc(f))}\mm,\qquad\text{where}\qquad\osc(f):=\sup f-\inf f.
\]
\textit{Here existence is in place without any sort of compactness assumption and uniqueness should be intended in a sense analogue to that in place for Regular Lagrangian Flows and Optimal Maps  (and is related to both these concepts). Finally, we also obtain  a  Rademacher-type result for Lipschitz maps between spaces as above.
}
\vskip0.3cm

\vspace{0.5cm}


\section{Introduction}

The study of harmonic maps and their regularity is a central topic in Geometric Analysis, with several ramifications depending on the kind of structure assumed on the  domains in considerations. In this paper, we are concerned with Lipschitz regularity in relation with lower Ricci curvature  bounds on the source space and upper sectional curvature bounds on the target. The inspiring work is the celebrated result \cite{ES64} by Eells-Sampson where among other things the authors noticed that in the smooth category  a negative curvature on the target domain plays `in favour' of Bochner inequality. Specifically, if the Ricci curvature of ${\rm M}$ is bounded from below by $K\in\R$, the sectional curvature of ${\rm N}$ is non-positive and $u:U\subset {\rm M}\to {\rm N}$ is sufficiently smooth and harmonic, they established what is  now-called Bochner-Eells-Sampson inequality, namely:
\begin{equation}
\label{eq:BES}\tag{BES}
\tfrac12\Delta |\d u|_\HS^2\geq K|\d u|^2_\HS,
\end{equation}
where $|\d u|_\HS$ is the Hilbert-Schmidt (also called Euclidean) norm of the differential $\d u$. Starting from this, it is not hard to derive quantitative Lipschitz regularity: from De Giorgi-Nash-Moser elliptic regularity theory (see e.g.\ \cite{GT01}, \cite{ACM18} for an overview) we see that the estimate
\begin{equation}
\label{eq:1}
\Lip(u\restr B)\leq C(K^-R^2,N)\sqrt{\fint_{2B}|\d u|_\HS^2\,\d \,{\rm vol} }
\end{equation}
holds if $B=B_r(x)$ is such that $2B=B_{2r}(x)$ is contained in $U$ and $r\in(0,R)$. Here $N$ is an upper bound on the dimension of ${\rm M}$ and the dependence of the constant on $K^-R^2$ and $N$ only can be seen by a scaling argument, recalling that  the constants in Moser's iteration only depend on those in the local doubling and Sobolev inequalities, and that these can be bounded in terms of a lower Ricci and upper dimension bounds (see e.g.\ \cite{Villani09} for a modern presentation of the topic, very related to the kind of discussion made in this paper).

At least in the scalar case ${\rm N}=\R$, one can then use the reverse Poincar\'e/Cacciopoli inequality
\begin{equation}
\label{eq:rpii}
\|\d u\|_{L^2(B)}\leq\frac {C(N)}r\|u\|_{L^2(2B)}
\end{equation}
to deduce from \eqref{eq:1} that
\begin{equation}
\label{eq:2}
\Lip(u\restr B)\leq\frac{C(K^-R^2,N)}{r}\sqrt{\fint_{4B}|u|^2\,\d \,{\rm vol}}
\end{equation}
provided $4B\subset U$. We notice that this last inequality provides a similar scaling, in terms of $r$ and the size of $u$, given by the celebrated Cheng-Yau's gradient estimate 
\[
\sup_{B}|\nabla \log f|\leq C(N)\Big(\sqrt K+\frac1r\Big)\qquad\text{ if $f> 0$ on $2B\subset U$}
\]
and, much like this one, it implies for $K=0$ that a bounded harmonic function on a manifold with non-negative Ricci curvature must be constant. For the case of maps with values on a simply connected manifold ${\rm N}$ with non-positive sectional curvature, a suitable gradient estimate and a consequent Lioville-Yau type of theorem has been obtained by Cheng in \cite{Cheng80}.

Let us underline that one of the many advantages of dealing with target spaces that are simply connected and with non-positive curvature is that the relevant energy is convex in a very natural sense. In particular a map is a critical point of such functional if and only if it is a minimizer: this basic fact allows to avoid more complicated taxonomies that  are otherwise present (see e.g.\ \cite[Chapter 7]{Jost17}) and suggests a variational way of defining harmonic maps in more singular settings.

\bigskip

Given the nature of the estimates and the kind of assumptions present, it is natural to wonder whether results like the above hold in the context of synthetic geometry: this sort of question has been raised several times in the literature, see e.g.\ \cite{Lin97}, \cite{Jost1997}. Here the counterpart for `lower Ricci and upper dimension bounds' on the source manifold is the $\RCD(K,N)$ condition (introduced in \cite{Gigli12}, after \cite{Lott-Villani09}, \cite{Sturm06I,Sturm06II}, \cite{AmbrosioGigliSavare11-2} - see the survey \cite{DGG} for more about the history of the topic), while that for `simply connected and non-positive sectional curvature bounds' on the target manifold is the $\CAT(0)$ notion (see e.g.\ \cite{Bacak14} for an overview on the topic).

In the last thirty years, studies in this kind of direction have attracted the interest of a number of mathematicians; without pretending to be  complete, we mention \cite{GS92}, \cite{KS93},  \cite{Jost1997},  \cite{Lin97}, \cite{Sturm97}, \cite{Jost98}, \cite{Sturm99}, \cite{Gregori98}, \cite{KS03}, \cite{KuwSt08}, \cite{CvS16}, \cite{Guo17v1}, \cite{ZZ18}, \cite{Freidin19}, \cite{ZZZ19}, see also the celebrated series of papers \cite{SU82}, \cite{SU82corr}, \cite{SU83}, \cite{SU84} for partial regularity results available without prescribed curvature bounds.  These papers contain key geometric/analytic insights but none of them cover the case of the - often more recent - theory of $\RCD$ spaces, to the point that even the existence of a suitable energy is a priori unclear. Notice in particular that the key `subpartition lemma' from \cite{KS93} may fail in this framework.
\bigskip

This motivated us to develop a research program \cite{GT20,GT18,DMGSP18,GPS18,gn20}, in collaboration with a diverse set of coauthors, aimed  at generalizing the above to its natural framework of $\RCD$ spaces.  In particular, in  \cite{GT20}, suitably adapting   Korevaar-Schoen's approach, we showed that the Dirichlet problem is well posed in this framework, so that at least the concept of `harmonic map from a domain in an $\RCD(K,N)$ space to a $\CAT(0)$ one' is well defined. The findings in \cite{GT20} have been used in an updated version \cite{Gu017v5} of \cite{Guo17v1} to show that such harmonic maps are H\"older continuous (inspired by ideas in  \cite{Jost1997},  \cite{Lin97} - in  \cite{Guo17v1} such regularity was obtained under a certain  technical assumption, shown in  \cite{Gu017v5}  to be satisfied  by $\RCD$ spaces using the results in \cite{GT20}). However, given that the arguments presented here would only be marginally affected by such continuity, in order to provide a more self-contained exposition we will not rely on such result.

With all this said, our main theorem is:
\begin{theorem}[See Theorem \ref{thm:main}]\label{thm:mainintro} Let $(\X,\sfd,\mm)$ be $\RCD(K,N)$, $(\Y,\sfd_\Y)$ be $\CAT(0)$, $U\subset\X$ open with $\mm(\X\setminus U)>0$ and $u:U\to \Y$ harmonic. Then the  Zhang-Zhong-Zhu inequality
\begin{equation}
\label{eq:ZZZ}\tag{ZZZ}
\Delta\frac{|\d u|^2}2\geq K|\d u|^2\qquad on\ U
\end{equation}
holds in the sense of distributions and  the map $u$ has a representative, still denoted by $u$, with
\begin{equation}
\label{eq:mainintro}
\Lip(u\restr B)\leq\frac {C(K^-R^2,N)}r\inf_{\o\in\Y}\,\sqrt{\fint_{2B}\sfd_\Y^2(u(\cdot),\o)\,\d\mm},
\end{equation}
for any ball $B=B_r(x)$ such that $2B=B_{2r}(x)\subset U$ and $r\in(0,R)$.
\end{theorem}
\noindent Some comments are in order:
\begin{itemize}[leftmargin=*]
\item[-] The first appearance of an inequality like \eqref{eq:ZZZ} is in the recent paper \cite{ZZZ19} (but see also \cite{ZZ18}), where it has been established for smooth domains and general $\CAT(\kappa)$, $\kappa>0$, target spaces under the - usual in this context - proviso that the image of $u$ is contained in a ball of radius $\frac{\pi}{2\sqrt \kappa}$ (to be more precise, in \cite{ZZZ19} also the non-negative term $|\nabla\lip u|^2$ appears at the right hand side). 

Notice that, even in the smooth category and for smooth maps, inequality \eqref{eq:ZZZ} is different from \eqref{eq:BES} as in the latter the Hilbert-Schmidt norm  $|\d u|_\HS$ of the differential $\d u$ appears, while in the former $|\d u|$  is the operator norm (see also \cite[Remark 1.5 - (1)]{ZZZ19}). It is unclear to us whether there is any direct link between the two. Let us just mention that they both reduce to the standard Bochner inequality $\Delta\frac{|\d f|^2}2\geq K|\d f|^2$ for harmonic functions in the case $\Y=\R$ and that at least for $\Y=\R^d$ they both can be deduced from this formula: \eqref{eq:BES} follows adding up Bochner inequality for the (harmonic) components of $u$, while for \eqref{eq:ZZZ} we apply the Bochner inequality to the harmonic functions $u\cdot {\sf v}$ for  ${\sf v}\in \S^{d-1}\subset\R^d$, notice that $|\d u|=\sup_{{\sf v}\in \S^{d-1}}|\d(u\cdot{\sf v})|$ and use the stability of lower bounds on the Laplacian under `sup' and the locality of the Laplacian to conclude (these technical arguments are valid even in the current context - see Lemmas \ref{le:inflap2}, \ref{le:lapcomb2} - and will be used in a somehow similar spirit).
\item[-] Starting from \eqref{eq:ZZZ} and via Moser's iteration it is immediate to deduce an inequality like \eqref{eq:1}. To achieve \eqref{eq:mainintro} we combine such argument with a reverse Poincar\'e/Cacciopoli type inequality, that to the best of our understanding has gone unnoticed even in the smooth category. It can be stated as: with the same notation of Theorem \ref{thm:mainintro} we have
\[
\int_B|\d u|^2\,\d\mm\leq \frac{C(N)}{r^2}\inf_{\o\in\Y}\int_{2 B}\sfd_\Y^2(u(x),\o)\,\d\mm.
\]
To prove this we follow ideas that are classical in the real-valued context. We recall that by the $\CAT(0)$ condition, for any $\o\in\Y$ the function $\sfd_o:=\sfd_\Y(\cdot,\o)$ is convex, thus  the harmonicity of $u$ implies $|\d u|^2\leq C(N)\Delta (\sfd_o^2\circ u)$ (by \cite{gn20} - see also \cite[Corollary 7.2,5]{Jost17}, \cite[Lemma 5]{Jost1997}). Multiplying this bound by $\varphi^2$, where $\varphi$ is a $\frac 2r$-Lipschitz map with support in $2B$ and identically 1 on $B$, integrating, integrating by parts and then using Young's inequality we conclude. See Proposition \ref{prop:invP} for the details.

\item[-] In \cite{ZZZ19} inequality \eqref{eq:ZZZ} is written with the (geo)metric quantity `local Lipschitz constant $\lip(u)$' in place of the analytically oriented `minimal weak upper gradient $|\d u|$'. We can state \eqref{eq:ZZZ} the way we did thanks to the Rademacher-like result:
\begin{equation}
\label{eq:radintro}
\text{for $\X,\Y,U$ as above and $v:U\to\Y$ locally Lipschitz we have }\lip(v)=|\d v|\quad\mm-a.e.,
\end{equation}
see Proposition \ref{prop:rad} and Remark \ref{re:rad}. This result should be compared to the analogous one obtained by Cheeger in \cite{Cheeger00} for real valued maps on PI spaces: here we trade more generality in the target space for more regularity on the source one.  Property \eqref{eq:radintro} is a quite direct consequence of the studies in \cite{GT20} and is related to both Kirchheim's notion of metric differential \cite{Kir94} and to the concept of differential of metric-valued maps introduced in \cite{GPS18}.

Because of \eqref{eq:radintro}, choosing whether to state \eqref{eq:ZZZ} with $\lip(u)$ or $|\d u|$ is mainly a matter of taste: we chose the latter because we believe it better fits the distributional nature of the inequality, but let us remark that the proof of such bound actually produces the quantity $\lip(u)$, and we use \eqref{eq:radintro} only at the very end of the proof to formulate \eqref{eq:mainintro} the way we did.
\item[-] Estimate   \eqref{eq:mainintro} immediately implies the following Liouville-Yau type of result:
\begin{theorem}
Let $\X$ be $\RCD(0,N)$ with $\supp(\mm)$ unbounded, $\Y$ be $\CAT(0)$  and $u:\X\to\Y$ be  harmonic with sublinear growth. Then $u$ is constant.

Here `harmonic' means that $u\restr U$ is harmonic for any $U\subset\X$ open bounded with $\mm(\X\setminus U)>0$ and the growth condition means that  for some $\bar x\in\X$, $\o\in\Y$ and function $f:\R^+\to\R^+$ with $\lim_{r\uparrow\infty}f(r)=0$ we have $\sup_{B_r(\bar x)}\sfd_\Y(u(x),\o)<rf(r)$ for every $r>0$. 
\end{theorem}
\begin{proof} By the assumptions  we can pick $B=B_r(\bar x)$ with $r=R$ arbitrarily large in  \eqref{eq:mainintro}. We then conclude from the fact that the right hand side is bounded by $C(0,N)f(2r)$.
\end{proof}
\end{itemize}
The proof of \eqref{eq:ZZZ} builds on top of the geometric backbone coming from \cite{ZZ18}, \cite{ZZZ19}, in turn inspired by \cite{Petrunin03}, (we refer to Section \ref{se:zzz} for both an outline and the actual  argument), analytic machinery previously built and  tools that we develop here. Among the latter ones are: 
\begin{itemize}
\item[-] a variational principle, 
\item[-] a study  of the relation between Laplacian upper bounds and  Hopf-Lax formula, 
\item[-] a clearer picture of the interplay between `pointwise' and `distributional' Laplacian bounds. 
\end{itemize}
These results are related one to the other and will be studied on general $\RCD(K,\infty)$ spaces.  Let us give some details.

The variational principle we prove (see Theorem \ref{prop:varregf}) can be roughly formulated as follows: let $(\X,\sfd,\mm)$ be an $\RCD(K,\infty)$ space, $f\in C_b(\X)$ be with $\Delta f\leq C$ and $T>0$. Then for $\mm$-a.e.\ $x\in\X$ there is a unique minimizer $F(x)$ for $f(\cdot)+\frac{\sfd^2(x,\cdot)}{2T}$ and the map $F$ satisfies
\begin{equation}
\label{eq:estbase}
(F_T)_*\mm\leq e^{T(C+2K^-\osc(f)}\mm\qquad\text{ where }\osc(f):=\sup f-\inf f.
\end{equation}
The relevance of this statement for our purposes is in this latter estimate, that allows to  perform a perturbation argument very much in the spirit of the classical Jensen's maximum principle (adapted to the Alexandrov setting in \cite{Petrunin03}, \cite{ZZ18}). Indeed, at some point in the proof we will have  a set $E\subset\X$ of `nice points', where some relevant quantity is controlled, of full measure and  a function $f$ with Laplacian bounded from above. We would like $f$ to attain a minimum in $E$, but this is a priori not granted. We thus apply the above with $t\gg1$ to ensure that for $\mm$-a.e.\ $x$ (and thus for at least 1) the perturbed function  $f(\cdot)+\frac{\sfd^2(x,\cdot)}{2t}$  has minimum in $E$. Notice that by the Laplacian comparison estimates for the squared distance, when performed on finite dimensional spaces this procedure does not really disrupt the assumption on the Laplacian upper bound. 

We point out that on Riemannian manifolds, estimates like \eqref{eq:estbase} - typically stated as lower bound for appropriate integrals over the `contact set' - have  been investigated by the community working on elliptic PDEs and viscosity solutions:  the idea of perturbing with the squared  distance (as opposed to the affine perturbations used in $\R^d$ that are unavailable in the curved setting) goes back to the seminal paper \cite{Cabre97}, and the role of Ricci curvature has already been realized, see \cite{WZ13}. In comparison with existing literature, our result  seems the first to be derived in a genuine infinite dimensional setting  and  we are not aware of  works indicating the link between the variational problem above and the optimal transport problem that we discuss below, see \eqref{eq:mapott}. One of the things one learns from such link is that the minimum is $\mm$-a.e.\ unique (but, admittedly, this seems to be of little use in general).

The idea for proving the above variation principle is the following. Say $T=1$, put 
\[
Q_tf(x):=\inf_yf(y)+\frac{\sfd^2(x,y)}{2t},
\]
assume for a moment to be in the smooth category and for $x\in\X$ let $\gamma$ be  solving $$\gamma_t'=-\nabla Q_{1-t}f(\gamma_t)$$ and starting from $x$. Then the standard theory for Hamilton-Jacobi equation (and a simple computation if everything is smooth) ensures that $\gamma_1$ is a minimizer for $Q_1f(x)$, hence
\begin{equation}
\label{eq:F1fl}
\text{the map $F_1$ can obtained as flow of the vector fields $(-\nabla Q_{1-t}f)$}
\end{equation}
and thus the bound \eqref{eq:estbase} will follow if we show that
\begin{equation}
\label{eq:lapqtintro}
\Delta Q_tf\leq  \|(\Delta f)^+\|_{L^\infty}+2K^-\osc(f)\qquad\forall t>0.
\end{equation}
Also, the general theory of Optimal Transportation ensures that
\begin{equation}
\label{eq:mapott}
\text{the resulting map $F_1$ is optimal and $Q_1f$ is a Kantorovich potential relative to it}.
\end{equation}
With this in mind, we point out that in  the recent \cite{GTT22}, by passing to the limit in the viscous approximation of the Hamilton Jacobi equation (see also \cite{GigTam18})  it has been proved  that
\begin{equation}
\label{eq:GTT}
\Delta Q_tf\leq \|(\Delta f)^+\|_{L^\infty}+tK^-\Lip(f)^2\qquad\forall t>0,
\end{equation}
at least for sufficiently good functions, see Lemma \ref{prop:GTT}. For $K\geq0$ this bound is equivalent to  \eqref{eq:lapqtintro} and is therefore an indication that the strategy outlined above might work. Here it is natural to take inspiration from the theory of Regular Lagrangian Flows (see \cite{Ambrosio-Trevisan14} after  \cite{DiPerna-Lions89}, \cite{Ambrosio04}) that, very shortly and roughly said, is satisfactory provided two conditions are met.  The first is the presence of a lower bound on the divergence of the vector fields, that provides an upper bound on the compression constant of the flow, the second is a uniqueness statement for the associated continuity equation, typically obtained under some Sobolev-type assumption on the vector fields.

In our setting and at least if $K\geq 0$ we have the first ingredient (estimate \eqref{eq:GTT}, that leads to \eqref{eq:estbase}), but we cannot really hope for any sort of Sobolev regularity. Still,  the particular features of our minimization problem permit to rely on a different, more geometric, sort of uniqueness result without the need of PDE considerations. Indeed, the link between our problem and Optimal Transport noticed in \eqref{eq:mapott} remains valid also on $\RCD(K,\infty)$ spaces  and allows us to use the known results about optimal maps in this framework (\cite{Gigli12a}, \cite{RajalaSturm12}). 

Notice that in particular we produce minimizers  without really minimizing a functional, but rather following the flow of some vector fields: this partially explains why we do not need any compactness. Still, inspecting our proof one sees  that the tightness encoded in estimate \eqref{eq:estbase}, and thus ultimately in the upper Laplacian bound,  plays a crucial role - similar to the role that tightness has in producing Regular Lagrangian Flows or Optimal Maps in this setting.

For $K<0$ the bound \eqref{eq:GTT} is not sufficient for our purposes, as we do not have Lipschitz estimates for our function $f$. To overcome this issue we borrow an idea from the recent \cite{MS21}; there it was used in the finite dimensional setting, but the principle remains valid even in infinite dimension. Putting $\tilde\Delta f(x):=\lims_{t\downarrow0}\frac{\h_tf(x)-f(x)}t$, where $\h_t$ is the heat flow, we have
\begin{equation}
\label{eq:MSintro}
Q_tf(x)=f(y)+\frac{\sfd^2(x,y)}{2t}\qquad\Rightarrow\qquad\tilde\Delta Q_tf(x)\leq \tilde \Delta f(y)-K\frac{\sfd^2(x,y)}{t}.
\end{equation}
This is proved starting from the  known (see \cite[Lemma 3.4]{AmbrosioGigliSavare12} and its proof) bound 
\[
\h_sQ_tf(x)\leq \h_sf(y)+e^{-2Ks}\frac{\sfd^2(x,y)}{2t}
\]
valid for any couple of $x,y\in\X$ and $t,s\geq0$, and differentiating it in $s=0$  at points as in \eqref{eq:MSintro}, i.e.\ for which equality holds for $s=0$. We remark that a bound analogue to \eqref{eq:MSintro} was obtained in \cite{ZZ18} (see also \cite{Petrunin03}) and was crucial in their proof of Lipschitz regularity. Since we follow the strategy in \cite{ZZ18}, the estimate \eqref{eq:MSintro} is important for us, regardless of its application to the variational principle we are discussing: see the proof of Proposition \ref{prop:lapft} and notice that the $K$ appearing in \eqref{eq:ZZZ} is `the same' $K$ that appears in \eqref{eq:MSintro}.

Now observe that for  $x,y$ as in \eqref{eq:MSintro} we must have  $\frac{\sfd^2(x,y)}{2t}\leq\osc(f)$, hence the bound \eqref{eq:MSintro}   hints toward \eqref{eq:lapqtintro}.  Still, the natures of these inequalities are different: the former is a pointwise bound, while the latter should be intended in the weak sense of integration by parts (as in \cite{Gigli12}). Therefore part of the job we shall do here is to put in communication these a priori different notions. A result in this direction that we shall use frequently can roughly be stated as: 
\begin{quote} coupling a `bad' distributional bound with a `good' pointwise bound valid $\mm$-a.e.\ yields a `good' distributional bound,\end{quote} see Lemmas \ref{le:lapcomb}, \ref{le:lapcomb2} for the precise formulation. Using this principle we can combine the distributional bound \eqref{eq:GTT} with the pointwise bound \eqref{eq:MSintro} and deduce the distributional bound \eqref{eq:lapqtintro}, as desired.  We remark that this sort of analysis also permits to deduce the \eqref{eq:ZZZ} inequality from purely distributional-type consideration, avoiding the viscous/comparison arguments used in \cite{ZZ18}, \cite{ZZZ19}. Let us also point out that links between distributional/ pointwise/ viscous notions of Laplacian - well understood in the smooth category - have been investigated in the $\RCD$ framework in \cite{MS21}, with  results  mostly in the setting of non-collapsed $\RCD(K,N)$ spaces (\cite{GDP17}), see e.g.\ \cite[Theorems 1.3, 1.5]{MS21}; the analysis performed here strongly suggests that neither non-collapsing nor finite dimensionality are actually needed, see e.g.\ Lemma \ref{le:lapcomb} and Remark \ref{re:infdim}.

\medskip

We conclude noticing that Theorem \ref{thm:mainintro} leaves open some important questions, beside the expected generalization to harmonic maps in $\CAT(\kappa)$ spaces, $\kappa>0$. For instance: can one establish the Bochner-Eells-Sampson formula \eqref{eq:BES} in this setting? See e.g.\ \cite{Freidin19}, \cite{FY20} for partial results. More importantly: can one prove a version of this formula for non-harmonic maps? Notice that this would require in particular to give a meaning to the `gradient of the Laplacian' of any such map. A positive answer to this latter question seems necessary if one wants to develop  a regularity theory for the harmonic map heat flow in this setting.

\medskip

I wish to thank Prof. Hui-Chun Zhang for interesting comments on a preliminary version of this work and the referee for a very careful review and detailed comments.

\medskip

\emph{When the work on this paper was nearly completed I got knowledge of the related independent work \cite{MS22} containing partially overlapping results.}

\section{Preliminaries}
In this paper we study harmonic maps from $\RCD(K,N)$ spaces to $\CAT(0)$ spaces. Here we recall their definition and some basic calculus tool. More refined notions/results will be collected along the text.
\subsection{$\RCD$ spaces and calculus on them}\label{se:introRCD}

$\RCD(K,N)$ spaces are a non-smooth counterpart of the notion of Riemannian manifold with Ricci curvature $\geq K$ and dimension $\leq N$.  A key advantage of dealing with possibly non-smooth spaces, pointed out by Gromov in \cite{Gro81}, is that one gains natural compactness properties, the archetypical result being Gromov's precompactness theorem. Starting from this, Colding and Cheeger-Colding studied in the nineties limits of smooth Riemannian manifolds with uniform lower Ricci and upper dimension bounds (\cite{Colding97}, \cite{Cheeger-Colding96},  \cite{Cheeger-Colding97I}, \cite{Cheeger-Colding97II}, \cite{Cheeger-Colding97I}). 

Later, inspired by the results in \cite{CEMCS01}, \cite{OV00}, \cite{vRS05},  Lott-Villani \cite{Lott-Villani09} and Sturm \cite{Sturm06I}, \cite{Sturm06II} independently introduced the class of $\CD(K,N)$ spaces, these being metric measure spaces where the lower Ricci and the upper dimension bounds are interpreted via suitable convexity properties of entropy functionals defined on the Wasserstein space. This class of spaces contains Finsler manifolds, that are known to be not Ricci limit spaces thanks to Cheeger-Colding's almost splitting theorem. 

The need of finding a stable notion ruling out Finsler geometries was one of the motivations for which, under the influence of \cite{JKO98} and conversations with Sturm, I started studying the heat flow on  $\CD$ spaces. The resulting series of papers \cite{Gigli10}, \cite{Gigli-Kuwada-Ohta10}, \cite{AmbrosioGigliSavare11}, \cite{AmbrosioGigliSavare11-2} culminated in \cite{Gigli12} where I proposed to focus on $\CD$ spaces that are also infinitesimally Hilbertian, i.e.\ such that the Sobolev space $W^{1,2}(\X)$ of real valued functions is Hilbert (see also below). Proofs of concept of the geometric significance of this proposal came first in \cite{Gigli-Mosconi12} with the proof of the Abresch-Gromoll inequality (answering to an open criticism raised in \cite{Petrunin11} about $\CD$ spaces) and then in \cite{Gigli13} with the proof of the splitting theorem. See also \cite{AmbrosioGigliMondinoRajala12} for the reconciliation of this approach to the $\RCD$ definition with the original one made in \cite{AmbrosioGigliSavare11-2}   in the infinite dimensional case, and \cite{GMS15} for  stability results about the heat flow that, among other things, generalize those obtained in \cite{Gigli10} in the compact setting.

In a different direction, and much earlier, Bakry-\'Emery studied in the eighties Ricci curvature bounds for diffusion operators, in particular introducing the concept of Curvature-Dimension condition, see \cite{BakryEmery85} and the survey \cite{BakryGentilLedoux14}. Their approach was tailored to Dirichlet forms, but after  \cite{Gigli-Kuwada-Ohta10} and \cite{AmbrosioGigliSavare11-2}  it became viable in the metric-measure setting and was taken in \cite{AmbrosioGigliSavare12} as possible alternative `Eulerian' definition of lower Ricci curvature bounds, called ${\sf BE}(K,N)$ condition after Bakry-\'Emery. In  \cite{AmbrosioGigliSavare12} it has been also proved that $${\sf BE}(K,\infty)=\RCD(K,\infty)\,.$$ This result has been extended to the finite dimensional case in \cite{Erbar-Kuwada-Sturm13} (and later in \cite{AmbrosioMondinoSavare13}). In studying the interplay between $W_2$-convexity of entropies and Bochner inequality it is also useful to consider the so-called `reduced' curvature dimension condition $\CD^*(K,N)$ introduced in \cite{BacherSturm10} and that in a large class of spaces, including $\RCD$ ones, this notion has been proved to be equivalent to the $\CD(K,N)$ one in \cite{CavMil16} (at least for normalized spaces).

For an overview of the theory of $\RCD$ spaces see the surveys \cite{Villani2017}, \cite{AmbICM}, \cite{DGG}.

\medskip

We shall work with real-valued and metric-valued Sobolev maps defined on such space $\X$. Here I recall the concept of real-valued Sobolev function, see Section \ref{se:sobmet} for the metric case. Notions of higher order calculus (e.g.\ Laplacian, Hessian) will be recalled in the body of the work when needed. There are various possible equivalent approaches one can take to define first-order Sobolev functions in our setting (developed originally in \cite{Cheeger00} then in \cite{Shanmugalingam00}  and later in \cite{AmbrosioGigliSavare11}), here I shall recall one introduced in \cite{AmbrosioGigliSavare11} as it works better with the arguments in this manuscript.

Let $(\X,\sfd,\mm)$ be a complete and separable metric space equipped with a non-negative Borel measure finite on bounded sets. With $C([0,1],\X)$ I shall denote the complete and separable space of continuous curves on $[0,1]$ with values in $\X$, equipped with the `sup' distance. For $t\in[0,1]$, $\e_t:C([0,1],\X)\to\X$ denotes the \emph{evaluation map} sending $\gamma$ to $\gamma_t$.
\begin{definition}[Test plans]
A test plan is a Borel probability measure $\ppi$ on $C([0,1],\X)$ such that
\[
\begin{split}
(\e_t)_*\ppi&\leq C\mm\qquad\forall t\in[0,1],\\
\iint_0^1|\dot\gamma_t|^2\,\d t\,\d\ppi(\gamma)&<\infty,
\end{split}
\]
for some $C>0$.
\end{definition}
In the above,  the object $|\dot\gamma_t|$ is the so-called \emph{metric speed} of the absolutely continuous curve $\gamma$ (see \cite{Ambr90}), and it is part of the requirements that $\ppi$ is concentrated on absolutely continuous curves.

The concept of Sobolev function is given in duality with that of test plan:
\begin{definition}[Sobolev functions] A Borel function $f:\X\to\R$ belongs to the Sobolev class $S^2(\X)$ provided there is $G\in L^2(\X)$, $G\geq 0$ such that
\begin{equation}
\label{eq:defsob}
\int|f(\gamma_1)-f(\gamma_1)|\,\d\ppi(\gamma)\leq\iint_0^1G(\gamma_t)|\dot\gamma_t|\,\d t\,\d\ppi(\gamma)\qquad\forall\ppi\text{ test plan.}
\end{equation}
\end{definition}
It can be proved that there is a minimal $G$ in the $\mm$-a.e.\ sense: it will be denoted $|\d f|$ and called \emph{minimal weak upper gradient} of $f$. It satisfies natural calculus rules, mimicking those valid for the `modulus of the distributional differential of a Sobolev function' in the smooth world can be developed. These are better seen via the properties of the differential of Sobolev functions that we are now going to describe (still, notice that from a conceptual standpoint in order to prove the calculus rules for the differential, at least some knowledge of the calculus rules for the minimal weak upper gradient are needed). 

The Sobolev space $W^{1,2}(\X)$ is defined as $L^2(\X)\cap S^2(\X)$ and equipped with the norm
\begin{equation}
\label{eq:normw12}
\|f\|_{W^{1,2}}^2:=\|f\|_{L^2}^2+\||\d f|\|^2_{L^2}.
\end{equation}
It is always a Banach space and $(\X,\sfd,\mm)$ is called \emph{infinitesimally Hilbertian} provided $W^{1,2}(\X)$ is Hilbert (see \cite{Gigli12}).

We introduce the following concept (see \cite{Gigli14}):
\begin{definition}[$L^2$-normed $L^\infty$-module] An $L^2$-normed $L^\infty$-module on $\X$ is a Banach space $(\mathcal M,\|\cdot\|_{\mathcal M})$ that is also a module 
over the commutative ring with unity $L^\infty(\X,\mm)$ and which possesses a \emph{pointwise $L^2$ norm}, i.e.\ a map $|\cdot|:\mathcal M\to L^2(\X,\mm)$ such that
\[
\begin{split}
|v|&\geq 0,\\
|fv|&=|f|\,|v|,\\
\|v\|_{\mathcal M}^2&=\int|v|^2\,\d\mm,
\end{split}
\]
for every $v\in \mathcal M$ and $f\in L^\infty(\X,\mm)$, the first two identities being intended $\mm$-a.e.. We say that $\mathcal M$ is \emph{Hilbertian} if, when seens as Banach space, is an Hilbert space.
\end{definition}
It is not hard to check that $\mathcal M$ is Hilbertian if and only if
\[
2(|v|^2+|w|^2)=|v+w|^2+|v-w|^2\qquad\mm-a.e.,
\]
holds for any $v,w\in \mathcal M$ and in this case the natural polarization identity $\la v,w\ra:=\tfrac12(|v+w|^2-|v|^2-|w|^2)$ defines a pointwise scalar product, i.e.\ an $L^\infty$-bilinear map from $\mathcal M$ to $L^1(\X,\mm)$ satisfying the expected (pointwise a.e.) calculus rules.

Modules defined this way possess a natural dual $\mathcal M^*$, defined as the collection of $L^\infty$-linear and continuous maps $L:\mathcal M\to L^1(\X,\mm)$, equipped with the natural multiplication with functions in $L^\infty$ and with the pointwise norm
\[
|L|_*:=\esssup_{|v|\leq 1\ \mm-a.e.}|L(v)|.
\]
Hilbertian modules are canonically isomorphic to their duals via a Riesz' isomorphism: for any $L\in\mathcal M^*$ there is a unique $v\in\mathcal M$ such that $L(w)=\la v,w\ra$ $\mm$-a.e.\ for any $w\in \mathcal M$ and it holds $|v|=|L|_*$.

These structures are linked to Sobolev calculus via the following theorem:
\begin{theorem}
Let $(\X,\sfd,\mm)$ be a complete and separable metric space equipped with a non-negative and non-zero Borel measure giving finite mass to bounded sets. Then there is a unique, up to unique isomorphism, couple $(L^2(T^*\X),\d)$ with $L^2(T^*\X)$ module in the sense above and $\d:S^2(\X)\to L^2(T^*\X)$ linear and such that
\begin{itemize}
\item the pointwise norm of $\d f$ equals the minimal w.u.g.\ of $f$ for any $f\in S^2(\X)$,
\item $L^\infty$-linear combinations of elements in $\{\d f:f\in S^2(\X)\}$ are dense in $L^2(T^*\X)$.
\end{itemize}
\end{theorem}
The module $L^2(T^*\X)$ is called \emph{cotangent module} and the operator $\d$ \emph{differential}. The differential satisfies
\begin{align*}
\d f&=\d g&&\mm-a.e.\ on\ \{f=g\}, \ \forall f,g\in S^2(\X),\\
\d(\varphi\circ f)&=\varphi'\circ f\d f,&&\forall f\in S^2(\X),\ \varphi\in\Lip(\R),\\
\d(f g)&=f\d g+g\d f&&\forall f,g\in S^2\cap L^\infty(\X).
\end{align*}
The dual of $L^2(T^*\X)$ is called \emph{tangent module} and denoted $L^2(T\X)$. It turns out that $(\X,\sfd,\mm)$ is infinitesimally hilbertian if and only if $L^2(T^*\X)$ is Hilbertian. In this case, for $f\in S^2(\X)$ we denote by $\nabla f\in L^2(T\X)$ the element corresponding to $\d f$ via the Riesz isomorphism. Elements of $L^2(T\X)$ are called \emph{vector fields} on $\X$. The (opposite of the) adjoint of the differential is called \emph{divergence}, in other words a vector field $v$ belongs to the domain of the divergence provided there is $h\in L^2(\X)$ such that
\[
\int fh\,\d\mm=-\int \d f(v)\,\d\mm\qquad\forall f\in W^{1,2}(\X).
\]
In this case it is easily seen that $h$ is uniquely defined: we shall denote it ${\rm div}(v)$. Then on infinitesimally Hilbertian spaces we define $D(\Delta)\subset W^{1,2}(\X)$ as the space of functions $f$ such that $\nabla f\in D({\rm div})$ and put $\Delta f:={\rm div}(\nabla f)$.

We also recall the definition of $W^{1,2}(U)$ for $U\subset\X$ open with $\mm(U)>0$:
\begin{definition}[The spaces $W^{1,2}(U)$ and $W^{1,2}_0(U)$] The space
$W^{1,2}(U)$ consists  of those functions $f\in L^2(U)$ for which there is $G\geq 0$, $G\in L^2(U)$ such that for any Lipschitz function $\varphi:\X\to\R$ with support bounded and contained in $U$ we have $\varphi f\in W^{1,2}(\X)$ with $|\d (\varphi f)|\leq G$ $\mm$-a.e.\ on $\{\varphi=1\}$. The least, in the $\mm$-a.e.\ sense, such function $G$ is denoted $|\d f|$ and then the $W^{1,2}(U)$-norm is defined as in \eqref{eq:normw12}.

The space $W^{1,2}_0(U)$ is the $W^{1,2}(U)$-completion of the space of functions $f\in W^{1,2}(U)$ with support at positive distance from $\partial U$ (i.e.\ $\inf_{x\in\supp(f),y\in\partial U}\sfd(x,y)>0$).
\end{definition}
Here the fact that $|\d f|$ is well defined and that $W^{1,2}(U)$, $W^{1,2}_0(U)$ are Banach spaces can easily be proved starting from the analogous properties of globally defined Sobolev functions.

\medskip

We conclude this introduction recalling the `Eulerian' approach to  $\RCD$ spaces (see the already mentioned \cite{Gigli-Kuwada-Ohta10}, \cite{AmbrosioGigliSavare11-2}, \cite{AmbrosioGigliSavare12}, \cite{Erbar-Kuwada-Sturm13}, \cite{AmbrosioMondinoSavare13}):
\begin{definition}
Let $K\in\R$, $N\in[1,\infty]$ and $(\X,\sfd,\mm)$ a complete and separable metric space equipped with a non-negative, and non-zero, Borel measure finite on bounded sets. We say that $\X$ is an $\RCD(K,N)$ space provided:
\begin{itemize}
\item[i)] For some $C>0$ and $\bar x\in\X$ we have $\mm(B_r(\bar x))\leq Ce^{Cr^2}$ for any $r\geq0$.
\item[ii)] Any $f\in W^{1,2}(\X)$ with $|\d f|\leq 1$ $\mm$-a.e.\ admits a 1-Lipschitz representative.
\item[iii)] The space is infinitesimally Hilbertian.
\item[iv)] For any $f,g\in D(\Delta)$ with $\Delta f\in W^{1,2}(\X)$, $g,\Delta g\in L^\infty(\X)$, $g\geq 0$ we have
\[
\tfrac12\int |\d f|^2\Delta g\,\d\mm\geq\int g\big(\tfrac1N(\Delta f)^2+\la\d f,\d\Delta f\ra+K|\d f|^2\big)\,\d\mm.
\]
\end{itemize}
\end{definition}

\subsection{$\CAT(0)$ spaces}

$\CAT(0)$ spaces are the non-smooth analogue of simply connected Riemannian manifolds with non-positive Sectional curvature. 

Recall that a metric space $(\Y,\sfd_\Y)$ is said \emph{geodesic} provided for any $x,y\in\Y$ there is a curve $\gamma$, called constant speed geodesic from $x$ to $y$, which satisfies 
\[
\sfd_\Y(\gamma_t,\gamma_s)=|s-t|\sfd_\Y(\gamma_0,\gamma_1),\quad\forall t,s\in[0,1],\qquad\gamma_0=x,\ \gamma_1=y.
\]
\begin{definition} A $\CAT(0)$ space is a complete geodesic space $(\Y,\sfd_\Y)$ such that for any geodesic $\gamma$ and any $z\in\Y$ we have
\[
\sfd_\Y^2(z,\gamma_t)\leq(1-t)\sfd_\Y^2(z,\gamma_0)+t\sfd^2_{\Y}(z,\gamma_1)-t(1-t)\sfd^2_\Y(\gamma_0,\gamma_1)\qquad\forall t\in[0,1].
\]
\end{definition}
The geometric structure of $\CAT(0)$ spaces is crucial for the development of a (Lipschitz) regularity theory for harmonic maps. This will be evident from the discussion in Chapter \ref{ch:main} and from the references collected therein. For an introduction to the geometric and analytic properties of $\CAT(0)$ spaces we refer to \cite{Bacak14}.

\section{A variant of the Hopf-Lax formula}\label{se:HLnew}
The content of this section is strongly inspired by the discussions in \cite{AmbrosioGigliSavare11}, \cite{NCS14}, \cite{ZZ18} and constitutes a simple variant of the results therein.

Here we shall always assume that   $(\X,\sfd)$  is  a complete length (i.e.\ the distance is realized as infimum of length of curves) metric space and $f:\X\times\X\to \R\cup\{+\infty\}$ is a lower semicontinuous function bounded from below satisfying the reverse triangle inequality
\begin{equation}
\label{eq:rti}
f(x,z)\geq f(x,y)+f(y,z)\qquad\forall x,y,z\in \X.
\end{equation}
In our application $f$ will (almost) be of the form $f(x,y)=-\sfd_\Y(u(x),u(y))$, where $u$ is a map from $\X$ to some other metric space $\Y$, see \eqref{eq:deff}, \eqref{eq:defdu}.
For $x,y\in\X$ and $t>0$ we define $F(t,x;y)\in\R$ as
\begin{equation}
\label{eq:defft}
F(t,x;y):=f(x,y)+\frac{\sfd^2(x,y)}{2t}\qquad \text{and then}\qquad f_t(x):=\inf_{y\in\X} F(t,x;y).
\end{equation}
We also put $f_0(x):=f(x,x)$ for every $x\in\X$. In other words, if  for given $x\in\X$ we  denote by $f^x:\X\to  \R\cup\{+\infty\}$ the function defined as $f^x(y):=f(x,y)$ and we recall the usual metric Hopf-Lax formula
\[
Q_tg(x):=\inf_{y\in\X}g(y)+\frac{\sfd^2(x,y)}{2t},\qquad\qquad Q_0g(x)=g(x),
\]
the above definition reads as
\begin{equation}
\label{eq:idqt}
f_t(x)=(Q_tf^x)(x).
\end{equation}
Our aim in this section is to study the map $(t,x)\mapsto f_t(x)$. To this aim, let us introduce the functions $D^{\pm}:\X\times(0,\infty)\to[0,\infty)$ as
\begin{equation}
\label{eq:defdpm}
\begin{split}
D^+_t(x)&:=\max_{(y_n)}\lims_{n\to\infty}\sfd(x,y_n),\\
D^-_t(x)&:=\min_{(y_n)}\limi_{n\to\infty}\sfd(x,y_n),
\end{split}
\end{equation}
where in both cases $(y_n)$ varies among minimizing sequences   of $F(t,x;\cdot)$. Standard diagonal arguments show that the $\max$ and $\min$ in the definitions of $D^\pm$ are achieved while the boundedness from below of $f$ grants that $D^\pm_t(x)<\infty$ for any $t,x$.

Identity \eqref{eq:idqt} and the know result concerning the metric Hopf-Lax formula immediately give the following:
\begin{proposition}\label{prop:paseHL}
With the above notation and assumptions, for every $x\in\X$ the following holds:
\begin{itemize}
\item[i)] The function $t\mapsto f_t(x)$ is continuous on $[0,\infty)$, locally Lipschitz and locally semiconcave on $(0,\infty)$ and satisfies
\begin{equation}
\label{eq:derdshl}
\frac{\d^-}{\d t}f_t(x)=-\frac{(D^-_t(x))^2}{2t^2}\qquad\qquad\frac{\d^+}{\d t}f_t(x)=-\frac{(D^+_t(x))^2}{2t^2}.
\end{equation}
\item[ii)] The function $t\mapsto D^+_t(x)$ (resp.\ $D^-_t(x)$) is upper (resp.\ lower) semicontinuous, converge to 0 as $t\downarrow0$ and satisfy
\begin{equation}
\label{eq:mondpm}
D^-_t(x)\leq D^+_t(x)\leq D^-_s(x)\qquad\forall 0\leq t< s.
\end{equation}
Moreover, the identity $D^+(t)=D^-(x)$ holds for every $t$ except at most a countable number.
\end{itemize}
\end{proposition}
\begin{proof} Considering the function $g(\cdot):=f(x,\cdot)$, all the stated results follow from the analysis carried out in \cite[Section 3]{AmbrosioGigliSavare11}. Specifically, 
the monotonicity \eqref{eq:mondpm} and the last claim in item $(ii)$ are proved in \cite[Proposition 3.1]{AmbrosioGigliSavare11}, the semicontinuity in  \cite[Proposition 3.2]{AmbrosioGigliSavare11} and the convergence to 0 is obvious from the fact that $f$ is bounded from below (see also the explicit bound \eqref{eq:stimadqt2}.
%
\end{proof}
We now want to estimate $\frac{D^+_t(x)}{t}$ from above and to this aim  we recall that the  ascending/ descending slopes $|\partial^\pm g|(x)$ of the function $g:\X\to\R$ at the point $x\in\X$ are defined as
\[
|\partial^+g|(x):=\lims_{y\to x}\frac{(g(y)-g(x))^+}{\sfd(x,y)}\qquad\text{ and }\qquad |\partial^-g|(x):=\lims_{y\to x}\frac{(g(y)-g(x))^-}{\sfd(x,y)},
\]
where $z^+:=z\vee 0$, $z^-:=(-z)\vee0$ are the positive and negative parts of the real number $z$, respectively. We shall also define the `tilt' $\tilt (f)(x)$  of $f$ at $x$ as
\begin{equation}
\label{eq:tilt}
\tilt(f)(x):=\lims_{y\to x}\frac{(f(x,y))^-}{\sfd(x,y)}
\end{equation}
Rephrasing \cite[Sublemma 6.16]{ZZ18} we then have (notice that here the assumption \eqref{eq:rti} plays a role):
\begin{lemma}\label{le:stimalipa}
We have
\[
\frac{D^+_t(x)}t\leq |\partial^- f_t|(x)+\tilt(f)(x)\qquad\forall t>0,\ x\in\X.
\]
\end{lemma}
\begin{proof}  If $D_t^+(x)=0$ there is nothing to prove, thus we assume $D_t^+(x)>0$. Let $x\in\X$, $t>0$ and $(y_n)\subset\X$ a minimizing sequence for $F(t,x;\cdot)$ such that $\sfd(x,y_n)\to D_t^+(x)$. For every $n\in\N$, let $\gamma_n:[0,1]\to\X$ be with constant speed equal to $\sfd(x,y_n)(1+\frac1n)$ such that $\gamma_n(0)=x$, $\gamma_n(1)=y_n$. Then $\sfd(\gamma_n(s),y_n)\leq (1-s)(1+\frac1n)\sfd(x,y_n)$ for every $s,n$ and thus
\[
\begin{split}
\lims_n  f_t(x)-f_t(\gamma_n(s))&\geq \lims_n \big(F(t,x;y_n)-F(t,\gamma_n(s);y_n)\big)\\
&=\lims_n\big( f(x,y_n)- f(\gamma_n(s),y_n)  +\frac1{2t}\big(\sfd^2(x,y_n)-\sfd^2(\gamma_n(s),y_n)\big)\big)\\
\text{(by \eqref{eq:rti})}\qquad&\geq \lims_n f(x,\gamma_n(s))  +\frac{\sfd^2(x,y_n)}{t}\big(s-\tfrac{s^2}2\big)\\
&\geq\limi_n f(x,\gamma_n(s))  +\frac{D_t^+(x)^2}{t}\big(s-\tfrac{s^2}2\big)\geq -\lims_n f(x,\gamma_n(s))^-  +\frac{D_t^+(x)^2}{t}\big(s-\tfrac{s^2}2\big)
\end{split}
\]
Rearranging, dividing by $sD_t^+(x)>0$ and recalling that $\sfd(x,\gamma_n(s))\leq s\sfd(x,y_n)(1+\frac1n)\to sD_t^+(x)$ we get
\[
\frac{D_t^+(x)}{t}\leq \lims_n\frac{ f_t(x)-f_t(\gamma_n(s))}{\sfd(x,\gamma_n(s))}+\lims_n \frac{f(x,\gamma_n(s))^- }{\sfd(x,\gamma_n(s))}+\frac s2\frac{D_t^+(x)}{t}\qquad\forall s\in(0,1).
\]
By diagonalization we can find sequences $n_j,m_j\uparrow+\infty$ such that
\[
\frac{D_t^+(x)}{t}\leq \lims_j\frac{ f_t(x)-f_t(\gamma_{n_j}(\tfrac1j))}{\sfd(x,\gamma_{n_j}(\tfrac1j))}+\lims_j \frac{f(x,\gamma_{m_j}(\tfrac1j))^- }{\sfd(x,\gamma_{m_j}(\frac1j))}\leq |\partial^- f_t|(x)+\tilt(f)(x),
\]
as desired.
\end{proof}
Collecting what proved so far we obtain (recall that $f_0(x):=f(x,x)$):
\begin{theorem}\label{thm:hlvar} For any $x\in\X$ and $t\geq 0$ we have: 
\[
|f_0(x)-f_t(x)|\leq \frac12\int_0^t\big( |\partial^- f_s|(x)+\tilt(f)(x)\big)^2\,\d s
\]
\end{theorem}
\begin{proof} The inequality $f_0(x)\geq f_t(x)$ is obvious by the definitions, while from the regularity   established in item $(i)$ of Proposition \ref{prop:paseHL} we know that
\[
f_0(x)-f_t(x)= \int_0^t-\frac{\d^+}{\d s}f_s(x)\,\d s=\int_0^t\frac{D_s^+(x)^2}{2s^2}\,\d s.
\]
The  conclusion follows from Lemma \ref{le:stimalipa}.
\end{proof}

We conclude with the following `duality formula' for $\tilt(f)$, analogue of \cite[Lemma 3.1.5]{AmbrosioGigliSavare08}:
\begin{proposition}\label{prop:dualtilt} For any $x\in\X$ we have
\begin{equation}
\label{eq:dualtilt}
\frac12\tilt(f)^2(x)=\lims_{t\downarrow0}-\frac{f_t(x)}t=\lims_{t\downarrow0}\frac{D^\pm(t,x)^2}{2t^2}
\end{equation}
and a sequence $t_n\downarrow0$ realizes the $\lims$ of $-\frac{f_t(x)}t$ if and only if it does so for $\frac{D^\pm(t,x)^2}{2t^2}$.
\end{proposition}
\begin{proof} Letting $y_t$ be (almost) minimizers for $f_t(x)$ we get
\begin{equation}
\label{eq:at}
\begin{split}
0&\leq\lims_{t\downarrow0}-\frac{f_t(x)}{t}=\lims_{t\downarrow0}\Big(-\frac{f(x,y_t)}{t}-\frac{\sfd^2(x,y_t)}{2t^2}\Big)\\
&\leq \lims_{t\downarrow0}\underbrace{\Big(\tilt (f)(x)\frac{\sfd(x,y_t)}{t}-\frac{\sfd^2(x,y_t)}{2t^2}\Big)}_{=a_t}\stackrel *\leq\frac{\tilt(f)^2(x)}2,
\end{split}
\end{equation}
where the last inequality comes from Young's inequality (and we shall use at the end of the proof that $*$ is actually an equality). Similarly with the $\limi$. If $\tilt(f)(x)=0$ this gives   the first equality in \eqref{eq:dualtilt}, that  $a_t=-\frac{\sfd^2(x,y_t)}{2t^2}$ and that all the $\lims$ must be limits. Thus by  \eqref{eq:mondpm} we get the second inequality in  \eqref{eq:dualtilt}.

 We thus we assume $\tilt(f)(x)>0$ and  notice that for $a,c\geq 0$ the quantity  $\sup_{b>c}ab-\frac{b^2}2$ is equal to $\frac{a^2}2$ if $a\geq c$ and to $ac-\frac{c^2}2$ if $a\leq c$. It follows that for any $x,y\in\X$ and $\eps>0$ we have
\[
\sup_{\tau\in(0,\eps)}\Big(\frac{f(x,y)^-}{\sfd(x,y)}\frac{\sfd(x,y)}{\tau}-\frac12\frac{\sfd^2(x,y)}{\tau^2}\Big)=\left\{\begin{array}{ll}
\frac12\frac{(f(x,y)^-)^2}{\sfd^2(x,y)},&\qquad\text{ if }\frac{f(x,y)^-}{\sfd(x,y)}\geq\frac{\sfd(x,y)}\eps\\
&\\
\frac{f(x,y)^-}{\eps}-\frac{\sfd^2(x,y)}{2\eps^2},&\qquad\text{ if }\frac{f(x,y)^-}{\sfd(x,y)}\leq\frac{{\sfd(x,y)}}\eps\\
\end{array}\right.
\]
Taking $\lims_{y\to x}$  and noticing that by $\tilt(f)(x)>0$  we are eventually in the first case we obtain
\[
\begin{split}
\frac12\tilt(f)^2(x)&=\lims_{y\to x}\sup_{\tau\in(0,\eps)}\Big(\frac{f(x,y)^-}{\tau} -\frac{\sfd^2(x,y)}{2\tau^2}\Big)\leq \sup_{\tau\in(0,\eps)}\frac{\sup_{y\neq  x}\big(f(x,y)^- -\frac12\frac{\sfd^2(x,y)}{\tau}\big)}\tau\\
&= \sup_{\tau\in(0,\eps)}-\frac{\inf_{y\in\X}\big(f(x,y) +\frac12\frac{\sfd^2(x,y)}{\tau}\big)}\tau= \sup_{\tau\in(0,\eps)}-\frac{f_t(x)}{t}
\end{split}
\]
and the arbitrariness of $\eps>0$ gives $\frac12\tilt(f)^2(x)\leq \lims_{t\downarrow0}-\frac{f_t(x)}{t}$.
Hence the inequalities in  \eqref{eq:at} must all be equalities: then the first in \eqref{eq:dualtilt} follows, while the second is a consequence of the equality in the starred inequality in \eqref{eq:at} and of \eqref{eq:mondpm}. The last claim also follows from these arguments.
\end{proof}

\section{Laplacian bounds on $\RCD$ spaces}

\subsection{The infinite dimensional case}

We shall assume the reader familiar with calculus on $\RCD$ spaces as developed in \cite{AmbrosioGigliSavare11-2}, \cite{Gigli12}, \cite{Gigli14}. Here we collect some useful properties about Laplacian bounds, in particular concerning how they are affected by operations/limiting procedures and the relation between distributional bounds in the sense of   \cite{Gigli12} and pointwise upper bounds in terms of short time asymptotic of the heat flow considered in  \cite{MS21} (see also \cite[Definition 4.33]{Gigli12} for measure-valued upper bounds given in the same spirit). Symmetric statements for lower bounds are also trivially valid and will not be discussed.

Throughout this section, $(\X,\sfd,\mm)$ is an $\RCD(K,\infty)$ space, $K\in\R$, with $\supp(\mm)=\X$.

\bigskip

We shall denote by $\Lip_\b(\X)$ and $\Lip_\bs(\X)$ the spaces of real valued functions on $\X$ that are Lipschitz and bounded and, respectively,  Lipschitz, bounded and with bounded support.

Start recalling (\cite{AmbrosioGigliSavare11-2}, \cite{AmbrosioGigliMondinoRajala12}) that on $\RCD(K,\infty)$ spaces there is a well defined notion of heat flow $(\h_t)$ that is a strongly continuous semigroup of contractions in $L^p$ for every $p\in[1,\infty)$ and is weakly$^*$-continuous in $L^\infty$. The heat flow admits an heat kernel $\h_t\delta_x=\rho_t[x]\mm$ with symmetric transition probabilities, i.e\ $\rho_t[x](y)=\rho_t[y](x)$ for $\mm\times\mm$-a.e.\ $x,y$. The representation formula
\begin{equation}
\label{eq:reprht}
\h_tf(x)=\int f\,\d\h_t\delta_x
\end{equation}
holds for every $f\in L^p$, $p\in[1,\infty]$, and provides a canonical precise representative for the function $\h_tf$. In what follows, we shall always use this representative and notice that formula \eqref{eq:reprht} can be used to define $\h_t f$ for any $f\in L^1+L^\infty(\X)$ (this is the space of functions $f$ that can be written as $f_1+f_2$ for some $f_1\in L^1(\X)$ and $f_2\in L^\infty(\X)$ - notice that it is a Banach space when equipped with the norm $\|f\|:=\inf \|f_1\|_{L^1}+\|f_2\|_{L^\infty}$, the infimum being taken among all such writings). A consequence of \eqref{eq:reprht} is the {\bf weak maximum principle} for the heat flow, i.e.\ $\h_tf\leq C$ $\mm$-a.e.\ provided $f\leq C$ $\mm$-a.e.\ and similarly for lower bounds. Formula \eqref{eq:reprht} and the symmetry of the transition probabilities immediately imply
\begin{equation}
\label{eq:heattrans}
\int f\h_tg\,\d\mm=\int g\h_tf\,\d\mm\qquad\forall f\in L^p(\X),\ g\in L^q(\X),\ p,q\in[1,\infty],\ \tfrac1p+\tfrac1q=1.
\end{equation}
Properties of the heat flow that we shall frequently use are  the {\bf Bakry-\'Emery}  estimate \cite{Savare13}
\begin{equation}
\label{eq:BE}
|\d\h_tf|\leq e^{-Kt}\h_t(|\d f|)\qquad\mm-a.e.\ \forall f\in \Lip_b(\X),\ t>0,
\end{equation}
and the {\bf $L^\infty-\Lip$ regularization}:
\begin{equation}
\label{eq:lintylip}
\Lip(\h_tf)\leq C(K,t)\|f\|_{L^\infty},\qquad\forall f\in L^\infty(\X),\ t>0.
\end{equation}
These estimates are typically written for $f\in W^{1,2}(\X)$ and $f\in L^2\cap L^\infty(\X)$ respectively (see \cite{AmbrosioGigliSavare11-2}, \cite{AmbrosioGigliMondinoRajala12}), but the extension we wrote follows rather trivially via multiplication with a sequence of uniformly Lipschitz  cut-off functions with bounded support taking into account, for \eqref{eq:BE},  the lower semicontinuity of weak upper gradients \cite{Cheeger00}, \cite{AmbrosioGigliSavare11} (or, which is more or less the same,  the closure of the differential \cite{Gigli14}).  Notice that \eqref{eq:BE} also provides the metric information (see \cite{AmbrosioGigliSavare11-2})
\begin{equation}
\label{eq:BElip}
\Lip(\h_tf)\leq e^{-Kt}\Lip(f)\qquad\forall t\geq 0.
\end{equation}

We recall that $D(\Delta)\subset W^{1,2}(\X)$ is the space of functions $f$ for which there is $g\in L^2$ such that $\int gh\,\d\mm=-\int \d f\cdot\d h\,\d\mm$ for every $h\in W^{1,2}(\X)$. The density of $W^{1,2}$ in $L^2$ ensures that the function $g$ is unique: we will denote it  $\Delta f$.  Notice that this definition coincides with the one given in Section \ref{se:introRCD}.

We shall also work with a  regularization procedure that works better than the bare heat flow: fix once and for all $\kappa\in C^1_c(0,1)$ non-negative, with $\int_0^1\kappa=1$ and define
\begin{equation}
\label{eq:tildehn}
\tilde\h_nf:=n\int_0^{+\infty}\kappa(nt)\h_tf\,\d t=\int_0^{1}\kappa(t)\h_{t/n}f\,\d t,\qquad \forall n\in\N, \ n>0,\ f\in L^1+L^\infty(\X).
\end{equation}
The closure of $\Delta$ justifies the computation
\begin{equation}
\label{eq:deltatildehn}
\Delta\tilde\h_n f=n\int_0^{+\infty}\kappa( nt)\Delta\h_tf\,\d t=n\int_0^{+\infty}\kappa( nt)\partial_t\h_tf\,\d t=-n^2\int_0^{+\infty}\kappa'( nt)\h_tf\,\d t
\end{equation}
for any $f\in L^2(\X)$, $n\in\N$, $n>0$, where the integrals are intended in the Bochner sense. This bound, together with  the weak maximum principle and the $L^\infty-\Lip$ regularization imply
\begin{equation}
\label{eq:apprf}
\osc(\tilde \h_n f)+\Lip(\tilde \h_n f)+\|\Delta(\tilde \h_n f)\|_{L^\infty}\leq C(n,K)\osc(f)\qquad\forall n\in\N
\end{equation}
for every $f\in L^2(\X)$, where the oscillation $\osc(f)$ of $f:X\to\R$ is defined as
\begin{equation}
\label{eq:defosc}
\osc(f):=\sup f-\inf f.
\end{equation}
Another useful property that we shall occasionally use is (here and below convergence in $L^1_{loc}(\X)$ means convergence in $L^1(\X,\mm\restr B)$ for every bounded Borel set $B\subset\X$):
\begin{equation}
\label{eq:l1loc}
f\in L^\infty(\X)\qquad\Rightarrow\qquad\sup_n\|\tilde\h_nf\|_{L^\infty}<\infty\quad\text{ and }\quad  \tilde\h_nf\to f\quad\text{in }L^1_{loc}(\X).
\end{equation}
Indeed, the uniform $L^\infty$-bound is obvious, and for $L^1_{loc}$ convergence we  notice that
\begin{equation}
\label{eq:after}
\int_B|\tilde\h_nf-f|\,\d\mm\leq \|\tilde\h_n(\nchi_Bf)-\nchi_Bf\|_{L^1}+\int_B|\tilde\h_n(\nchi_{B^c}f)|\,\d\mm\quad\text{$\forall B\subset\X$ Borel bounded}.
\end{equation}
Now observe that \eqref{eq:reprht} yields $|\tilde\h_n(\nchi_{B^c}f)|\leq \tilde\h_n(|\nchi_{B^c}f|)$, thus taking also \eqref{eq:heattrans} into account we see that $\int_B|\tilde\h_n(\nchi_{B^c}f)|\,\d\mm\leq \int \tilde\h_n(\nchi_B)|\nchi_{B^c}f|\,\d\mm\leq \|f\|_{L^\infty} \int \tilde\h_n(\nchi_B)\nchi_{B^c}\,\d\mm$ and the claim follows by the strong continuity of $(\h_t) $ in $L^1(\X)$.

We want extend the domain of the Laplacian and propose two definitions of Laplacian bounds. To this aim we first introduce a suitable class of `test functions' where to `throw' derivatives. Recall that in \cite{Savare13}, \cite{Gigli14} it has been defined the space $\testi$ as
\[
\testi:=\big\{f\in \Lip_\b(\X)\cap D(\Delta)\ :\ \Delta f\in L^\infty\cap W^{1,2}(\X)\big\}.
\]
In what follows we shall mainly work with bounded functions, thus it is better to deal with test functions that are in $L^1$ and to this aim we define the space $\tvar(\X)\subset\testi$ as
\[
\begin{split}
\tvar(\X):=\big\{f\in L^1\cap\testi \ :\  |\d f|,\Delta f\in L^1\big\}.
\end{split}
\] 
Notice that since  $\testi$ is an algebra (see \cite{Savare13}) the same can easily be proved for $\tvar(\X)$.

Observe that the classical identity $\h_t\Delta f=\Delta\h_tf$ valid  for  $f\in D(\Delta)\supset\tvar(\X)$ and the continuity of $\h_t$ as a map from $L^p$ into itself for any $t\geq 0$ and $p\in[1,\infty]$ give
\begin{equation}
\label{eq:stabtvar}
\varphi\in\tvar(\X)\qquad\Rightarrow\qquad \h_t\varphi, \tilde\h_n\varphi\in \tvar(\X)\qquad\forall t,n>0,
\end{equation}
whereas from  identity \eqref{eq:deltatildehn}, the bound \eqref{eq:BE} and arguing as for  \eqref{eq:apprf} we easily obtain that
\begin{equation}
\label{eq:liptvar}
\varphi\in \Lip_\bs(\X)\qquad\Rightarrow\qquad \tilde\h_n\varphi\in\tvar(\X)\qquad\forall n\in\N,\ n>0.
\end{equation}
In particular, letting $\tvar^+(\X)\subset\tvar(\X)$ be the cone of non-negative functions, we see that
\begin{equation}
\label{eq:denstvar}
\varphi \in L^1\cap L^\infty(\X),\ \varphi\geq 0\qquad\Rightarrow\qquad \exists(\varphi_n)\subset\tvar^+(\X) \text{ with }\left\{
\begin{array}{l}
\varphi_n\to\varphi\qquad\text{ in }L^1(\X),\\
\sup_n\|\varphi_n\|_{L^\infty}<\infty.
\end{array}
\right.
\end{equation}
We are now ready to propose the following:
\begin{definition}[Laplacian and Laplacian bounds] Let $(\X,\sfd,\mm)$ be an $\RCD(K,\infty)$ space.

The space $D(\Delta_{loc})$ is the collection of functions  $f\in L^1+L^\infty(\X)$ for which there is $g\in L^1+ L^\infty(\X)$ such that
\begin{equation}
\label{eq:defdeltaloc}
\int  f\Delta \varphi\,\d\mm=\int g\,\varphi\,\d\mm\qquad\forall\varphi\in \tvar(\X).
\end{equation}
In this case the function $g$ (that is clearly unique by \eqref{eq:denstvar}) will be denoted $\Delta f$.

For $f,g\in L^1+L^\infty(\X)$ we   say that the Laplacian of $f$ is bounded above by $g$ if  
\begin{equation}
\label{eq:deflub}
\int f\,\Delta\varphi\,\d\mm\leq \int g\, \varphi\,\d\mm,\qquad\forall \varphi\in \tvar^+(\X).
\end{equation}
In this case we write $\bd f\leq g\mm$. Finally, for $f:\X\to\R$ bounded and Borel we define 
\begin{equation}
\label{eq:deftd}
\tilde\Delta f(x):=\lims_{t\downarrow0}\frac{\h_tf(x)-f(x)}{t}\qquad\forall x\in\X.
\end{equation}
\end{definition} 
Let us collect some properties of these notions. We start remarking that in writing $\bd f\leq g\mm$ we are not really defining who $\bd f$ is. In many relevant circumstances, $\bd f$ can be interpreted as a suitable polar measure (see \cite[Proposition 2.16, Chapter 2]{Peressini67}, \cite[III, 2.1]{MaRockner92}, \cite{Gigli12}, \cite{Savare13}) but for the discussion we are doing here this is not relevant.

A  direct consequence of the definition is the following stability result:
\begin{equation}
\label{eq:stablub}
\left.
\begin{array}{l}
f,f_n,g,g_n,F\in L^1+L^\infty(\X),\\
\bd f_n\leq g_n\mm,\quad\forall n\in\N,\\
f_n\to f\quad\mm-a.e.,\\
g_n\to g\quad\mm-a.e.,\\
|f_n|,|g_n|\leq F\quad\mm-a.e.
\end{array}
\right\}\qquad\qquad\qquad\Rightarrow\qquad\qquad\qquad \bd f\leq g\mm
\end{equation}
proved by passing to the limit in  \eqref{eq:deflub}: this is possible by dominate convergence thanks to the dominations for $f_n,g_n$ and the assumption   $\varphi,\Delta\varphi\in L^1\cap L^\infty(\X)$. Also,  from  \eqref{eq:stabtvar} we get
\begin{equation}
\label{eq:stabub2}
\bd f\leq g\mm\qquad\Rightarrow\qquad \bd\tilde\h_nf\leq \tilde \h_n g\mm
\end{equation}

There is a quite natural link between upper bounds on the Laplacian in the `distributional' sense  $\bd f\leq C\mm$ and in the pointwise sense $\tilde\Delta f\leq C$. 

For technical reasons that will be clear later on, part of the result is stated for  \emph{almost continuous functions}, i.e.\ those functions $f:\X\to\R$ such that
\begin{subequations}
\label{eq:defacm}
\begin{align}
\label{eq:Acont}
&\text{there is $A\subset \X$  Borel with $\mm(\X\setminus A)=0$ so that   $f\restr A$ is continuous and bounded,}\\
\label{eq:acmc}
&f(x)=\limi_{y\to x}f(y)\qquad\text{ holds for every }x\in\X.
\end{align}
\end{subequations}
The collection of such functions (that are clearly Borel) will be denoted  $\acm(\X)$. 
\begin{lemma}\label{le:lapcomb}
Let $(\X,\sfd,\mm)$ be  $\RCD(K,\infty)$  and assume that $\bd f\leq g\mm$ with $g\in L^\infty(\X)$. Then:
\begin{itemize}
\item[i)] For every $t\geq 0$ the bound $ \h_tf-f\leq \int_0^t\h_sg\,\d s$ holds $\mm$-a.e..
\item[ii)] It holds  $\bd f\leq\tilde\Delta f\mm$.
\item[iii)] If  $f\in\acm(\X)$ and $g$ is upper semicontinuous, then $\tilde\Delta f\leq g$ everywhere on $\X$.
\end{itemize}
\end{lemma}
\begin{proof}  For $(i)$ let $\varphi\in\tvar^+(\X)$ and recall  \eqref{eq:stabtvar} to get
\[
\int (\h_tf-f)\varphi=\int f(\h_t\varphi-\varphi)=\iint_0^tf\Delta\h_s\varphi\leq \iint_0^tg\h_s\varphi=\int\varphi\Big(\int_0^t\h_sg\,\d s\Big)\,\d\mm,
\]
so the claim   follows from the density result \eqref{eq:denstvar}.  For $(ii)$  we pick $\varphi\in \tvar^+(\X)$ and notice that
\[
\int f\Delta\varphi\,\d\mm=\lim_{t\downarrow0}\int f\frac{\h_t\varphi-\varphi}t\,\d\mm=\lim_{t\downarrow0}\int \frac{\h_tf-f}t\varphi\,\d\mm,
\]
thus  $(i)$, the weak maximum principle, the fact that $\varphi\in L^1(\X)$ and  (reverse) Fatou's lemma give 
\[
\int f\Delta\varphi\,\d\mm\leq \int \lims_{t\downarrow0} \frac{\h_tf-f}t\varphi\,\d\mm=\int \tilde\Delta f\,\varphi\,\d\mm,
\]
as desired. For $(iii)$, let $A\subset \X$ be as in \eqref{eq:Acont}. Then $(i)$ and the continuity of $\h_tf$ (recall \eqref{eq:lintylip}) and $ \int_0^t\h_sg\,\d s$ (that follows  from the representation formula \eqref{eq:reprht} and dominated convergence) ensure that $\h_t f(x)\leq f(x)+(\int_0^t\h_sg\,\d s)(x)$ for every $x\in A$. Then  \eqref{eq:acmc}, again the continuity of $\h_tf$ and $\lims_t\fint_0^t\h_s g\,\d s\leq g$ (from \eqref{eq:reprht} and upper semicontinuity) yield the  claim.
\end{proof}
We pass to study some properties of $D(\Delta_{loc})$ and notice that since $\h_t\Delta\varphi=\Delta\h_t\varphi=\partial_t\h_t\varphi$ holds for any $\varphi\in D(\Delta)\supset\tvar(\X)$, it is easy to see that formula \eqref{eq:deltatildehn} holds even for functions $f\in L^1+L^\infty(\X)$, meaning that   for any $\varphi\in\tvar(\X)$  we have $\int \tilde\h_n f\Delta\varphi\,\d\mm=-n^2\iint_0^{+\infty}\kappa'( nt)\h_tf\varphi\,\d t\,\d\mm.$ In particular, arguing as for   \eqref{eq:apprf} we see that
\begin{equation}
\label{eq:extappr}
\text{the estimate \eqref{eq:apprf} is valid even for $f\in L^\infty(\X).$}
\end{equation}
In defining $D(\Delta_{loc})$ we used $\tvar(\X)$ as class of test functions. If we have additional regularity on $f$, an equivalent approach (more convenient in deriving calculus rules) is possible:
\begin{equation}
\label{eq:equivdloc}
\begin{split}
&\text{for $f\in\Lip_\b(\X)$ and $g\in L^1+L^\infty(\X)$ we have:}\\
 &f\in D(\Delta_\loc)\text{ with }\Delta f=g \qquad\Leftrightarrow\qquad  \int\nabla f\cdot\nabla\psi\,\d\mm=-\int g\psi\,\d\mm\qquad\forall\psi\in \Lip_\bs(\X).
\end{split}
\end{equation}
The key observation to prove such result is that for $f\in\Lip_\b(\X)$ and $\varphi\in\tvar(\X)$ we have
\begin{equation}
\label{eq:intp}
\int f\Delta\varphi\,\d\mm=-\int\nabla f\cdot\nabla\varphi\,\d\mm.
\end{equation}
This can be established recalling that the integration by parts formula \cite[Chapter 4]{Gigli12} ensures that for any $\eta\in\Lip_\bs(\X)$ we have $\int \eta f\Delta\varphi\,\d\mm=\int-\eta \nabla f\cdot\nabla\varphi- f\nabla \eta\cdot\nabla\varphi\,\d\mm$, thus picking $\eta=\eta_n:=(1-\sfd(\cdot,B_n(\bar x)))^+$ for some $\bar x\in\X$ and letting $n\to\infty$ we conclude that \eqref{eq:intp} holds (here the assumption $|\d \varphi|\in L^1(\X)$ matters). The the implication $\Leftarrow$ in \eqref{eq:equivdloc} follows by a similar cut-off procedure, while for $\Rightarrow$, thanks to \eqref{eq:liptvar} it is enough to prove that $|\d\tilde\h_n\psi-\d\psi|\to 0$ in $L^1(\X)$. This will follow if we show that
\[
\begin{split}
\lims_{n\to\infty}\int\eta |\d\tilde\h_n\psi-\d\psi|\,\d\mm&=0\qquad\forall \eta\in\Lip_\bs(\X),\qquad\text{and}\qquad
\lim_{j\to\infty}\lims_{n\in\N}\int_{B_j(\bar x)^c}|\d\tilde\h_n\psi|\,\d\mm=0.
\end{split}
\]
The first of these is a consequence of $\int\eta |\d\tilde\h_n\psi-\d\psi|\,\d\mm\leq C\sqrt{\int |\d\tilde\h_n\psi-\d\psi|^2\,\d\mm}$ and the continuity of $[0,\infty)\ni t\mapsto\h_t\psi\in W^{1,2}(\X)$. For the second we recall \eqref{eq:BE} and  argue as after \eqref{eq:after}

A consequence of \eqref{eq:equivdloc} are the following calculus rules:
\begin{equation}
\label{eq:calclap}
\begin{split}
f,g\in \Lip_\b(\X)\cap D(\Delta_{loc})\qquad&\Rightarrow\qquad \Delta(fg)=f\Delta g+g\Delta f+2\nabla f\cdot\nabla g,\\
f\in \Lip_\b(\X)\cap D(\Delta_{loc})\qquad&\Rightarrow\qquad \Delta(u\circ f)=u'\circ f\Delta f+u''\circ f|\d f|^2,
\end{split}
\end{equation}
where in the second formula $u$ is a smooth function defined on a neighborhood of the image of $f$ and  it is part of the claim that  $fg,u\circ f\in D(\Delta_{loc})$. To check the first, by \eqref{eq:equivdloc} we need to prove that for any $\psi\in\Lip_\bs(\X)$ we have
\[
\begin{split}
-\int\nabla(fg)\cdot\nabla\psi\,\d\mm=\int\big(f\Delta g+g\Delta f+2\nabla f\cdot\nabla g\big)\psi\,\d\mm.
\end{split}
\]
In turn, this follows from the Leibniz rule for the gradient \cite[Equation (4.16)]{Gigli12}, the fact that $f\psi,g\psi\in\Lip_\bs(\X)$ and again \eqref{eq:equivdloc}. The formula for $\Delta(u\circ f)$ is proved  analogously.

An interesting consequence of the calculus rules \eqref{eq:calclap} is the following result about stability of upper bounds for the Laplacian under the `inf' operation:
\begin{lemma}\label{le:inflap}
Let $(\X,\sfd,\mm)$ be $\RCD(K,\infty)$, $K\in\R$, $I$ is a set of indexes, not necessarily countable, and $g\in L^1+L^\infty(\X)$. Assume that $\sup_{i\in I}\|f_i\|_{L^\infty}<\infty$, that $\bd f_i\leq g\mm$ for every $i\in I$  and let $\bar f:=\essinf_{i\in I} f_i$. Then $ \bd\bar f\leq g\mm$.

Assume also that for some $E\subset\X$ Borel with $\mm(\X\setminus E)=0$ we have that $f_i\restr E$ is upper semicontinuous for every $i\in I$.  Then the infimum $f:=\inf_i f_i$ coincides $\mm$-a.e.\ with  $\bar f$.
\end{lemma}
\begin{proof}
Since the essential infimum can be realized as infimum of a suitably chosen countable subfamily, recalling the stability result \eqref{eq:stablub} it is sufficient to prove the claim for $I=\{1,2\}$. Thus let $f_1,f_2\in L^\infty(\X)$ be with $\bd f_1,\bd f_2\leq C\mm$ and  recall  \eqref{eq:l1loc} to see that $\tilde\h_nf_1\wedge \tilde\h_nf_2\to  f_1\wedge f_2$ and $\tilde\h_ng\to g$ in $L^1_{loc}(\X)$. Thus from \eqref{eq:stabub2} and \eqref{eq:stablub} (possibly after passing to a suitable $\mm$-a.e.\ converging sequences) we see that we can replace $f_i$ with $\tilde\h_nf_i$, $i=1,2$, in our claim. 

By \eqref{eq:extappr} we thus reduced to proving that for $f_1,f_2\in\Lip_\b(\X)\cap D(\Delta_{loc})$ with $\Delta f_1,\Delta f_2\leq g$ $\mm$-a.e.\ we have $\bd(f_1\wedge f_2)\leq g\mm$. Since the functions $h_n:=-\frac1n\log(e^{-nf_{1}}+e^{-nf_{2}})$ are uniformly bounded and pointwise converge to $f_1\wedge f_2$, by \eqref{eq:stablub} again it is sufficient to show that $\Delta h_n\leq g$ $\mm$-a.e.\ for every $n\in\N$. This follows  by direct computation,  as  \eqref{eq:calclap} gives
\[
\begin{split}
\Delta h_n&=-n\underbrace{\Big(\frac{e^{-nf_1}|\d f_1|^2+e^{-nf_2}|\d f_2|^2}{e^{-nf_1}+e^{-nf_2}}-\frac{|e^{-nf_1}\d f_1+e^{-nf_2}\d f_2|^2}{|e^{-nf_1}+e^{-nf_2}|^2}\Big)}_{\geq 0\quad\mm-a.e.}+\frac{e^{-nf_1}\Delta f_1+e^{-nf_2}\Delta f_2}{e^{-nf_1}+e^{-nf_2}},
\end{split}
\]
so that $\Delta f_1,\Delta f_2\leq g$ implies $\Delta h_n\leq g$ as well, as desired.

For the second claim notice that the inequality $f\leq \bar f$ is trivial. For the other start defining, for $\eps>0$, the function $f_\eps(x):=\eps+\sup_{B_\eps(x)}f$ and notice that since $f\restr E$ is upper semicontinuous (as infimum of upper semicontinuous functions) we have that $f_\eps\downarrow f$ as $\eps\downarrow0$ on $E$.

Now fix $\eps>0$ and for every $x\in E$ use the upper semicontinuity of the $f_i$'s to find $i_x\in I$ and $r_x\in(0,\eps)$ such that $f_{i_x}(y)< f(x)+\eps$ for every $y\in B_{r_x}(x)\cap E$. Thus we have $f_{i_x}\leq f_\eps$ on $ B_{r_x}(x)$. By the Lindelof property of $E$ (that is a separable, though not complete in general, metric space when equipped with the restricted distance -  see \cite[Chapter I, Thm. 15]{Kelley75}) there is a countable collection $(x_n)\subset E$ such that $\cup_nB_{r_{x_n}}(x_n)\supset E$. Since $\bar f\leq f_{i_{x_n}}\leq f_\eps$ $\mm$-a.e.\ on $B_{r_{x_n}}(x_n)$ for every $n\in\N$, we deduce that $\bar f\leq f_\eps$ $\mm$-a.e.. The conclusion follows from the arbitrariness of $\eps>0$. 
\end{proof}
\subsection{The finite dimensional case} Here we  study Laplacian bounds for functions defined on some open subset $U\subset\X$ of $\X$. In this case we certainly cannot use  $\tvar(\X)$ as space of test functions, we therefore introduce the space $\tvar(U)\subset \tvar(\X)$ made of those test functions $\varphi$ with $\supp(\varphi)\subset\subset  U$ (here and below  we write $A\subset\subset B$ to say that the closure of $A$ is compact and contained in $B$). Notice that in general $\RCD(K,\infty)$ spaces it is unclear whether $\tvar(U)$ contains anything beside the 0 function. On the other hand, an approximation argument based on the mollified heat flow $\tilde\h_n$, the Bakry-\'Emery estimate \eqref{eq:BE} and post-composition shows that 
\begin{equation}
\label{eq:tcut}
\begin{split}
&\text{suppose that $\X$ is proper and let  $K\subset U\subset\X$ be with $K$ compact and $U$ open,}\\
&\text{then there is $\varphi\in \tvar(\X)$ with values in $[0,1]$, support in $U$ identically 1 on $K$,}
\end{split}
\end{equation}
see e.g.\ the arguments in \cite[Lemma 6.7]{AmbrosioMondinoSavare13-2}.
With this said, we can give the following definition:

\begin{definition}[Local Laplacian and Laplacian bounds]\label{def:lapbU} Let $(\X,\sfd,\mm)$ be a proper $\RCD(K,\infty)$ space and $U\subset\X$ open.

The space $D(\Delta_{loc},U)$ is the collection of functions  $f\in L^1+L^\infty(U)$ for which there is $g\in L^1+ L^\infty(U)$ such that
\begin{equation}
\label{eq:defdeltalocU}
\int  f\Delta \varphi\,\d\mm=\int g\,\varphi\,\d\mm\qquad\forall\varphi\in \tvar(U).
\end{equation}
In this case the function $g$ (that is clearly unique by \eqref{eq:denstvar}, \eqref{eq:tcut}) will be denoted $\Delta f$.

For $f,g\in L^1+L^\infty(U)$ we   say that the Laplacian of $f$ is bounded above by $g$ on $U$, and write $\bd f\restr U\leq g\mm$ or $\bd f\leq g\mm$ on $U$,  provided
\begin{equation}
\label{eq:deflubU}
\int f\,\Delta\varphi\,\d\mm\leq \int \varphi\,g\,\d\mm,\qquad\forall \varphi\in \tvar^+(U).
\end{equation}
Finally, for $f:U\to\R$ bounded and Borel we define $\tilde\Delta f:U\to\R\cup\{\pm\infty\}$ as
\begin{equation}
\label{eq:deftdU}
\tilde\Delta f(x):=\lims_{t\downarrow0}\frac{\h_tf(x)-f(x)}{t},
\end{equation}
where in defining $\h_tf$ we are extending $f$ outside $U$  by setting it to 0.
\end{definition} 
Let us collect some comments. We start noticing that in defining $\tilde\Delta f(x)$ for $f$ defined only on $U$, we chose to extend $f$ by 0 outside $U$, but in fact the constant value chosen to extend it is not really relevant. This is a consequence of the limiting property $\lims_{t\downarrow0}t^{-1}\h_t\delta_x(B_r^c(x))=0$ valid for any $x\in\X$ and $r>0$. Such limiting property is in turn a direct consequence of the Large Deviations  upper bound for the heat kernel 
\begin{equation}
\label{eq:LDP}
\lims_{t\downarrow0}t\log\big(\h_t\delta_x(B_r^c(x))\big)\leq-\frac{r^2}{4},
\end{equation}
established in the recent \cite{GTT22}.  Also, a direct consequence of the definitions is that
\begin{equation}
\label{eq:stablubloc}
\text{the stability property \eqref{eq:stablub} holds even for local Laplacian bounds}
\end{equation}
and that Laplacian bounds are local, as expected, in the sense that
\begin{equation}
\label{eq:localappr}
\left.
\begin{array}{l}
U_i\subset\X,\quad\text{open},\\
\bd f\restr{U_i}\leq g\mm,\quad\forall i\in I,
\end{array}
\right\}\qquad\qquad\Rightarrow\qquad\qquad \bd f\restr{U}\leq g\mm\quad \text{where}\quad U:=\cup_iU_i.
\end{equation}
To see this latter property pick $\varphi\in\tvar^+(U)$ and use \eqref{eq:tcut} to find a partition of  unit $(\eta_i)$ of $\supp(\varphi)$ subordinate to the cover made by (a finite subcover of) the $U_i$'s and made of functions in $\tvar^+(\cup_iU_i)$. Since clearly $\Delta \varphi=\sum_i\Delta(\eta_i\varphi)$ and $\eta_i\varphi\in \tvar^+(U_i)$, the claim follows.

Another trivial observation is that if $f\in W^{1,2}_{loc}(U)$ (see e.g.\ \cite[Section 2.5]{Bjorn-Bjorn11} for the definition -  alternatively, see  \cite[Section 5.1]{GT20} for a presentation oriented towards the results presented here), then the integration by parts
\[
\int f\Delta\varphi\,\d\mm=-\int\nabla f\cdot\nabla\varphi\,\d\mm\qquad\forall\varphi\in\tvar(U)
\]
is justified by the very definition of $\Delta$ and the fact that $\varphi$ has compact support. Hence arguing as for \eqref{eq:equivdloc} it is easy to see that 
\begin{equation}
\label{eq:considerazioni}
\begin{split}
\bd f\restr U\leq g\mm\ \Leftrightarrow\  -\int\nabla  f\cdot\nabla \psi\,\d\mm&\leq \int g\,\psi\,\d\mm\quad\forall\psi\in \Lip_\bs(\X)\text{ s.t.  }\psi\geq 0,\,\supp(\psi)\subset U,
\end{split}
\end{equation}
and similarly for \eqref{eq:defdeltalocU}. Indeed the fact that functions in $\tvar(U)$ are in $\Lip_\bs(\X)$ and with support in $U$ provides one implication. For the other we mollify a given $\psi\in\Lip_\bs(\X)$ with support in $U$ by considering $\eta\tilde\h_n\psi$, where $\eta$ is as in \eqref{eq:tcut} with $K:=\supp(\psi)$.

In what follows we will need to extend a function defined on some open set with a local  bound on the Laplacian to a function defined on the whole space having a global bound on the Laplacian. This will be done via the following lemma. Here the assumption $f\in W^{1,2}_{loc}(U)$ is often redundant (by a Cacciopoli-type inequality) but for simplicity we keep it.
\begin{lemma}[Extension lemma]\label{le:ext}
Let $(\X,\sfd,\mm)$ be a proper $\RCD(K,\infty)$ space, $K\in\R$,  $U\subset\X$  open, $f\in W^{1,2}_{loc}(U)$ be with $\bd f\restr U\leq g\mm$ for some $g\in L^1(\X)$, $g\geq 0$, and so that  for some $c\in\R$ the set $V:=\{f\leq c\}$ is relatively compact in $U$.

Then $\bd(f\wedge c)\leq g\mm$ as well,  where the function $f\wedge c$ is intended to be equal to $c$ on $\X\setminus U$. 
\end{lemma}
\begin{proof} Replacing $f$ with $f-c$ we can assume that $c=0$. Then let $\eta$ be as in \eqref{eq:tcut} with $K\subset\subset U$ being a compact neighbourhood of $\bar V$ and let $\tilde f:=\eta f$. Then starting from  \eqref{eq:considerazioni} and arguing as for \eqref{eq:calclap} taking into account that  $\supp(\eta)\subset\subset U$  it is easy to prove that $ \bd(\eta f)\leq g'\mm$ on $U$ for $g':=g\eta+f\Delta \eta+2\nabla f\cdot\nabla\eta$.  Then letting $\eta f$ be 0 outside $U$ and using the locality property \eqref{eq:localappr} with $U_1:=U$ and $U_2=\X\setminus K$ we see that  $ \bd(\eta f)\leq g'\mm$  holds on $\X$ and thus  from Lemma \ref{le:inflap} that $\bd(f^+)\leq (g')^+\mm$. To replace $(g')^+$ with $g^+=g$  we use again \eqref{eq:localappr}, this time with   $U_1:=\{\text{interior of }\{\eta=1\}\}$ and $U_2:=\X\setminus \bar V$.
\end{proof}
The equivalence \eqref{eq:considerazioni} shows that the notion of Laplacian bound given in Definition \ref{def:lapbU} is compatible with the analogue given in \cite{Gigli12}.  In particular,  we get (see \cite[Theorem 5.14]{Gigli12}):
\begin{theorem}[Laplacian comparison for the squared distance]\label{thm:lapbd}
Let $(\X,\sfd,\mm)$ be an $\RCD(K,N)$ space with $K\in \R$ and $N<\infty$. Let $R>0$  and $\varphi:\X\to\R$  of the form $\varphi(x)=\inf_{y\in\X}\psi(y)+\frac{\sfd^2(x,y)}{2}$ for some $\psi:\X\to\bar \R$, so that  for any $x\in\X$ there is a minimizer $y$ for $\varphi(x)$ in  $B_R(x)$. Then
\[
\bd\varphi\leq C(K^-R^2,N)\mm.
\]
In particular $\bd\tfrac12\sfd^2(\cdot,\bar x)\leq C(K^-R^2,N)\mm$ on $B_R(\bar x)$ for every $ \bar x\in\X$.
\end{theorem}
These bounds and the extension Lemma \ref{le:ext} allow to deduce appropriate `local' versions of `global' results previously obtained. An example is the following  local variant of Lemma \ref{le:inflap}:
\begin{lemma}\label{le:inflap2}
Let $(\X,\sfd,\mm)$ be $\RCD(K,N)$, $K\in\R$, $N<\infty$, $U\subset\X$ open, $I$ is a set of indexes, not necessarily countable, and $g\in L^1+L^\infty(\X)$. Assume that $\sup_{i\in I}\|f_i\|_{L^\infty(U)}<\infty$, that $f_i\in W^{1,2}_{loc}(U)$ with  $\bd f_i\restr U\leq g\mm$ for every $i\in I$ and let $\bar f:=\essinf_{i\in I} f_i$. Then $ \bd\bar f\leq g\mm$.

Assume also that for some $E\subset U$ Borel with $\mm(U\setminus E)=0$ we have that $f_i\restr U$ is upper  semicontinuous for every $i\in I$.  Then the infimum $f:=\inf_i f_i$ coincides $\mm$-a.e.\ with the $\mm$-essential infimum $\bar f$.
\end{lemma}
\begin{proof}
The second statement follows as in Lemma \ref{le:inflap}. For the first we notice that by the locality property \eqref{eq:localappr}  we can assume that $U$ is bounded. Then fix $V\subset\subset U$ open and notice that for $C>0$ sufficiently big we have: for the functions $\tilde f_i:=f_i+C\sfd^2(\cdot,V)$ there is $c>0$ independent on $i$ such that $V\subset \{\tilde f_i\leq c\}\subset\subset U$. Then Theorem  \ref{thm:lapbd}, Lemma  \ref{le:ext}  and  the locality of the Laplacian ensure that $\bd (\tilde f_i\wedge c)\leq (\nchi_Vg+(\nchi_{V^c}(g+C'))\mm$ for some $C'>0$.  Hence from Lemma  \ref{le:inflap} we deduce  that $\bd(\essinf (\tilde f_i\wedge c))\leq (\nchi_Vg+(\nchi_{V^c}(g+C'))\mm$ as well and since $\tilde f_i\wedge c=f_i$ on $V$ and $V\subset\subset U$  was arbitrary, the conclusion follows.
\end{proof}
$\RCD(K,N)$ spaces, $N<\infty$, also admit \emph{Gaussian estimates for the heat kernel} 
\begin{equation}
\label{eq:gauss}
\rho_t[x](y)\leq \frac{C(K,N)}{\mm(B_{\sqrt t}(y))}\exp\Big({-\frac{\sfd^2(x,y)}{5t}+C(K,N)t}\Big),\\
\end{equation}
for every $x,y\in\X$ where $\rho_t[x]:=\frac{\d\h_t\delta_x}{\d\mm}$ (see \cite{Sturm96II} - these follow from the fact that these are doubling spaces supporting a Poincar\'e inequality \cite{Sturm06II}, \cite{Rajala12}) and \emph{growth estimates for the volume of balls} 
\begin{equation}
\label{eq:expgr}
\mm(B_R(x))\leq C(K,N)\mm(B_1(x))e^{C(K,N)R}\qquad\forall x\in\X,\ R>1
\end{equation}
(from Bishop-Gromov-Sturm inequality, see  \cite{Sturm06II}). Combining these it is possible to prove that:
\begin{equation}
\label{eq:lebheat}
\text{for $g:\X\to\R$  integrable we have $\h_tg(x)\to g(x)$ for any Lebesgue point $x$ of $g$,}
\end{equation}
see \cite[Lemma 2.40]{MS21}. This  and  Lemma \ref{le:ext}  give the following variant of Lemma \ref{le:lapcomb}:
\begin{lemma}\label{le:lapcomb2}
Let $(\X,\sfd,\mm)$ be an $\RCD(K,N)$ space with $K\in\R$, $N<\infty$, $U\subset\X$ open bounded and $f:\X\to\R$ be Borel, bounded, in $W^{1,2}(U)$ and so that $\bd f\restr U\leq  g\mm$ for some Borel integrable function $g:U\to\R$. Assume also that $A\subset\X$ is Borel with $\mm(U\setminus A)=0$ so that $f\restr A$ is continuous. 

Then for any  $x\in A$  Lebesgue point of $g$ we have
\begin{equation}
\label{eq:lebtd}
\tilde\Delta f(x)\leq g(x).
\end{equation}
Conversely, if $g$ is bounded we have $\bd f\leq \tilde\Delta f\mm$ on $U$. Finally, if $f_1,f_2\in \acm(\X)\cap  W^{1,2}(\X)$ are so that $\bd f_i\restr U\leq g_i\mm$ for some $g_i\in L^\infty(U)$, then  $\bd(f_1\wedge f_2)\restr U\leq (\nchi_{\{f_1\leq f_2\}}g_1+\nchi_{\{f_2<f_1\}}g_2)\mm$.
\end{lemma}
\begin{proof} Let $x\in A$, $B\subset\subset U$ be an open neighbourhood of $x$ and notice that for $C>0$ big enough we have: for the function $ f':=f+C\sfd^2(\cdot,B)$ there is $c>0$ such that $B\subset\{f'\leq c\}\subset\subset U$. Then as in the proof of Lemma \ref{le:inflap2} above we have that  $\bd\tilde f\leq \tilde g\mm$, where $\tilde f:=f'\wedge c$ and $\tilde g:=\nchi_Bg+\nchi_{B^c}(g+C')$ for some $C'$. Up to subtract a constant to $f$ we can assume that $c=0$, so that $\tilde f$ has bounded support and thus is in $L^2(\X)$. Then $\h_tf\in D(\Delta)$ for every $t>0$ and, quite obviously, $\Delta\h_t\tilde f\leq\h_t\tilde g$ $\mm$-a.e.\ for every $t>0$, hence for $\eps\in(0,t)$ we have
\begin{equation}
\label{eq:teps}
\h_t\tilde f-\h_\eps\tilde  f=\int_\eps^t\Delta\h_s\tilde f\,\d s\leq \int_\eps^t\h_s\tilde g\,\d s
\end{equation}
holds $\mm$-a.e.. Then notice that by the Gaussian and volume growth estimates \eqref{eq:gauss}, \eqref{eq:expgr}, the $L^\infty-\Lip$-regularization \eqref{eq:lintylip} and the fact that $\tilde g$ is constant outside a bounded set it follows that $\h_s\tilde g$ is continuous uniformly on $s\in[\eps,t]$. Since so are $\h_t\tilde f,\h_\eps\tilde f$, we see that \eqref{eq:teps} holds everywhere on $\X$. Now  observe that the construction ensures that $\tilde f\restr A$ is continuous and the fact that $\mm(\X\setminus A)=0$, the absolute continuity $\h_t\delta_y\ll\mm$ and the fact that $\h_t\delta_y\weakto\delta_y$ easily imply that $\h_\eps\tilde f(y)\to\tilde f(y)$ for every $y\in A$. Thus we can pass to the limit in \eqref{eq:teps} evaluated at $x\in A$ as $\eps\downarrow0$, noticing that the right hand side is continuous in $\eps$, to deduce that
\[
\frac{\h_t\tilde f(x)-\tilde f(x)}{t}\leq \fint_0^t\h_s\tilde g(x)\,\d s.
\]
We then recall  \eqref{eq:lebheat} and notice   that $\eqref{eq:LDP}$ and the fact that $f,\tilde f$ coincide on $B$ and are both bounded grant that $\tilde \Delta f(x)=\tilde \Delta\tilde f(x)$: the inequality \eqref{eq:lebtd} follows.

The second claim follows from the extension argument just used and item $(ii)$ in Lemma \ref{le:lapcomb} while for  the last claim we notice that by the arguments just used and the locality of the Laplacian we can reduce to the case $\bd f_i\leq g_i$ on $\X$ for $g_i\in L^\infty(\X)$. Then by Lemma \ref{le:lapcomb}   it is sufficient to show that $\tilde\Delta(f_1\wedge f_2)\leq g$ $\mm$-a.e.. Thus let $x\in\X$ be so that $f_1(x)\leq f_2(x)$ and notice that the monotonicity of the heat flow (or equivalently \eqref{eq:reprht}) tells that $\h_s(f_1\wedge f_2)(x)\leq \h_sf_1(x)$. Subtracting the equality that holds for $s=0$ and dividing by $s$ we get $\tilde \Delta (f_1\wedge f_2)(x)\leq \tilde\Delta f_1(x)$. Recalling \eqref{eq:lebtd} and with an analogous argument for points $x$ such that $f_2(x)<f_1(x)$ we conclude.
\end{proof}
\section{A variational principle on $\RCD(K,\infty)$ spaces}

\subsection{Reminders: Hopf-Lax formula and Regular Lagrangian Flows}
We start recalling some  facts about the Hopf-Lax formula, valid on general length spaces. For $f:\X\to\R$ lower semicontinuous and bounded from below  and $t>0$  we put
\[
\begin{split}
Q_tf(x)&:=\inf_{y\in \X}f(y)+\frac{\sfd^2(x,y)}{2t},\qquad\forall x\in\X
\end{split}
\]
We shall also put $Q_0f=f$. Notice that trivially 
\begin{equation}
\label{eq:boundovvi2}
\inf f\leq Q_t f\leq \sup f\qquad\text{ and thus
}\qquad\osc(Q_tf)\leq\osc (f),
\end{equation}
where the oscillation is defined in \eqref{eq:defosc}. We recall that 
\begin{equation}
\label{eq:contqt}
(t,x)\mapsto Q_tf(x)\text{ is continuous on $(0,\infty)\times\X$ and lower semicontinuous on $[0,\infty)\times\X$,}
\end{equation}
see \cite[Lemma 3.1.2]{AmbrosioGigliSavare08} for the continuity. For lower semicontinuity in 0 let $t_n\downarrow0$, $x_n\to x$ and $(y_n)\subset \X$ so that $\limi_nQ_{t_n}f(x_n)=\limi_n f(y_n)+\frac{\sfd^2(x_n,y_n)}{2t_n}<\infty$. Since  $f$ is bounded from below we get $\limi_n\sfd(x_n,y_n)=0$. Hence $\limi_n\sfd(x,y_n)=0$ and the lower semicontinuity of $f$ gives $\limi_n f(y_n)+\frac{\sfd^2(x_n,y_n)}{2t_n}\geq \limi_n f(y_n)\geq f(x)$.

Now
using the fact that $x$ can be chosen as competitor for $Q_tf(x)$ we easily get that
\begin{equation}
\label{eq:stimadqt2}
\text{$(y_n)$  minimizing for $f(\cdot)+\frac{\sfd^2(\cdot,x)}{2t}$}\qquad\Rightarrow\qquad \lims_n\frac{\sfd^2(x,y_n)}{2t}\leq \lims_nf(x)-f(y_n)\leq \osc(f).
\end{equation}
We also claim that
\begin{equation}
\label{eq:lipQt}
\Lip(Q_tf)\leq \sqrt{2t\osc(f)}\qquad\forall t>0.
\end{equation}

To see why, let $\bar x\in\X$, $x\in B_\eps(\bar x)$ and use \eqref{eq:stimadqt2} to get  that any minimizing sequence for $Q_tf(x)$ must eventually belong to $B_{r+\eps}(x)\subset B_{r+2\eps}(\bar x)$ for $r:=\sqrt{2t\osc(f)}$. It follows that
\[
Q_tf=\inf_{y\in B_{r+2\eps}(\bar x)} \frac{\sfd^2(\cdot ,y)}{2t}+f(y)\qquad \text{on }B_\eps(\bar x) 
\]
and recalling the trivial bound  $\Lip(\sfd^2(\cdot,y)\restr{B_r(y)})\leq 2r$, we conclude that $\Lip(Q_tf\restr{B_\eps(\bar x)})\leq r+2\eps$. Then  the arbitrariness of $\eps$ in conjunction with the length property of $\X$ give the claim \eqref{eq:lipQt}.


The Hopf-Lax semigroup produces solutions of the Hamilton-Jacobi equation in the sense that:
 \begin{equation}
\label{eq:HLHJ2}
\begin{split}
&\text{for every $x\in\X$ the map $[0,\infty)\ni t\mapsto Q_tf(x)$ is  continuous and}\\
 &\text{locally absolutely continuous on $(0,\infty)$ with} \qquad\tfrac \d{\d t}Q_tf(x)+\tfrac12\lip^2(Q_tf)(x)=0\qquad  a.e.\ t,
 \end{split}
\end{equation}
see \cite[Theorem 3.5, Proposition 3.3]{AmbrosioGigliSavare11}, 
where the \emph{local Lipschitz constant} $\lip f:\X\to[0,\infty]$ of  the function $f$ is defined as
\[
\lip f(x):=\lims_{y\to x}\frac{|f(y)-f(x)|}{\sfd(x,y)}\qquad\forall x\in\X.
\]
In connection with Sobolev calculus we recall that
\begin{equation}
\label{eq:lipfdf}
\text{for $f:\X\to\R$ locally Lipschitz we have $|\d f|\leq \lip(f)$\ $\mm$-a.e.,}
\end{equation}
see e.g.\ \cite{Cheeger00} or \cite{AmbrosioGigliSavare11} for the proof.

\bigskip 

The evolution via Hopf-Lax semigroup will be used in conjunction with the concept of Regular Lagrangian Flow. The definition below as well as Theorem \ref{thm:AT} come from \cite{Ambrosio-Trevisan14} (see also \cite{AT15}), that in turn is the generalization of the analogous concept introduced by Ambrosio \cite{Ambrosio04} in the study of the DiPerna-Lions theory \cite{DiPerna-Lions89}.

In the forthcoming discussion, by $|\dot\gamma_t|:=\lim_{h\to 0}\frac{\sfd(\gamma_{t+h},\gamma_t)}{|h|}$ we denote the metric speed of the absolutely continuous curve $\gamma:[0,1]\to\X$ (see e.g.\ \cite[Theorem 1.1.2]{AmbrosioGigliSavare08}) and by $\frac12\int_0^{1}|\dot\gamma_t|^2\,\d t$ the kinetic energy of $\gamma\in C([0,1],\X)$, that is intended to be $+\infty$ if $\gamma$ is not absolutely continuous. It is easy to see that $\gamma\mapsto \frac12\int_0^{1}|\dot\gamma_t|^2\,\d t$ is lower semicontinuous (see e.g.\ \cite[Proposition 1.2.7]{GP19}).
\begin{definition}[Regular Lagrangian Flow]
\label{def:rlf}
Let $(\X,\sfd,\mm)$ be an ${\sf RCD}(K,\infty)$ space, $K\in\R$ and $(v_t)\in L^1([0,1], L^\infty(T\X))$.  Then $\Fl:[0,1]\times \X\to \X $  is said to be a \emph{Regular Lagrangian Flow} for $(v_t)$ provided it is Borel and the following properties are verified:
\begin{itemize}
\item[i)] For some  $C>0$  we have $(\Fl_t)_*\mm\leq C  \mm$ for any $t\in[0,1]$.
\item[ii)] For every $x\in\X$ the curve $[0,1]\ni t \mapsto \Fl_t(x)\in\X$
is continuous and starts from $x$.
\item[iii)] Given $f\in W^{1,2}(\X)$ and $t\in[0,1]$, one has that for $\mm$-a.e.\ $x\in\X$
the map $[0,1]\ni t\mapsto(f\circ \Fl_t)(x)\in\R$ belongs to $W^{1,1}(0,1)$ and satisfies
\begin{equation}\label{eq:RLF}
\frac{\d}{\d t}\,(f\circ \Fl_t)(x)=\d f(v_t)\big(\Fl_t(x)\big)
\quad\text{ for }\mathcal{L}^1\text{-a.e.\ }t\in[0,1].
\end{equation}
\end{itemize}
\end{definition}
We shall typically write $(\Fl_t)$ in place of $\Fl$ for Regular Lagrangian Flows (RLF in short). We notice that if $(\Fl_t)$ is a RLF of $(v_t)$, choosing a suitable countable family of 1-Lipschitz functions in $(iii)$ above we deduce that:  For $\mm$-a.e.\ $x\in\X$ the curve $t\mapsto \Fl_t(x)$ is absolutely continuous and
\begin{equation}
\label{eq:spRLF}
\text{the metric speed of $t\mapsto \Fl_t(x)$ coincides with $|v_t|(\Fl_t(x))$ for $\mm\times\mathcal L^1$}-a.e.\ (x,t)
\end{equation}
as can be proved closely following the arguments in \cite[Section 2.3.5]{Gigli14}.

The following crucial result establishes existence and uniqueness of RLFs:
\begin{theorem}\label{thm:AT}
Let $(\X,\sfd,\mm)$ be a ${\sf RCD}(K,\infty)$ space, $K\in\R$. Let $(v_t)\in L^1([0,1],L^\infty(T\X))$ be such that  $| v_t|\in L^\infty([0,1]\times\X)$. Assume that  for any $\varphi\in\Lip_{\bs}(\X)$ we have $\varphi v_t\in  D(\div)$ for a.e.\ $t\in[0,1]$ with  $\div (\varphi v_t)\in L^\infty([0,1]\times\X)$ and  $(\varphi v_t)\in L^1([0,1],W^{1,2}_C(T\X))$.

Then:
\begin{itemize}
\item[-]\underline{\rm Existence}  There exists a  Regular Lagrangian Flow (RLF, in short) ${(\Fl_t)}$ of  $(v_t)$. 

\item[-]\underline{\rm Uniqueness} Such flow is unique in the following sense:  if $(\tilde \Fl_t)$ is another
Regular Lagrangian Flow, then for  $\mm$-a.e.\ point $x\in\X$ it holds  $\tilde {\Fl}_t(x)=\Fl_t(x)$ for every $t\in[0,1]$.
\item[-]{\underline {\rm Regularity}} The flow $(\Fl_t)$ satisfies the bound
\begin{equation}
\label{eq:daRLF}
(\Fl_t)_*\mm\leq e^{C_t}\mm\qquad\text{ with }\qquad C_t:=\int_0^t\|(\div (v_s))^-\|_{L^\infty}\,\d s.
\end{equation}
\end{itemize}
\end{theorem}
Let us notice that if $v\in L^2(T\X)$ is a vector field such that $\varphi v\in D(\div)$ for every $\varphi\in\Lip_{\bs}(\X)$ (see \cite[Definition 2.3.11]{Gigli14}), then by the locality of the divergence \cite[Equations (3.5.13), (3.5.14)]{Gigli14} the function $\div(v):\X\to\R$ is $\mm$-a.e.\ well-defined by
\begin{equation}
\label{eq:defdivloc}
\div(v)=\div(\varphi v)\qquad\mm-a.e.\text{ on the interior of } \{\varphi=1\},\qquad\forall \varphi\in \Lip_\bs(\X).
\end{equation}
This is the function appearing in \eqref{eq:daRLF}. 
We remark that in the literature  Theorem \ref{thm:AT} is presented under the assumptions  $\div ( v_t)\in L^1([0,1],L^\infty(\X))$ and  $( v_t)\in L^1([0,1],W^{1,2}_C(T\X))$ in place of the weaker  $\div (\varphi v_t)\in L^1([0,1],L^\infty(\X))$ and  $(\varphi v_t)\in L^1([0,1],W^{1,2}_C(T\X))$ for every $\varphi\in\Lip_{\bs}(\X)$ (as for what concerns $W^{1,2}_C(T\X)$, to be more precise in \cite{Ambrosio-Trevisan14} the authors speak about $(r,s)$-inequalities satisfied by the symmetric part of some sort of covariant derivative, but as is clear from \cite[Definition 3.4.1]{Gigli14} this condition is satisfied for $r=s=4$ if  $( v_t)\in L^1([0,1],W^{1,2}_C(T\X))$). We chose the version above because it better fits our purposes: it   can  be derived from the one with global integrability assumptions on $\div(v_t), \nabla v_t$ using  the locality of  both divergence and covariant derivative and the finite speed of propagation coming from  \eqref{eq:spRLF} and  assumption $(v_t)\in L^1([0,1],L^\infty(T\X))$.

We shall typically apply Theorem \ref{thm:AT} for vector fields as in the next statement:
\begin{lemma}\label{le:perRLF}
Let $(\X,\sfd,\mm)$ be an ${\sf RCD}(K,\infty)$ space and $(f_t)\subset\Lip_\bs(\X)\cap D(\Delta_{loc})$ be such that
\begin{equation}
\label{eq:assperRLF}
\sup_{t\in[0,1]}\Lip(f_t)+\|\Delta f_t\|_{L^\infty}<\infty.
\end{equation}
Then the vector fields $v_t:=\nabla f_t$ satisfy the assumptions of Theorem \ref{thm:AT} and $\div(v_t)=\Delta f_t$.
\end{lemma}
\begin{proof}  Clearly we have $\sup_t\||v_t|\|_{L^\infty}<\infty$. Then notice that the definitions \eqref{eq:defdivloc},  \cite[Definition 2.3.11]{Gigli14} and the characterization \eqref{eq:equivdloc} trivially give $\div(v_t)=\Delta f_t$. Also, from the Leibniz rule \cite[2.3.13]{Gigli14} we see that 
\begin{equation}
\label{eq:divb1}
\div(\varphi v_t)=\varphi \Delta f_t+\d\varphi\cdot\d f_t\in L^\infty([0,1]\times\X)\qquad\forall \varphi\in \Lip_\bs(\X).
\end{equation}
Now recall that the identity $\d(g_1\d g_2)=\d g_1\wedge \d g_2$ holds for $g_1,g_2\in\testi$ (see \cite[Theorem 3.5.2 (iv)]{Gigli14}) and use the closure of the exterior differential (see \cite[Theorem 3.5.2 (ii)]{Gigli14}) to conclude that $\d (\varphi \d f_t)=\d \varphi\wedge \d f_t$ and thus that   $\||\d (\varphi \d f_t)|\|_{L^2}\leq \||\d f_t|\|_{L^\infty}\||\d\varphi\|_{L^2}$.  By  \cite[Corollary 3.6.4]{Gigli14}, this \eqref{eq:divb1} and assumption \eqref{eq:assperRLF} are  sufficient to conclude that   $\sup_{t}\|\varphi \nabla f_t\|_{W^{1,2}_C}<\infty$.
\end{proof}
We now turn to a characterization of RLFs that we shall often use in what comes next:
%
\begin{proposition}\label{prop:RLF}  Let $(\X,\sfd,\mm)$ be $\RCD(K,\infty)$, $(v_t)$  a Borel family of uniformly bounded vector fields and $\Fl:\X\times[0,1]\to\X$ a Borel map satisfying $(i),(ii)$ in Definition \ref{def:rlf}. 

Then  the following are equivalent:
\begin{itemize}
\item[a)]  $\Fl$ is a RLF of $(v_t)$,
\item[b)] For any  probability measure $\mu$ on $\X$ such that $\mu\leq C\mm$ for some $C>0$ and any $(f_t)\in AC([0,1],L^2(\X))\cap L^\infty([0,1],W^{1,2}(\X))$, putting $\mu_t:=(\Fl_t)_*\mu$ we have that the map $t\mapsto\int f_t\,\d\mu_t$ is absolutely continuous with
\begin{equation}
\label{eq:leibfl}
\frac{\d}{\d s}\Big(\int f_s\,\d\mu_s\Big)\restr{s=t}=\int \big(\frac\d{\d s}f_s\restr{s=t}\big)\,\d\mu_t+\int \d f_t(v_t)\,\d\mu_t \qquad a.e.\ t.
\end{equation}
\item[c)] Same as $(b)$ above, but with the regularity assumption and derivation formula for  $t\mapsto\int f_t\,\d\mu_t$  just for the case $f_t\equiv f\in W^{1,2}(\X)$ for every $t\in[0,1]$.
\end{itemize}
\end{proposition}
\begin{proof}\ \\
\st{$(a)\Rightarrow(b)$} By \eqref{eq:daRLF}, \eqref{eq:spRLF} and the assumption $|v_t|\in L^\infty_{t,x}$ we see that  $\ppi:=(\Fl_\cdot)_*\mu$ is a test plan (see \cite{AmbrosioGigliSavare11}). Then the conclusion  follows from \cite[Lemma 5.3]{GigTam18} combined  with formula \eqref{eq:RLF}.\\
\st{$(b)\Rightarrow(c)$} Obvious.

\st{$(c)\Rightarrow(a)$} Pick $f_t\equiv f\in W^{1,2}(\X)$ and then apply \cite[Proposition 2.7]{GR17}.
\end{proof}
\subsection{The variational principle}
The proof of our variational principle will come by combining suitable existence and uniqueness statements. We start with the former:
\begin{proposition}[``Existence"]\label{prop:exist}
Let $(\X,\sfd,\mm)$ be an $\RCD(K,\infty)$ space and $f:\X\to\R$ bounded and lower semicontinuous. Assume that 
for some $C>0$ we have
\begin{subequations}
\begin{align}
\label{eq:lapass}
\bd Q_tf&\leq C\mm,\qquad\qquad\forall t\in[0,1],\\
\label{eq:ugualass}
|\d Q_tf|(x)&=\lip(Q_tf)(x)\qquad\mm\times\mathcal L^1-a.e.\ (x,t)\in\X\times[0,1].
\end{align}
\end{subequations} 

Then for every bounded probability density $\rho$ with bounded support there is $\ppi\in\pr(C([0,1],\X))$ with $(\e_0)_*\ppi=\rho\mm$ such that
\begin{subequations}
\begin{align}
\label{eq:regestpi}
(\e_t)_*\ppi&\leq e^{tC}\|\rho\|_{L^\infty}\mm,\qquad\forall t\in[0,1]\\
\label{eq:daexist}
\int  f\,\d(\e_{1})_*\ppi+ \frac12 \iint_0^{1}|\dot\gamma_t|^2\,\d t\,\d\ppi(\gamma)&\leq\int Q_1f\,\d(\e_0)_*\ppi.
\end{align}\end{subequations}
\end{proposition}
\begin{proof}
For $n\in \N$, $n>0$ and $t\in[0,1]$ we  define $f^n_t:=\tilde \h_n Q_tf$ (recall  \eqref{eq:tildehn}) and then $v^n_t:=-\nabla f^n_{1-t}$. It is  clear from \eqref{eq:boundovvi2}, \eqref{eq:extappr} and Lemma \ref{le:perRLF} that for every $n\in\N$ the vectors $(v^n_t)$ satisfy the assumptions of Theorem \ref{thm:AT} and  (using also \eqref{eq:lapass}, \eqref{eq:stabub2})  that $\div(v^n_t)\geq -C$ for any $t,n$. It follows that for every $n\in\N$ the vector fields $(v^n_t)$ admit a (unique) RLF $(\Fl^n_t)$ satisfying
\begin{equation}
\label{eq:qf}
(\Fl_t^n)_*\mm\leq e^{tC}\mm\qquad\forall t\in[0,1],\ n\in\N,\ n>0.
\end{equation}
Let  $\rho$ be as in the statement, put $\mu:=\rho\mm$ and $\ppi^n:=(\Fl^n_\cdot)_*\mu\in\pr(C([0,1],\X))$: the desired plan $\ppi$ will be found as weak limit of these. Start noticing that \eqref{eq:qf} gives
\begin{equation}
\label{eq:bdpinappr}
(\e_t)_*\ppi_n\leq   e^{tC}\|\rho\|_{L^\infty}\mm \qquad\forall t\in[0,1],\ n\in\N,\ n>0
\end{equation}
and that  for  $\delta\in(0,1)$ from \eqref{eq:lipQt} and \eqref{eq:BElip} we see that 
\begin{equation}
\label{eq:unifdelta}
\Lip(Q_{1-t}f),\Lip(f^n_{1-t}),\||v^n_t|\|_{L^\infty}\leq C(\delta)\qquad\forall t\in[0,1-\delta],\ n\in\N,\ n>0.
\end{equation}

For later use we also observe that
\begin{equation}
\label{eq:limvt}
\lim_{n\to\infty}\nchi_Bv^n_t= -\nchi_B\nabla Q_{1-t}f\qquad\text{ in }L^2(T\X),\text{  $\forall t\in[0,1)$ and $B\subset\X$ Borel and bounded.}
\end{equation}
Indeed, for $t\ni[0,1)$, and $\eta\in\Lip_\bs(\X)$ identically 1 on $B$ the functions $\eta f^n_{1-t}$ are uniformly Lipschitz (by  \eqref{eq:unifdelta}), have uniformly bounded supports and converge pointwise to $\eta Q_{1-t}f$ as $n\to\infty$. Thus the reflexivity of $L^2(T^*\X)$ and the closure of the differential ensure that $\d (\eta f^n_{1-t})\weakto \d (\eta Q_{1-t}f)$ in the weak topology of $L^2(T^*\X)$. By locality we deduce that $\nchi_Bv^n_t\weakto -\nchi_B\nabla Q_{1-t}f$ and since \eqref{eq:BE} implies $\lims_n\int_B|v^n_t|^2\,\d\mm\leq \int_B|\d Q_{1-t}f|^2\,\d\mm $ the claim \eqref{eq:limvt} follows.

Thus denoting by $\ppi^n_\delta$ the restriction  of $\ppi^n$ to $[0,1-\delta]$ (more precisely: the image of $\ppi^n$ under the restriction map that takes $\gamma\in C([0,1],\X)$ and returns $\gamma\restr{[0,1-\delta]}\in C([0,1-\delta],\X)$) we see from \eqref{eq:spRLF} and \eqref{eq:unifdelta} that the plans $\ppi^n_\delta$ are concentrated on equiLipschitz curves. In particular, from \eqref{eq:bdpinappr} and the boundedness of $\supp(\rho)$ we see that $\supp(\mu^n_t)\subset B(\delta)$ for some bounded Borel set $B(\delta)\subset\X$ and every $n\in\N$, $n>0$, $t\in[0,1-\delta]$. This is enough to prove that $(\ppi^n_\delta)_{n}$ is tight, indeed from the tightness of $\mm\restr{B(\delta)}$  there is a function  $\varphi:\X\to[0,\infty]$ with compact sublevels such that $\int_{B(\delta)}\varphi\,\d\mm<\infty$, then the function
\[
\Phi(\gamma):=\int_0^{1-\delta}\big(\varphi(\gamma_t)+|\dot\gamma_t|^2\big)\,\d t
\]
also has compact sublevels in $C([0,1-\delta],\X)$ (see \cite[Lemma 5.8]{GMS15}) and \eqref{eq:unifdelta},\eqref{eq:bdpinappr},\eqref{eq:spRLF} give
\[
\int\Phi\,\d\ppi^n_\delta=\int_0^{1-\delta} \int\varphi\,\d(\e_t)_*\ppi^n\,\d t+\iint_0^{1-\delta}|\dot\gamma_t|^2\,\d t\,\d\ppi^n(\gamma)\stackrel{}\leq e^{C}\|\rho\|_\infty\int_{B(\delta)}\varphi\,\d\mm+C(\delta)^2
\]
for every $n$. By Prokhorov's theorem this proves the tightness of $(\ppi^n_\delta)$ so that up to pass to a non-relabeled subsequence we can assume that it weakly converges to some $\ppi_\delta\in C([0,1-\delta],\X)$.

Now observe that   \eqref{eq:HLHJ2} and \eqref{eq:unifdelta} tell that for any $\eta\in\Lip_\bs(\X)$ the curve $t\mapsto \eta f^n_{1-t}$ is in  $AC([0,1-\delta],L^2(\X))\cap L^\infty([0,1-\delta],W^{1,2}(\X))$ and $\frac\d{\d t}(\eta f^n_{1-t})= \tfrac12\eta\,\tilde\h_n(\lip^2(Q_{1-t}f))$. Then Proposition \ref{prop:RLF} (applied in $[0,1-\delta]$ rather than $[0,1]$) gives that $[0,1-\delta]\ni t\mapsto \int\eta f^n_{1-t}\circ\e_t\,\d\ppi^n$ is  absolutely continuous and formula \eqref{eq:leibfl}  yields
\[
\begin{split}
\int \eta f^n_{\delta }\,\d\mu^n_{1-\delta}&-\int \eta f^n_{1}\,\d\mu^n_0\\
&=\int_{0}^{1-\delta}\bigg(\int \eta \Big(\tfrac12\tilde\h_n(\lip^2(Q_{1-t}f))-|\d f^n_{1-t}|^2\Big)-f^n_{1-t}\d\eta\cdot\d f^n_{1-t}\,\d\mu^n_t\bigg)\,\d t,
\end{split}
\]
where $\mu^n_t:=(\e_t)_*\ppi^n$.  Picking $\eta\equiv 1$ on $B(\delta)$ we get
\begin{equation}
\label{eq:integr}
\begin{split}
\int f^n_{1}\,\d\mu^n_0-\int f^n_{\delta}\,\d\mu^n_{1-\delta}&=\int_{0}^{1-\delta}\int |\d f^n_{1-t}|^2-\tfrac12\tilde\h_n(\lip^2(Q_{1-t}f))\,\d\mu^n_t\,\d t\\
\text{(by \eqref{eq:spRLF})}\qquad&=\tfrac12\iint_0^{1-\delta}|\dot\gamma_t|^2\,\d t\,\d\ppi^n(\gamma)+\tfrac12 \underbrace{\int_{0}^{1-\delta}\!\!\!\!\int |v^n_t|^2-\tilde\h_n(\lip^2(Q_{1-t}f))\,\d\mu^n_t\,\d t}_{=:A_n}.
\end{split}
\end{equation}
We wish to pass to the limit in this identity and notice that by  \eqref{eq:lipQt} and  \eqref{eq:l1loc}  we get that the functions $(\tilde\h_n(\lip^2(Q_{1-t}f)))_{n\in\N}$ are uniformly bounded and converge in $L^1_{loc}$  to $\lip^2(Q_{1-t}f)$ as $n\to\infty$. This, the convergence \eqref{eq:limvt}, the uniform bound \eqref{eq:unifdelta},  the assumption  \eqref{eq:ugualass} together with  the weak convergence of $\mu^{n}_t$ to $\mu_t:=(\e_t)_*\ppi_\delta$ and the fact that  these measures have uniformly bounded supports and densities (by \eqref{eq:bdpinappr}) give $A_n\to 0$ as $n\to\infty$.

Similarly, we have $\int f^{n}_{1-t}\,\d\mu^{n}_{t}\to \int Q_{1-t} f\,\d \mu_{t}$ for every $t\in[0,1]$, so that passing to the limit in the above taking into account the lower semicontinuity of the kinetic energy gives
\begin{equation}
\label{eq:ancheperunic}
\int Q_1f\,\d\mu_0-\int Q_\delta f\,\d \mu_{1-\delta}\geq\frac12\iint_0^{1-\delta}|\dot\gamma_t|^2\,\d t\,\d\ppi_\delta(\gamma).
\end{equation}
We now wish to send $\delta\downarrow 0$ and to this aim we apply the construction above to  $\delta=\delta_j\downarrow0$ and proceed by diagonalization to find plans $\ppi_{\delta_j}\in \pr(C([0,1-\delta_j],\X))$ so that $\ppi_{\delta_j}$ is the `restriction' - in the sense previously made rigorous - of $\ppi_{\delta_k}$ for $k>j$. From \eqref{eq:ancheperunic} and \eqref{eq:boundovvi2} we also deduce the uniform bound
\begin{equation}
\label{eq:percomp}
\iint_0^{1-\delta}|\dot\gamma_t|^2\,\d t\,\d\ppi_\delta(\gamma)\leq 2\int Q_1f(\gamma_0)-Q_\delta f(\gamma_{1-\delta})\,\d\ppi_\delta(\gamma)\leq 2\osc(f)
\end{equation}
and from this we infer that there is $\ppi\in C([0,1],\X)$ whose `restriction' to $[0,1-\delta_i]$ is $\ppi_{\delta_j}$ for every $j\in\N$: just `extend' each $\ppi_{\delta_j}$ to a probability measure on $C([0,1],\X)$ imposing that curves in the support are constant on $[1-\delta_j,1]$ and notice that the resulting sequence has uniformly bounded  kinetic energy and tight marginals (by \eqref{eq:bdpinappr} and the uniform bound on second moments that follows from \eqref{eq:percomp}). Then weak compactness follows from the same arguments previously used.

Finally, we recall the lower semicontinuity property in \eqref{eq:contqt}  and Fatou's lemma for weakly converging sequences of measures (see e.g.\ \cite[Lemma 2.5]{GDP17}) to see that $\limi_{j}\int Q_{\delta_j} f\,\d \mu_{1-\delta_j}\geq \int f\,\d\mu_0$. Therefore we can pass to the limit in \eqref{eq:ancheperunic} and conclude.
\end{proof}
We turn to uniqueness:
\begin{proposition}[``Uniqueness"]\label{prop:unique}
Let $(\X,\sfd,\mm)$ be an $\RCD(K,\infty)$ space, $K\in\R$, $f:\X\to\R$ be bounded and lower semicontinuous  and $\mu\in\pr(\X)$ be with $\mu\ll\mm$.. Then there is at most one plan $\ppi\in \pr(C([0,1],\X))$ such that  $(\e_0)_*\ppi=\mu$, $(\e_1)_*\ppi\ll\mm$ and
\begin{equation}
\label{eq:perunicog}
\int  f\,\d(\e_{1})_*\ppi+\frac12 \iint_0^{1}|\dot\gamma_t|^2\,\d t\,\d\ppi(\gamma)\leq \int Q_1f\,\d(\e_0)_*\ppi.
\end{equation}
Moreover, such plan (if it exists) satisfies the following:
\begin{itemize}
\item[i)] It is induced by a map, namely there is $F:\X\to C([0,1],\X)$ Borel such that $\e_0\circ F={\rm Id}_\X$ $\mm$-a.e.\ and $\ppi=F_*\mu$.
\item[ii)] It is concentrated on geodesics of length bounded above by $\sqrt{2\osc(f)}$.
\item[iii)]  For any $\gamma\in\supp(\ppi)$  we have
\[
Q_{1}f(\gamma_0)=f(\gamma_1)+\frac12\int_0^1|\dot\gamma_t|^2\,\d t =f(\gamma_1)+\frac{\sfd^2(\gamma_1,\gamma_0)}{2}.
\]
\item[iv)] $\ppi$ is an optimal geodesic plan from $(\e_0)_*\ppi$ to $(\e_1)_*\ppi$ for the cost $c(x,y):=\frac{\sfd^2(x,y)}{2}$.
\item[v)] $Q_1f$ is a Kantorovich potential for $(\e_0,\e_1)_*\ppi$.
\item[vi)] For any $t\in (0,1)$ we have
\begin{equation}
\label{eq:bellaid2}
\lip(Q_tf)=|\d Q_tf|\qquad(\e_{1-t})_*\ppi-a.e..
\end{equation}
\item[vii)] If $(\e_t)_*\ppi\leq C\mm$ for some $C>0$ and every $t\in[0,1]$, then
$\ppi$ represents the gradient of $-Q_1f$ (in the sense of  \cite[Definition 3.7]{Gigli12} with $q=2$). 
\end{itemize}
If the space $(\X,\sfd,\mm)$ is $\RCD(K,N)$ for some $N<\infty$, then the same conclusions hold without the need of assuming $(\e_1)_*\ppi\ll\mm$ and for $\mu$-a.e.\ $x$ there is only one minimizer for $f(\cdot)+\frac{\sfd^2(\cdot,x)}{2}$.
\end{proposition}
\begin{proof} For arbitrary $\gamma\in C([0,1],\X)$ the definition of $Q_1f$ gives
\[
Q_1f(\gamma_0)- f(\gamma_{1})\stackrel{(*)}\leq \frac{\sfd^2(\gamma_0,\gamma_{1})}{2}.
\]
On the other hand, the assumption \eqref{eq:perunicog} tells
\begin{equation}
\label{eq:stimalungh}
\int Q_1f\circ\e_0- f\circ\e_1\,\d\ppi\geq \frac12\iint_0^{1}|\dot\gamma_t|^2\,\d t\,\d\ppi(\gamma)\stackrel{(**)}\geq \int \frac{\sfd^2(\gamma_0,\gamma_{1})}{2}\,\d\ppi(\gamma)
\end{equation}
thus forcing $(**)$ to be an equality and  $(*)$ to be an equality for $\ppi$-a.e.\ $\gamma$. Equality in $(**)$ holds iff $\ppi$-a.e.\ $\gamma$ is a (constant speed) geodesic, and since the collection of geodesics is a closed subset of $C([0,1],\X)$, property $(ii)$ follows (the estimate on the length also comes from \eqref{eq:stimalungh}). A similar argument by continuity based on the continuity of $Q_1f$ (recall \eqref{eq:contqt}), on the lower semicontinuity of $f$ and on $\ppi$-a.e.\ equality on $(*)$ tells that equality in $(*)$ holds for every $\gamma\in\supp(\ppi)$. It is now clear that item $(iii)$ holds.

Now observe that for the cost function $c(x,y):=\frac{\sfd^2(x,y)}{2}$ we have  $Q_1f=(- f)^c$ and since the identity $g^{cc}\geq g$ holds for arbitrary functions $g$ (see e.g.\ \cite[Proposition 2.2.9 (ii)]{G11}), from the fact that $(*)$ is an equality we deduce
\begin{equation}
\label{eq:perKP}
Q_1f(\gamma_0)+(- f)^{cc}(\gamma_{1})\geq Q_1f(\gamma_0)- f(\gamma_{1})=  \frac{\sfd^2(\gamma_0,\gamma_{1})}{2}\qquad\forall\gamma\in\supp(\ppi).
\end{equation}
This is sufficient to ensure that $(\e_0,\e_1)_*\ppi$ is $c$-optimal and since we already established that $\ppi$ is concentrated on geodesics, $(iv)$ follows. Item $(v)$ is equivalent to \eqref{eq:perKP}. For $(vi)$ we notice that considering the (normalized) restriction of $\ppi$ to $\{\gamma:\frac{\d(\e_0)_*\sppi}{\d\mm}(\gamma_0)+\frac{\d(\e_1)_*\sppi}{\d\mm}(\gamma_1)+\sfd(\gamma_0,\bar x)+\sfd(\gamma_1,\bar x)\leq n\}$ for $n\gg1$ and keeping in mind the estimate in \cite[Theorem 4.2]{AmbrosioGigliMondinoRajala12} we can assume that $(\e_t)_*\ppi\leq C\mm$ for every $t\in[0,1]$. Now  keeping in mind item $(ii)$, by \cite[Proposition 3.11]{Gigli12} we need to prove that
\[
\lim_{t\downarrow0}\frac{Q_1f(\gamma_0)-Q_1f(\gamma_t)}{\sfd(\gamma_0,\gamma_t)}=\sfd(\gamma_0,\gamma_1)  =|\d Q_1f|(\gamma_0)\qquad\text{ in }L^2(\ppi)
\]
and these follow from the metric Brenier theorem \cite[Theorems 10.3, 10.4]{AmbrosioGigliSavare11}. 
For $(vi-b)$ we notice that from $(v)$ and the general theory of interpolation of Kantorovich potentials (see \cite[Theorem 2.2.10]{G11} or \cite[Theorem 7.36]{Villani09} paying attention to the different sign convention) we have that $Q_tf,Q_{1-t}(-Q_1f)$ are Kantorovich potentials for $(\e_t,\e_0)_*\ppi,(\e_t,\e_1)_*\ppi$ respectively and - recalling that these are continuous function - that 
\begin{equation}
\label{eq:perlip}
Q_tf+Q_{1-t}(-Q_1f)\geq 0\quad\text{ on $\X$ with equality at  $\gamma_t$ for any $\gamma\in\supp(\ppi)$}
\end{equation}
Then again the metric Brenier theorem \cite[Theorems 10.3]{AmbrosioGigliSavare11} and a simple scaling show that 
\begin{equation}
\label{eq:qp}
|\d Q_tf|(\gamma_{1-t})=|\partial^+Q_tf|(\gamma_{1-t})=\sfd(\gamma_0,\gamma_1)=|\partial^+Q_{1-t}(-Q_1f)|(\gamma_t)\qquad\ppi-a.e.\ \gamma,
\end{equation}
while \eqref{eq:perlip} gives $Q_tf(\gamma_{1-t})-Q_tf(x)\leq  Q_{1-t}(-Q_1f)(x)-Q_{1-t}(-Q_1f)(\gamma_{1-t})$ for any $x\in\X$ and $\gamma\in \supp(\ppi)$. Taking positive parts this implies $|\partial^-  Q_tf  |(\gamma_t)\leq |\partial^+ Q_{1-t}(-Q_1f)|(\gamma_t)$ that together with \eqref{eq:qp} and the trivial identity $\lip\, g=|\partial^+g|\vee|\partial^-g|$ gives \eqref{eq:bellaid2}.

The fact that $\ppi$ is induced by a map now follows from \cite[Corollary 1.3]{RajalaSturm12}. The uniqueness of $\ppi$ is then a consequence of standard arguments in Optimal Transport: if $\ppi'\neq\ppi$ also satisfies the assumptions, the plan $\tfrac12(\ppi+\ppi')$ would also do so, without being induced by a map.

For the last claims we use uniqueness of optimal maps as stated in \cite[Theorem 1.1]{GigliRajalaSturm13} in place of  \cite[Corollary 1.3]{RajalaSturm12} and notice that if we pick, via a Borel selection argument, for $\mu$-a.e.\ $x$ a geodesic $F(x)$ with $F_0(x)=x$ and $F_1(x)$ minimizer for $f(\cdot)+\frac{\sfd^2(\cdot,x)}{2}$, then what said above     show that $\ppi:=(F_\cdot)_*\mu$ satisfies \eqref{eq:perunicog}. Hence it   is unique and a fortiori so is  $F_1(x)$ for $\mu$-a.e.\ $x\in\X$.
\end{proof}

Our proof of the variational principle Theorem \ref{prop:varregf} will be `fuelled' by the following result, proved in the recent \cite{GTT22}:
\begin{lemma}\label{prop:GTT}
Let $(\X,\sfd,\mm)$ be an $\RCD(K,\infty)$ space, $K\in\R$, and let $f\in\testi$. Then
\begin{subequations}\nonumber
\begin{align}
\bd Q_t f  &\leq \big(\|(\Delta f)^+\|_\infty +tK^-\Lip(f)^2\big)\mm\qquad\forall t\geq 0\\
\lip\, Q_tf(x)&=|\d Q_t f|(x),\qquad\qquad\qquad\qquad\qquad\mm\times\mathcal L^1-a.e.\ (x,t)\in\X\times[0,\infty).
\end{align}
\end{subequations}
\end{lemma}
In the case $K<0$ we will want to improve the above Laplacian upper bound to avoid dependance on the  Lipschitz constant of $f$. To do so we shall use Lemma \ref{le:lapcomb} in conjunction with the following result, that as said in the introduction basically comes from  \cite{MS21}:
\begin{lemma}\label{le:lappunt}
Let  $(\X,\sfd,\mm)$ be a $\RCD(K,\infty)$ space, $K\in\R$, and $f:\X\to\R$ be Borel and bounded. Assume that for some $ x, y\in\X$ and $t>0$ we have $Q_tf( x)=f( y)+\frac{\sfd^2( x, y)}{2t}$. Then
\[
\tilde\Delta Q_tf( x)\leq \tilde\Delta f( y)-K\frac{\sfd^2( x, y)}{t}
\]
\end{lemma}
\begin{proof}From \cite[Lemma 3.4]{AmbrosioGigliSavare12} and its proof we know that for any $s>0$ it holds
\[
\h_sQ_tf( x)\leq \h_sf( y)+e^{-2Ks}\frac{\sfd^2( x, y)}{2t}.
\]
Subtracting the identity $Q_tf( x)=f( y)+\frac{\sfd^2( x, y)}{2t}$, dividing by $s$ and letting $s\downarrow0$ we conclude.
\end{proof}
We are now ready to state and prove our variational principle. Recall that the space $\acm(\X)$ was defined in \eqref{eq:defacm}.
\begin{theorem}[A variational principle on $\RCD(K,\infty)$ spaces]\label{prop:varregf}
Let $(\X,\sfd,\mm)$ be an $\RCD(K,\infty)$ space and let $f\in\acm(\X)$  be such that $\bd f\leq C\mm$ for some $C>0$.

Then for every $t> 0$ we have 
\begin{subequations}
\label{eq:macchina}\begin{align}
\label{eq:stimalapqtf}
\bd Q_tf&\leq  (C+2K^-\osc(f))\mm\\
\label{eq:slopedqtf}
|\d Q_tf|&=\lip (Q_tf)\qquad\mm-a.e..
\end{align}
\end{subequations}
Moreover, for every $T>0$ the following holds.  For given $x\in\X$ consider the  problem of minimizing
\begin{equation}
\label{eq:prob}
\begin{split}
\gamma\ \mapsto\   f(\gamma_T)+\frac12\int_0^T|\dot\gamma_t|^2\,\d t
\end{split}
\end{equation}
among continuous curves $\gamma$ on $[0,T]$ starting from $x$. Then:
\begin{itemize}
\item[-]\underline{\rm Existence} For $\mm$-a.e.\ $x\in\X$   a minimizer $F(x)$ exists. Also, the point $F_T(x)$ minimizes
\begin{equation}
\label{eq:minHL}
\X\ni y\qquad\mapsto\qquad f(y)+\frac{\sfd^2(x,y)}{2T}.
\end{equation}
\item[-]\underline{\rm Uniqueness} If $\X$ is $\RCD(K,N)$ for some $N<\infty$, then for $\mm$-a.e.\ $x\in\X$ the minimizer $F(x)$  is unique and so is the minimizer of \eqref{eq:minHL}.

In the general case,  $F:\X\to C([0,T],\X)$ is unique up to $\mm$-a.e.\ equality in the class of maps $\tilde F:\X\to C([0,T],\X)$ such that  $\tilde F_0=\Id$ $\mm$-a.e., $(\tilde F_T)_*\mm\ll\mm$  and so that for $\mm$-a.e.\ $x\in\X$ the curve $\tilde F_\cdot(x)$ is a minimizer for the above problem.

\item[-]{\underline {\rm Regularity}} It holds
\begin{equation}
\label{eq:stimabella}
(F_t)_*\mm\leq  e^{ t(C+2K^-\osc (f))}\mm,\qquad\forall t\in[0,T].
\end{equation}
\item[-]{\underline {\rm $F$ as Regular Lagrangian Flow}} $F$  is the only RLF for the vectors $v_t:=-\nabla Q_{T-t}f$.
\item[-]{\underline{\rm Minimal value}} For $\mm$-a.e.\ $x\in\X$ we have
\begin{equation}
\label{eq:flussoQt}
Q_{T}f(x)=f(F_T(x))+\frac1{2}\int_0^T|\dot{F_t(x)}|^2\,\d t=f(F_T(x))+\frac{\sfd^2(x,F_T(x))}{2T}.
\end{equation}
\item[-]{\underline{\rm Range in geodesics}} For $\mm$-a.e.\ $x\in\X$ the curve $[0,T]\ni t\mapsto F_t(x)$ is a constant speed geodesic of length bounded above by $\sqrt{2T\osc(f)}$.
\end{itemize}
\end{theorem}
\begin{proof} We shall prove \eqref{eq:macchina} for $t=1$ only and the statements about minimizers for $T=1$ only. The general cases follow by the same arguments, or by replacing the distance $\sfd$ with $\sqrt t\sfd$, noticing that  both Laplacian and Ricci  bounds are affected  by a factor $\frac1t$ and taking also into account the natural rescaling properties for RLFs.\\
\st{Step 1: the core argument} We claim that all the stated conclusions follow if $f\in\acm(\X)$ is   such that $\bd f\leq C\mm$ and
\begin{subequations}
\label{eq:core}
\begin{align}
\label{eq:corelap}
\bd Q_tf&\leq C'\mm,\qquad\forall t\in[0,1]\qquad\text{ for some }C'>0\\
\label{eq:coresl}
|\d Q_tf|(x)&=\lip(Q_tf)(x),\qquad\mm\times\mathcal L^1-a.e.\ (x,t)\in\X\times[0,1],
\end{align}
\end{subequations}
(in particular, even if a priori we only have some `bad' upper bound on the Laplacian, we can improve it to the `correct' estimate \eqref{eq:stimalapqtf}).

Let us prove \eqref{eq:slopedqtf}. Put $A_t:=\{|\d Q_tf|=\lip(Q_tf)\}$ (this is a Borel set defined up to $\mm$-negligible sets), let $\bar x\in\X$ and  $0<R_1<R_2$. Put $\rho:=\mm(B_{R_1})^{-1}\nchi_{B_{R_1}}$ and apply  Proposition \ref{prop:exist} with $C'$ in place of $C$ to find a plan $\ppi $ satisfying  \eqref{eq:regestpi} and \eqref{eq:daexist}.  
 We can thus apply Proposition \ref{prop:unique} and deduce, by item $(ii)$ of such statement, that $\supp((\e_t)_*\ppi)\subset B_{R_2}(x)$ for $t\ll1$. Thus letting $\rho_t$ be the density of $(\e_t)_*\ppi$, we have 
\[
\mm(A_{1-t}\cap B_{R_2})\stackrel{\eqref{eq:bellaid2}}\geq \mm(\{\rho_t>0\})\geq \|\rho_t\|_\infty^{-1}\stackrel{\eqref{eq:regestpi}}\geq \mm(B_{R_1})e^{-tC'}\qquad for\ t\ll1. 
\]
We apply this estimate to the scaled space $(\X,\frac{\sfd}{\sqrt{1-t}},\mm)$ for $t\ll1$ noticing that the  Laplacian and Ricci bounds are only slightly affected and that $Q^\sfd_{1-t}f=Q^{\frac{\sfd}{\sqrt{1-t}}}_{1}f$  (equivalently: in the above argument we apply Propositions \ref{prop:exist}, \ref{prop:unique} to slightly different time intervals) . We conclude that
\[
\mm(A_1\cap B_{\frac{R_2}{\sqrt{1-t}}})\geq \mm(B_{\frac{R_1}{\sqrt{1-t}}})e^{-tC''}\qquad for\ t\ll1,
\]
 thus letting first $t\downarrow0$ and then $R_1\uparrow R_2$ we deduce that $\mm(A_1\cap B_{R_2})=\mm(B_{R_2})$. By the arbitrariness of $R_2$, this proves \eqref{eq:slopedqtf} for $t=1$. The general case follows by scaling.

We turn to the rest of the proof and observe that for $\ppi=\ppi_{R_1}$ built as above the plan $\ppi':=\sum_{n}2^{-n}\ppi_{R_{1,n}}$ for some $R_{1,n}\uparrow\infty$ satisfies  $(\e_0)_*\ppi,(\e_1)_*\ppi\ll\mm$ and \eqref{eq:perunicog}, thus   Proposition \ref{prop:unique} ensures that   is induced by a map $F$. The fact that $F$ takes values on geodesics of length $\leq \sqrt{2\osc(f)}$ and satisfies \eqref{eq:flussoQt} (for $T=1$) follows from the analogue statements in Proposition \ref{prop:unique}. Now observe that   from Lemma \ref{le:lapcomb} and the assumptions $f\in\acm(\X)$, $\bd f\leq C\mm$ we have that $\tilde\Delta f(y)\leq  C$ for every $y\in\X$. Hence by \eqref{eq:flussoQt} we are in position to apply  Lemma \ref{le:lappunt} above in conjunction with the trivial bound \eqref{eq:stimadqt2} to deduce that
\[
\tilde\Delta Q_1f(x)\leq C+2K^-\osc(f)\qquad\mm-a.e.\ x\in\X.
\] 
This  bound and the distributional inequality \eqref{eq:corelap} give, via  Lemma \ref{le:lapcomb},  estimate \eqref{eq:stimalapqtf}.

Now that we have \eqref{eq:stimalapqtf} we can apply once again Proposition \ref{prop:exist}, this time  with $C+2K^-\osc(f)$ in place of $C$ to conclude, by \eqref{eq:regestpi} and the construction, that \eqref{eq:stimabella} holds.

The validity of \eqref{eq:flussoQt}, the very definition of $Q_1f$ and the fact that $\int_0^1|\dot\gamma_t|^2\,\d t\geq\sfd^2(\gamma_0,\gamma_1)$ with equality iff $\gamma$ is a constant speed geodesic, ensure that $F_\cdot(x)$ is a minimizer for \eqref{eq:prob}. Thus existence is proved. Uniqueness follows from Proposition \ref{prop:unique}: if $\tilde F:\X\to C([0,1],\X)$ is also a selection of minimizers with $(\tilde F_1)_*\mm\ll\mm$, then for $\mu\in\pr(\X)$ with $\mm\ll\mu\ll\mm$ the plan $\ppi:=\frac12((F_\cdot)_*\mu+(\tilde F_\cdot)_*\mu)$ satisfies \eqref{eq:perunicog} with $(\e_0)_*\ppi=\mu$ and $(\e_1)_*\ppi\ll\mm$. By Proposition \ref{prop:unique} this is enough to ensure that $\ppi$ is induced by a map, i.e.\ that $F=\tilde F$ $\mm$-a.e.. The last claim in Proposition  \ref{prop:unique} also gives $\mm$-a.e.\ uniqueness of minimizers of \eqref{eq:minHL} if $\X$ is $\RCD(K,N)$.

It remains to prove that $(F_t)$ is the only Regular Lagrangian Flow of $v_t:=-\nabla Q_{1-t}f$.  To see that it is a RLF, since $F$ takes values in continuous curves, it is sufficient to show that for any $\delta\in(0,1)$ it is the RLF of $(v_t)$ on $[0,1-\delta]$.  In this direction notice that \eqref{eq:lipQt} tells
\begin{equation}
\label{eq:estvtagain}
|v_t|\leq C(\delta)\qquad\mm-a.e.\qquad \forall t\in[0,1-\delta],\qquad\forall \delta\in(0,1).
\end{equation}
Then    let $\rho$ be a bounded probability density with bounded support, put $\ppi:=(F_\cdot)_*(\rho\mm)$ as above and use  \eqref{eq:estvtagain} to apply Proposition \ref{prop:RLF} on $[0,1-\delta]$ and deduce that it is sufficient to prove that for any   $\varphi\in W^{1,2}(\X)$ the map $[0,1-\delta]\ni t\mapsto \int\varphi \circ\e_t\,\d\ppi$ is absolutely continuous with derivative equal to $-\int\d\varphi(\nabla Q_{1-t}f)\circ\e_t\,\d\ppi$ for a.e.\ $t$. Absolute continuity follows from the Sobolev regularity of $\varphi$ and the fact that $\ppi$ is a test plan (see e.g.\ \cite[Theorem 2.1.21]{GP19}). Then recalling that $\ppi$ is an optimal geodesic test plan and by  a simple scaling argument we see that it is sufficient to prove the formula for the derivative for $t=0$, i.e.\ to prove that
\[
\lim_{t\downarrow0}\int\frac{\varphi(\gamma_t)-\varphi(\gamma_0)}t\,\d\ppi(\gamma)=-\int\d\varphi(\nabla Q_{1}f)\circ\e_0\,\d\ppi.
\]
This follows recalling that $\ppi$ represents the gradient of $-Q_1f$ in the sense of \cite[Definition 3.7]{Gigli12} (recall item $(vii)$ in Proposition \ref{prop:unique}) and the lemma about `horizontal and vertical derivatives', see \cite[Theorem 3.10]{Gigli12} (and \cite[Lemma 4.5]{AmbrosioGigliSavare11-2}).

We pass to uniqueness, thus let $\rho$ as before, $\tilde \Fl$ be a RLF of the $v_t$'s, put $\tilde\ppi:=(\tilde \Fl_\cdot)_*\mu$, let $f_t:=Q_{t}f$ and notice that - by \eqref{eq:HLHJ2} and the assumption \eqref{eq:slopedqtf} -  the identity $\partial_tf_t=-\frac12|\d f_t|^2$ holds $\mm$-a.e.\ for a.e.\ $t\in[0,1]$. Then from \eqref{eq:estvtagain} and \eqref{eq:spRLF} and arguing as for  \eqref{eq:integr} we obtain
\[
\begin{split}
\int Q_1f\,\d\mu-\int Q_\delta f\,\d(\e_{1-\delta})_*\tilde\ppi&=\iint_0^{1-\delta}\big(|\d f_t|^2-\tfrac12|\d f_t|^2)\circ\tilde \Fl_t\,\d t\,\d\mu
\stackrel{\eqref{eq:spRLF}}=\frac12\iint_0^{1-\delta}  |\dot\gamma_t|^2\,\d t\,\d\tilde\ppi(\gamma).
\end{split}
\]
Now notice that the continuity of  $t\mapsto\tilde \Fl_t(x)$ for $\mm$-a.e.\ $x\in \X$ implies the weak convergence of  $(\e_{1-\delta})_*\tilde\ppi$ to $(\e_1)_*\tilde\ppi$, hence the lower semicontinuity stated in \eqref{eq:contqt} and Fatou's lemma for weakly converging measures (see e.g.\ \cite[Lemma 2.5]{GDP17}) imply $\limi_{\delta\downarrow0}\int Q_\delta f\,\d(\tilde \Fl_{1-\delta})_*\mu\geq \int f\,\d(\e_1)_*\tilde\ppi$.

We can therefore pass to the limit in the above and conclude that $\tilde\ppi$ satisfies \eqref{eq:perunicog}, so that the uniqueness claim follows again by Proposition \ref{prop:unique}.

\st{Step 2: the approximation argument} Let $f\in\acm(\X)$ be   such that $\bd f\leq C\mm$ and assume that $(f_n)\subset\acm(\X)$  is a sequence that is $\Gamma$-converging to  $f$  (see \cite{DalMaso93}, \cite{Braides02}), i.e.\ so that
\begin{subequations}
\label{eq:glim}
\begin{align}
\label{eq:glimi}
&\text{for every $x_n\to x$ we have  $\limi_nf_n(x_n)\geq f(x)$,}\\
\label{eq:glims}
&\text{for every $x\in\X$ there is $x_n\to x$ such that $\lims_nf_n(x_n)\leq f(x)$.}
\end{align}
\end{subequations}
Assume also that for some $C''>0$ the uniform bounds
\begin{subequations}
\label{eq:boundbrutti}
\begin{align}
\label{eq:lapbrutto}
\bd f_n&\leq C''\mm\\
\label{eq:oscbound}
\osc(f_n)&\leq C''
\end{align}
\end{subequations}
hold for every $n\in\N$ and that the conclusions of the theorem are true for the $f_n$'s (with the estimate \eqref{eq:stimabella} depending on $C''$). We claim that in this case the conclusions of the statement hold for $f$ (with the correct estimate \eqref{eq:stimabella} depending on $C,\osc(f)$ only). According to the previous step to this aim it is sufficient to show that the properties \eqref{eq:core} hold. In proving this we shall frequently use the fact  that \eqref{eq:glim} and \eqref{eq:oscbound} give that the $f_n$'s - and thus also the $Q_tf_n$'s - are equibounded.

We start claiming that 
\begin{equation}
\label{eq:daglims}
\lims_{n}Q_tf_n(x)\leq Q_tf(x)\qquad\forall t>0.
\end{equation}
Indeed, for $x,y\in\X$ we use \eqref{eq:glims} to find $y_n\to y$ with $\lims_nf_n(y_n)\leq f(y)$ and observe that
\[
\lims_{n}Q_tf_n(x)\leq\lims_n\Big( f_n(y_n)+\frac{\sfd^2(x,y_n)}{2t}\Big)\leq f(y)+\frac{\sfd^2(x,y)}{2t},
\]
so that  the arbitrariness of $y\in\X$ gives \eqref{eq:daglims}. 

To get (a sort of) the $\Gamma-\limi$ inequality analogue to \eqref{eq:daglims} we need to overcome the  lack  of compactness by using the uniform estimates \eqref{eq:boundbrutti}.

We assumed that for the $f_n$'s the conclusions are true, thus let $F^n:\X\to C([0,1],\X)$ be the corresponding maps as in the statement, fix a bounded probability density $\rho$ with bounded support, define $\ppi_n:=(F^n)_*(\rho\mm)\in\pr(C([0,1],\X))$ and notice that by \eqref{eq:oscbound} and the last claim in the statement   the $\ppi_n$'s are concentrated on a bounded set of geodesics. Also, by \eqref{eq:boundbrutti} and \eqref{eq:stimabella}, they satisfy 
\begin{equation}
\label{eq:unipin}
(\e_t)_*\ppi_n\leq e^{t(C''+2K^-C'')}\mm,\qquad\forall n\in\N,\ t\in[0,1].
\end{equation}
Arguing as in the proof of Proposition \ref{prop:exist}, this is enough to conclude that the family  $(\ppi_n)$ is tight, so that passing to a non-relabeled subsequence we can assume that it converges to a limit plan $\ppi$. Now observe that integrating \eqref{eq:flussoQt} we get
\[
\int Q_1f_n\,\d(\e_0)_*\ppi_n=\int f_n\,\d(\e_1)_*\ppi_n+\frac12\iint_0^1|\dot\gamma_t|^2\,\d t\,\d\ppi_n(\gamma)\qquad\forall n\in\N.
\]
The bound \eqref{eq:daglims},  identity $(\e_0)_*\ppi_n=(\e_0)_*\ppi=\rho\mm$   and (reverse) Fatou's lemma give $$\lims_n\int Q_1f_n\,\d(\e_0)_*\ppi_n\leq \int Q_1f\,\d(\e_0)_*\ppi\,,$$ while \eqref{eq:glimi} and  Fatou's lemma for weakly converging sequences of measures (see e.g.\ \cite[Lemma 2.5]{GDP17}) give $\limi_n\int f_n\,\d(\e_1)_*\ppi_n\geq \int f\,\d(\e_1)_*\ppi$. Thus by  the lower semicontinuity of the kinetic energy we can pass to the limit in the above and obtain
\begin{equation}
\label{eq:tantiug2}
\begin{split}
\int Q_1f\,\d(\e_0)_*\ppi&\geq \limi_{n\to\infty}\int Q_1f_n\,\d(\e_0)_*\ppi_n\geq \int f\,\d(\e_1)_*\ppi+\frac12\iint_0^1|\dot\gamma_t|^2\,\d t\,\d\ppi(\gamma).
\end{split}
\end{equation}
Since  \eqref{eq:unipin}  implies   $(\e_1)_*\ppi\ll\mm$,  we can apply Proposition \ref{prop:unique} and deduce that
\[
|\d Q_tf|=\lip(Q_tf)\qquad (\e_{1-t})_*\ppi-a.e..
\]
Then the arbitrariness of $\rho$ and the very same arguments used at the beginning of the previous step (based on the fact that the estimate \eqref{eq:unipin} passes to the limit in place of \eqref{eq:regestpi}) show that \eqref{eq:slopedqtf}, and in particular {\eqref{eq:coresl}, holds.} 

Also, since the inequality $f(\gamma_1)+\frac12\int_0^1|\dot\gamma_t|^2\,\d t\geq f(\gamma_1)+\frac{\sfd^2(\gamma_0,\gamma_1)}2\geq Q_1f(\gamma_0)$ is true for any $\gamma\in C([0,1],\X)$, the bound \eqref{eq:tantiug2}  gives  $\limi_n\int Q_1f_n\,\d(\rho\mm)\geq \int Q_1f\,\d(\rho\mm)$. Now   put $g_n:=\sup_{k\geq n}Q_1f_k$, $g:=\inf_n g_n=\lims_nQ_1f_n$ and observe that,  by dominated convergence,  $g_n\to g$ in $L^1(\X,\rho\mm)$. On the other hand
\[
\lims_n\|g_n-Q_1f_n\|_{L^1(\rho\mm)}=\lims_n\int g_n-Q_1f_n\,\d(\rho\mm)\leq \int \underbrace{ g-Q_1f}_{\leq 0\text{ by \eqref{eq:daglims}}}\,\d(\rho\mm)\leq 0.
\]
This forces at once $Q_1f_n\to g$ in $L^1(\rho\mm)$ and $g=Q_1f$ $\rho\mm$-a.e., hence  the arbitrariness of $\rho$ tells that $(Q_1f_n)$ converges to $Q_1f$ in $L^1_{loc}(\X)$. Now from \eqref{eq:boundbrutti} and    the validity  of \eqref{eq:stimalapqtf} for $f_n$ we get  $\bd Q_1f_n\leq (C''+2K^-C'')\mm$ and according to the stability property \eqref{eq:stablub} we can conclude that $\bd Q_1f\leq (C''+2K^-C'')\mm$. Since by a scaling argument the same inequality is true for $Q_tf$, \eqref{eq:corelap} holds. We thus established \eqref{eq:core}, as desired.

\st{Step 3: wrapping everything up}

\noindent\underline{The case of test functions}  Let $f\in\testi$. Then   clearly $f\in\acm(\X)$. Also, the properties  \eqref{eq:core} are granted by  Lemma \ref{prop:GTT}, hence in this case the conclusion follows from the first step.

\noindent\underline{Intermediate regularity} Assume that $f$ is in $\Lip_\b(\X)\cap D(\Delta_\loc)$ with $\Delta f\in \Lip_b(\X)$. Pick $\bar x\in\X$, let $\phi_n:=(1-\sfd(\cdot,B_n(\bar x)))^+$, define $\eta_n:=\tilde\h_1\phi_n\in\tvar(\X)$ and then $f_n:=\eta_nf$. It is clear from \eqref{eq:calclap} that $f_n\in\testi$ for every $n\in\N$, so by the previous step the conclusions of the theorem are true for it. We claim that the bounds \eqref{eq:boundbrutti} and the $\Gamma-$convergence \eqref{eq:glim} are in place.

The bound \eqref{eq:oscbound} is trivial and since \eqref{eq:apprf}  gives
\[
\sup_n\osc(\eta_n)+\Lip(\eta_n)+\|\Delta\eta_n\|_\infty<\infty,
\]
the bound  \eqref{eq:lapbrutto}  follows from the Leibniz rule for the Laplacian \eqref{eq:calclap}. We pass to \eqref{eq:glim} and   notice that the monotonicity of the heat flow grants that $(\eta_n(x))$ is a bounded non-decreasing sequence for every $x\in\X$, hence it converges to a limit $\eta(x)$. In particular  for $\varphi\in\tvar(\X)$ by dominated convergence we have $\int\eta_n\varphi\,\d\mm\to\int\eta\varphi\,\d\mm$. On the other hand we have
\[
\int \varphi\eta_n\,\d\mm=\int \varphi\tilde\h_1\phi_n\,\d\mm=\int \tilde\h_1\varphi\,\phi_n\,\d\mm\qquad\to\qquad\int \tilde\h_1\varphi\,\d\mm=\int \varphi\,\d\mm
\]
and by the density property \eqref{eq:denstvar} we conclude that $\eta=1$ $\mm$-a.e., and thus everywhere because it is a Lipschitz function (being limit of equiLipschitz functions). It follows that for $x_n\to x$ we have
\[
|\eta_n(x_n)-\eta(x)|\leq \sfd(x_n,x)\sup_n\Lip(\eta_n)+|\eta_n(x)-\eta(x)|\quad\to\quad 0\qquad\text{ as $n\to\infty$}
\]
and thus, since $f\in\Lip_\b(\X)$, that $ f_n(x_n)\to f(x)$. This property is clearly stronger than \eqref{eq:glim}, which therefore is proved. Hence the desired conclusion for $f$ follows from Step 2.

\noindent\underline{The general case}  Let $f:\X\to\R$ be as in the assumptions, put $f_n:=\tilde\h_n f$ and notice that by \eqref{eq:extappr} and \eqref{eq:lintylip} we have $f_n\in \Lip_\b(\X)\cap D(\Delta_\loc)$ with $\Delta f_n\in\Lip_b(\X)$. Then by the previous set the conclusions of the theorem are true for $f_n$, whereas by  \eqref{eq:stabub2}  and the weak maximum principle we see  that the uniform bounds \eqref{eq:boundbrutti} are in place.

According to Step 2 to conclude it is therefore sufficient to show that the $\Gamma$-convergence    \eqref{eq:glim} is in place. To prove \eqref{eq:glimi} we  let $g\in\Lip_{\b}(\X)$, $g_n:=\tilde\h_n g$ and notice that \eqref{eq:BElip} implies $L:=\sup_n\Lip(g_n)<\infty$. Thus if  $g\in \Lip_\b(\X)$, for any   $x_n\to x$  the monotonicity of the heat flow gives
\begin{equation}
\label{eq:perlimi}
\limi_{n\to\infty}f_n(x_n)\geq \limi_{n\to\infty} g_n(x_n)\geq \limi_{n\to\infty} (g_n(x)-L\sfd(x,x_n))= \limi_{n\to\infty}g_n(x)=g(x),
\end{equation}
where in the last equality we used the continuity of $g$, the representation $\h_tg(x)=\int g\,\d\h_t\delta_x$ and the fact that $\h_t\delta_x\weakto\delta_x$ as $t\downarrow0$. From the lower semicontinuity of $f$ it is easy to establish that $f=\sup\{g\, :\, g\in \Lip_b(\X),  g\leq f\}$ (see e.g.\ \cite[Section 5.1]{AmbrosioGigliSavare08}),  thus \eqref{eq:glimi} follows from \eqref{eq:perlimi}.

To prove \eqref{eq:glims} we recall that for some Borel set $E\subset\X$ with $\mm$-negligible complement we have that $f\restr E$ is continuous. Also, using \eqref{eq:acmc} and a diagunalization argument it is clear that it suffices to prove \eqref{eq:glims} for $x\in E$. Fix such $x$, let $\eps>0$, find $r>0$ such that $f\leq f(x)+\eps$ on $B_r(x)\cap E$ and notice that since $\h_t\delta_x\ll\mm$ for every $t>0$ and $\h_t\delta_x(B_r^c)\to 0$ as $t\downarrow0$ we have
\[
\lims_{t\downarrow0}\h_tf(x)=\lims_{t\downarrow0}\int f\,\d\h_t\delta_x=\lims_{t\downarrow0}\int_{B_r(x)\cap E} f\,\d\h_t\delta_x\leq f(x)+\eps.
\]
Since $\eps>0$ was arbitrary, the claim \eqref{eq:glims}  follows by picking $x_n\equiv x$.
\end{proof}
\begin{remark}\label{re:infdim}{\rm
Let  $\X,f$ be as in the previous statement, $T>0$ and $U,U'\subset \X$ be open sets such that $F_T(x)\in U'$ for $\mm$-a.e.\ $x\in U$. Assume also that for some $c\in\R$, in addition to the assumption $\bd f\leq C\mm$ we also have $\bd f\restr U\leq c\mm$. Then from Theorem \ref{prop:varregf} and  Lemmas \ref{le:lappunt}, \ref{le:lapcomb} we get
\begin{equation}
\label{eq:vargen}
\bd Q_Tf\leq ( c\nchi_U+C\nchi_{\X\setminus U}-K{\sf D}(U,U'))\mm.
\end{equation}
Comparing this bound with \cite[Theorem 1.5]{MS22} we see that a similar result, stated for $p=2$ in place of $p=1$, holds without assuming  non-collapsing or finite dimensionality.

We notice also that variants of \eqref{eq:vargen} and Theorem \ref{prop:varregf} for other values of $p$ are quite easy to obtain following the above strategy, at least on $\RCD(K,N)$ spaces where existence of optimal maps for the $L^p$-cost is known (\cite{CM15}, \cite{Deng21}, see also \cite{ACMCS21}). The key ingredients are the `$p$-versions' of Lemma \ref{le:lappunt} - whose validity has already been noticed in \cite{MS21} - and of Lemma \ref{prop:GTT}  - that can be obtained following the arguments in \cite{GTT22} using the $L^p$-version of Bochner inequality in \cite{Savare13}.
}\fr\end{remark}

\section{Harmonic maps from $\RCD(K,N)$ to  $\CAT(0)$ spaces}\label{ch:main}

\subsection{Reminders: basic definitions and results} \label{se:sobmet}
From now on we shall work under the following assumptions:
\begin{itemize}
\item[-] $(\X,\sfd,\mm)$ is  an $\RCD(K,N)$ space of essential dimension $d\in\N\cap[1,N]$ (\cite{AmbrosioGigliSavare11-2}, \cite{Gigli12}, \cite{BS18}).
\item[-] $U\subset\X$ is open, bounded and such that $\mm(\X\setminus U)>0$.
\item[-] $(\Y,\sfd_\Y)$  is a  $\CAT(0)$ space (see e.g.\ \cite{Bacak14} for an introduction to the topic).
\end{itemize}
Recall that the space $L^2(U;\Y)$ is the collection of Borel maps $u$ (defined up to $\mm$-a.e.\ equality)  from $U$ to $\Y$ that are essentially separably valued (i.e.\ $\mm(U\setminus u^{-1}(\Y'))=0$ for some separable subset $\Y'\subset\Y$) and with $\int_U\sfd_\Y^2(u(x),\o)\,\d\mm(x)<\infty$ for some $\o\in\Y$, and thus any by $\mm(U)<\infty$.

We start recalling two notions of metric-valued Sobolev functions, referring to \cite{HKST15} for more details on the first and to \cite{KS93}, \cite{GT20}  for the second. 
\begin{definition}[Sobolev space via post-composition] $W^{1,2}(U;\Y)\subset L^2(U;\Y)$ is the collection of those maps $u:U\to\Y$ such that $\varphi\circ u\in W^{1,2}(U)$ for every $\varphi:\Y\to\R$ Lipschitz and so that the function $|\d u|:U\to[0,\infty]$ defined up to $\mm$-a.e.\ equality as
\begin{equation}
\label{eq:defdun}
|\d u|:=\esssup_{\varphi:\Y\to\R\atop\Lip(\varphi)\leq 1}|\d(\varphi\circ u)|
\end{equation}
is in $L^2(U)$.
\end{definition}
\begin{definition}[Sobolev space \`a la Korevaar-Schoen]
For every  $r>0$ define the local energy $\kse[u,U]:U\to\R^+$ of $u$ at scale $r$   as
\begin{equation}
\label{eq:kse}
\kse_{2,r}[u,U] (x) := \begin{cases} \ \Big\vert \fint_{B_r(x)} \frac{\sfd_{\Y}^2(u(x),u(y))}{r^2} \, \d \mm(y) \Big\vert^{1/2} &\text{ if } B_r(x) \subset U, \\ \ 0 &\text{ otherwise.} \end{cases}
\end{equation}
Then $\KS(U;\Y)\subset L^2(U,\Y)$ is the collection of those $u$'s such that
\begin{equation}
\label{eq:defKS}
\E(u):=\limi_{r\downarrow0}\int_U\kse_{2,r}[u,U] ^2\,\d\mm<\infty.
\end{equation}
\end{definition}
We now collect those properties of metric-valued Sobolev functions that are relevant for our discussion, referring to  \cite[Theorem 3.13, Proposition 3.9, Theorem 6.4]{GT20} for the proof:
\begin{proposition}\label{prop:baseKS}With the same assumptions and notations as above we have:
\begin{itemize}
\item[i)] For every $u\in \KS(U;\Y)$  there is a function $\e_2[u]\in L^2(U)$, called \emph{energy density}, such that 
\begin{equation}
\label{eq:conven}
\kse_{2,r}[u,U]\quad\to\quad \e_2[u]\qquad\text{ in $L^2(U)$ and $\mm$-a.e.\ as $r\downarrow0$.}
\end{equation}
In particular, the $\limi$ in \eqref{eq:defKS} is a limit. 
\item[ii)] We have $W^{1,2}(U;\Y)= \KS(U;\Y)$ with
\begin{equation}
\label{eq:dfu}
\e_2[u]\leq |\d u |\leq c(d)\e_2[u]\qquad\mm-a.e.\ on\ U,
\end{equation}
for some  constant $c(d)$ depending only on the essential dimension $d$, and thus only on $N$.
\item[iii)] For any $u\in \KS(U;\Y)$ there is a unique minimizer $\bar u\in \KS(U;\Y)$ of the Korevaar-Schoen energy $\E$ in the class of functions $v\in \KS(U;\Y)$  such that $x\mapsto \sfd_\Y(u(x),v(x))$ is in $W^{1,2}_0(U)$. Such minimizer is called \emph{harmonic map}.
\end{itemize}
\end{proposition}
In what follows we will need to know that the energy density $\e_2[u]$ is not only the limit of $\kse_{2,r}[u,U]$, but also of an appropriate average of $\sfd_\Y(u(x),u(\cdot))$ via heat kernel:
\begin{proposition} With the same assumptions and notations as in the beginning of the section, the following holds. Let $u\in \KS(U;\Y)$. Then there is $E\subset U$ Borel with $\mm(U\setminus E)=0$ such that 
\begin{equation}
\label{eq:convenvar}
\lim_{t\downarrow0}\frac1t\int \sfd_\Y(u(x),u(y))\,\d\h_t\delta_x(y)= 2(d+2) \e_2[u]^2(x)\qquad\forall x\in E.
\end{equation}
\end{proposition}
The \emph{proof} of this proposition follows very closely that of  \cite[Theorem 3.13]{GT20} from which the result \eqref{eq:conven} recalled above is extracted. Since the arguments for proving  \cite[Theorem 3.13]{GT20} are rather lengthy (and based on few additional concepts such as that of mGH-convergence of spaces, approximate metric differentiability, Hajlasz-Sobolev functions...) and since the modifications needed to obtain \eqref{eq:convenvar} are straightforward, we limit ourselves to pointing out what follows. The basic idea behind the proof of \eqref{eq:conven} is that for $\mm$-a.e.\ $x\in U$ there is a seminorm ${\sf md}_u(x)$ on $\R^d$ that acts as a sort of `metric differential' by describing the infinitesimal variation of $u$ around $x$ (see  \cite[Definition 3.3]{GT20}). The energy density $\e_2[u](x)$ is then defined by the formula  
\begin{equation}
\label{eq:defe2}
\begin{split}
\e_2^2[u](x):=&\fint_{B_1^{\R^d}(0)} {\sf md}_u(x)(v)^2\,\d\mathcal L^d(v)\\
=&\frac1{|B^{\R^d}_1|}\int_0^1r^{d+1}\underbrace{\int_{S^{d-1}}{\sf md}_u(x)(v)^2\,\d\mathcal H^{d-1}(v)}_{=:|||{\sf md}_u(x)|||}\,\d r=\frac{|||{\sf md}_u(x)|||}{(d+2)|B^{\R^d}_1|}
\end{split}
\end{equation}
and  \eqref{eq:conven} follows by the `first order expansion of $u$ given by ${\sf md}_u$' by noticing that the measures $\mm(B_r(x))^{-1}\mm\restr{B_r(x)}$  appearing in the definition of $\kse_{2,r}[u,U]$ converge, after scaling, to  $\frac1{|B^{\R^d}_1|}\mathcal L^d\restr{B^{\R^d}_1(0)}$\footnote{Such convergence of measures is in place, by definition, at all points with Euclidean tangent by the very definition of mGH-convergence of pointed spaces. To get the desired \eqref{eq:conven} one uses also that $\RCD$ spaces are doubling and support a Poincar\'e inequality to call into play maximal estimates that are useful to control the term $\tfrac{\sfd_\Y(u(x),u(y))}{r}$ appearing in the definition of $\kse_{2,r}[u,U]$.  Such control is necessary due to the nature of the concept of approximate metric differentiability, that only speaks about $y$ belonging to suitabl charts for which the point $x$ is assumed to be a Lebesgue point, while in the definition of $\kse_{2,r}[u,U]$ all the points $y$ in $B_r(x)$ apear in the integral. See the proof of  \cite[Theorem 3.13]{GT20}  for more details}. Similarly, \eqref{eq:convenvar} is proved recalling that the heat kernels $\h_t\delta_x$ converge, after scaling, to the Gaussian measure $\h^{\R^d}_1\delta_0={(4\pi)^{-\frac d2}}e^{-\frac{|\cdot|^2}4}\mathcal L^d$ (this follows from the stability of the heat flow under lower Ricci curvature bounds first noticed in \cite{Gigli10} - here we use \cite[Theorem 7.7]{GMS15}). Using such convergence and arguing as for \eqref{eq:conven} we see that the limit in \eqref{eq:convenvar} is equal to
\[
\int  {\sf md}_u(x)(v)^2\,\d \h^{\R^d}_1\delta_0(v)=\frac{|||{\sf md}_u(x)|||}{(4\pi)^{\frac d2}}\int_0^{+\infty} r^{d+1}e^{-\frac{r^2}4}\,\d r,
\]
thus from \eqref{eq:defe2} we conclude that \eqref{eq:convenvar} is ultimately reduced to the verification of the identity
\[
\frac{1}{(4\pi)^{\frac d2}}\int_0^{+\infty} r^{d+1}e^{-\frac{r^2}4}\,\d r=\frac2{|B^{\R^d}_1|}
\]
that in turn is well known (and easy to establish by direct computation, e.g.\ by induction). In all this, the fact that the set $E$ can be chosen to be Borel follows from the construction or, alternatively, noticing that the right hand side of \eqref{eq:convenvar} is a Borel function, that the set of points for which the limit in the left hand side exists is Borel and that so is the associated limit function.

\bigskip

We shall also need the following result, see  \cite[Theorem 4.18]{gn20}:
\begin{theorem}\label{thm:lapcomp}  With the same assumptions and notations as in the beginning of the section, the following holds.
Let $u\in \KS(U;\Y)$ be harmonic,  $f:\Y\to\R$ be  $\lambda$-convex and $\Y'\subset\Y$ be a $\CAT(0)$ subspace with $\mm(U\setminus u^{-1}(\Y'))=0$ and so that $f\restr {\Y'}$ is Lipschitz. Then  
\begin{equation}
\label{eq:fconv}
\bd(f\circ u)\restr U\geq \lambda (d+2)\e_2[u]^2\mm.
\end{equation}
\end{theorem}
\begin{remark}{\rm
There are few differences between the statement as given above and that of \cite[Theorem 4.18]{gn20}, let us comment on them. 

One concerns the fact that in  \cite[Theorem 4.18]{gn20} we assumed $f$ to be globally Lipschitz, rather than only on an essential image of $u$ as above. The modification we did is inessential in  the proof (just replace $\Y$ with $\Y'$ and notice that $u$ remains harmonic) and useful later on, when, after proving that $u$ is locally bounded, we will choose as $f$ the squared distance from a point.

Another is about the use of $\e_2[u]$ in place of $|\d u|_\HS$ that appeared in \cite[Theorem 4.18]{gn20}. There is a strict relation between these two: in \cite[Proposition 6.7]{GT20} we proved that $(d+2)\e_2[u]^2=|\d u|_{\HS}^2$ so  that \eqref{eq:fconv} is equivalent to
\begin{equation}
\label{eq:fconv2}
\bd(f\circ u)\restr U\geq \lambda |\d u|_{\HS}^2\mm.
\end{equation}
We preferred to state the bound as in \eqref{eq:fconv} because in order to give a meaning to the pointwise Hilbert-Schmidt norm   $|\d u|_\HS$ one needs to rely on the  \emph{universally infinitesimally Hilbertian} of $\CAT(\kappa)$ spaces proved in  \cite{DMGSP18}.  Since, perhaps surprisingly, such result  plays no role in this paper (nor in \cite{gn20}), we decided to stick to the notation $\e_2[u]$ to emphasize this fact.

The last  difference is made more evident by the writing in \eqref{eq:fconv2}, as in  \cite[Theorem 4.18]{gn20} the right hand side is multiplied by an additional factor $\frac1{d+2}$. This is due to a computational mistake done in \cite{gn20} that made it all the way up to publication: namely, in deriving the bound \cite[(4.12)]{gn20} from \cite[(4.11)]{gn20} we forgot to multiply one of the terms in the right hand side by $\frac1{d+2}$. This results in an incorrect formula, the correct version being
\[
\lims_{t\downarrow0}\frac{\E(u_t)-\E(u)}t\leq -\frac1{d+2}\int_\Omega \lambda g|\d u|_{\HS}^2+\la\d(f\circ u),\d g \ra\,\d\mm. 
\]
Using this bound in place of the original \cite[(4.12)]{gn20} in the proof of  \cite[Theorem 4.18]{gn20} we obtain the correct result reported above.
}\fr\end{remark}
We conclude recalling the following key regularity statement about subharmonic functions. 
\begin{theorem}\label{thm:ellrt}
Let $(\X,\sfd,\mm)$ be an $\RCD(K,N)$ space, $K\in\R$, $N<\infty$, $U\subset\X$ open, $B_R(\bar x)\subset U$, $\alpha,\beta \in\R$ with $\beta\geq 0$ and $f\in W^{1,2}_{loc}(U)$ be with 
\begin{equation}
\label{eq:subsol}
\bd f\geq -R^{-2}(\alpha  f+\beta)\mm\qquad\ on\ U.
\end{equation}
 Then for every   $\lambda \in(0,1)$ and $r\in(0, R]$  we have
\begin{equation}
\label{eq:ellst}
\|f^+\|_{L^\infty(B_{ \lambda r}(\bar x))}\leq C_1(\alpha^+,\tfrac1{1-\lambda},K^-R^2,N)\fint_{B_{ r}(\bar x)}f^+\,\d\mm+\beta\frac{r^2}{R^2}\,C_2(\tfrac1{1-\lambda},K^-R^2,N) .
\end{equation}
Moreover, $f$ admits an upper semicontinuous representative. More precisely, it  coincides $\mm$-a.e.\ with its   essential upper semicontinuous envelope $f^*$ defined as $f^*(x):=\inf_{r>0}\esssup_{B_r(x)\subset U} f$.
\end{theorem}
This result is folklore in the field and its \emph{proof} can be achieved via Moser's iteration. In this direction let us point out what follows. 
\begin{itemize}[leftmargin=*]
\item[-] The constants appearing in Moser's technique depend only on the local doubling constant and on the constant in the Poincar\'e inequality. This has been realized in the smooth category in \cite{SC92}, where it has been proved that doubling \& Poincar\'e imply a Sobolev inequality (needed in Moser's iteration) with an argument that carries over to the present setting without modifications, much like it does so in the setting of Dirichlet forms, as noticed in \cite{Sturm96III}. Notice that the settings of Dirichlet forms and $\RCD$ spaces are fully compatible by \cite[Theorem 6.10]{AmbrosioGigliSavare11-2} (whose proof requires the chain and Leibniz rules established in \cite{AmbrosioGigliSavare11-2} and   \cite{Gigli12}).
\item[-] Chain rule and integration by parts formula in general metric measure spaces have been provided in \cite{Gigli12} and can be used to follow the proof, e.g., of \cite[Theorem 8.17]{GT01} without any relevant modification, even in non-Hilbertian spaces. These tools are not-necessary if $\alpha=\beta=0$ as in this case the inequality $\bd f\geq 0$ can be read in purely variational terms and, notably, this is sufficient to develop a satisfactory theory, see e.g.\ \cite{Bjorn-Bjorn11} and references therein. In the `dual' setting of Dirichlet forms, in place of the calculus as in \cite{Gigli12} one can use that based on the `Carr\'e du champ' operator, see e.g.\ \cite{MaRockner92}, \cite{FOM11} and references therein. 
\item[-] The dependance of $C_1,C_2$ on the given parameters can be seen e.g.\ inspecting the proof of \cite[Theorem 8.17]{GT01}. The dependance on $K^-R^2$ can also be proved via a scaling argument: replacing $\sfd$ with $\frac1r\sfd$ the original $\RCD(K,N)$ space becomes $\RCD(Kr^2,N)$, hence $\RCD(-K^-R^2,N)$, inequality \eqref{eq:subsol} becomes $\bd f\geq -\frac{r^2}{R^{2}}(\alpha  f+\beta)\mm$ and everything happens in a ball of radius 1.
\item[-] From  \eqref{eq:ellst} it is clear that $f(x)=f^*(x)$ for  $x$ Lebesgue point of $f$ such that $f(x)>0$. For the general case we  notice that  $f+k$, $k\gg1$, satisfies a similar differential inequality.
\item[-] In \cite[Theorem 8.17]{GT01}  there is an $L^p$ norm, $p>1$, at the right hand side. Still, once such bound is proved, as customary one can use the $L^\infty$ control it provides and the trivial inequality $\|f\|_{L^p}\leq \|f\|^{\frac{p-1}p}_{L^\infty}\|f\|_1^{\frac 1p}$ to get \eqref{eq:ellst}.
\end{itemize}

\subsection{First regularity results}

In this section we shall work under the following assumptions:
\begin{itemize}
\item[-] $(\X,\sfd,\mm)$ is an $\RCD(K,N)$ space of essential dimension $d\in\N$.
\item[-] $U\subset\X$ is open, bounded and such that $\mm(\X\setminus U)>0$.
\item[-] $(\Y,\sfd_\Y)$  is a $\CAT(0)$ space.
\item[-] $u\in\ks(U;\Y)\subset L^2(U,\Y)$ is harmonic. For convenience, we shall also fix  Borel representatives of $u$ and $\e_2[u]$.
\end{itemize} 
We recall that the $\CAT(0)$ condition implies (see e.g.\ \cite{Bacak14}):
\begin{equation}
\label{eq:convdacat3}
\text{for every $p\in\Y$ the map $\sfd_\Y(\cdot,p)$ is convex and the map $\sfd_\Y^2(\cdot,p)$ is 2-convex.}
\end{equation}
This and Theorem \ref{thm:ellrt} give the following important result about the `continuity set' of $u$:
\begin{proposition}\label{prop:boundlinfty}
With the above notation and assumptions, the following holds.

There exists $\cont(u)\subset U$ Borel with 
\begin{equation}
\label{eq:mcont}
\mm(U\setminus \cont(u))=0
\end{equation}
such that the restriction of $u$ to $\cont(u)$ is continuous and locally bounded. More precisely, for any $\o\in\Y$, ball $B_R(x)\subset U$,  $\lambda\in(0,1)$ and $r\in(0, R]$ we have
\begin{equation}
\label{eq:boundest}
\sup_{x\in\cont(u)\cap B_{\lambda r}(x)}\sfd_\Y(u(x),\o)\leq  C(\tfrac1{1-\lambda},K^-R^2,N)\sqrt{\fint_{ B_{ r}(x)}\sfd^2_\Y(u(x),\o)\,\d\mm(x)}.
\end{equation}
\end{proposition}
\begin{proof} 
By \eqref{eq:convdacat3} and Theorem \ref{thm:lapcomp} we know that for any $p\in\Y$ we have $\bd\sfd_{\Y}( u(\cdot),p)\geq 0$ on  $U$, thus  by  the second part of Theorem \ref{thm:ellrt} we know that there is a Borel $\mm$-negligible set $N_p\subset U$ such that $\sfd_{\Y}( u(\cdot),p)$ is upper semicontinuous. Now recall that since $u$ is essentially separable valued there is $\Y'\subset\Y$  separable such that $\mm(U\setminus u^{-1}(\Y'))=0$, then let $(p_n)\subset\Y'$ be countable and dense and put 
\[
\cont(u):=u^{-1}(\Y')\setminus(\cup_n N_{p_n}).
\] 
Then $\cont(u)$ is Borel with $\mm(U\setminus \cont(u))=0$ and 
\[
\lims_{y\to x\atop y\in \cont(u)}\sfd_\Y(u(x),u(y))\leq \sfd_\Y(u(x),p_n)+\lims_{y\to x\atop y\in \cont(u)}\sfd_\Y(p_n,u(y))\leq 2\sfd_\Y(u(x),p_n)
\]
for every $x\in \cont(u)$ and $n\in\N$. Taking the $\inf$ in $n$ we obtain the desired continuity, while    estimate  \eqref{eq:boundest}  follows directly from  \eqref{eq:ellst} with $C=C_1(0,\tfrac1{1-\lambda},K^-R^2,N)$.
\end{proof}
We now define the upper semicontinuous function $\sfd_u:U\times U\to\R$ as 
\begin{equation}
\label{eq:defdu}
\sfd_u(x,y):=\lims_{(x',y')\to(x,y)\atop (x',y')\in \cont(u)\times \cont(u)}\sfd_\Y(u(x'),u(y'))
\end{equation}
and notice that trivially from the definition it satisfies
\begin{subequations}
\begin{align}
\label{eq:duu}
\sfd_u(x,y)=\sfd_\Y(u(x),u(y))\qquad\forall x,y\in \cont(u),\\
\label{eq:dut}
\sfd_u(x,y)\leq\sfd_u(x,z)+\sfd_u(z,y)\qquad\forall x,y,z\in U.
\end{align}
\end{subequations}
Another direct consequence of the definition and \eqref{eq:duu}  is that
\begin{equation}
\label{eq:liptilt}
\lip(u)(x)=\lims_{y\to x}\frac{\sfd_\Y(u(y),u(x))}{\sfd(y,x)}=\tilt(-\sfd_u)(x)\qquad\forall x\in\cont(u),
\end{equation}
where the definition of $\tilt$ was given in \eqref{eq:tilt}. We shall also need the following properties:
\begin{proposition}\label{prop:lapdu} With the above notation and assumptions, the following holds:
\begin{itemize}
\item[i)] For every $x\in U$ we have  $\sfd_u(x,\cdot)\in W^{1,2}(U)$ and $\bd\sfd_u(x,\cdot)\restr U\geq 0$. 
\item[ii)] We have
\begin{equation}
\label{eq:tilte2}
\lip (u)\leq C(N)\,\e_2[u]\qquad\mm-a.e.\ on\ U.
\end{equation}
\end{itemize}
\end{proposition}
\begin{proof}\ \\
\st{(i)} For $x\in\cont(u)$ we have $\sfd_u(x,y)=\sfd_\Y(u(x),u(y))$ for $\mm$-a.e.\ $y\in U$ by \eqref{eq:duu} and \eqref{eq:mcont}. Thus in this case the claims follow from item $(ii)$ in Proposition \ref{prop:baseKS} and Theorem  \ref{thm:lapcomp} in conjunction with  property \eqref{eq:convdacat3}. For the general case we fix $x\in U$ and for $r>0$ put $g_r(y):=\sup_{x'\in B_r(x)\cap \cont(u)}\sfd_u(x',y)$. Then  \eqref{eq:defdu} and \eqref{eq:duu} ensure that
\begin{equation}
\label{eq:perlims}
\sfd_u(x,\cdot)=\inf_{r>0}g_r(\cdot)\qquad \text{on\ $\cont(u)$ and thus $\mm$-a.e.\ on\ $U$}.
\end{equation}
Now notice that $\sfd_\Y(u(x'),\cdot)$ is 1-Lipschitz for every $x'\in U$, thus the uniform estimate \eqref{eq:dfu} - recall also \eqref{eq:defdun} -, the lattice property of Sobolev functions (that in turn follows from the chain rule and the locality of minimal weak upper gradients) and the lower semicontinuity of minimal weak upper gradients easily give that $g_r\in W^{1,2}(U)$ with $|\d g_r|\leq c(d)\e_2[u]$. Then \eqref{eq:perlims} and again the lower semicontinuity of minimal weak upper gradients tell that $\sfd_u(x,\cdot)\in W^{1,2}(U)$.

For the Laplacian estimate we apply  Lemma \ref{le:inflap2}  (with inverted signs) to deduce that $\bd g_r\restr U\geq0$, so that by \eqref{eq:perlims} and  the stability property \eqref{eq:stablubloc} we conclude.

\st{(ii)} Let $x\in\cont(u)$ and $r>0$ be so that $B:=B_r(x)$ and $2B=B_{2r}(x)$ are contained in $U$. Then
\[
\sup_{y\in B}\sfd_u(x,y)\stackrel{\eqref{eq:defdu},\eqref{eq:mcont} }=\|\sfd(u(x),u(\cdot))\|_{L^\infty(B)}\stackrel{\eqref{eq:boundest}}\leq C(K^-r^2,N)\sqrt{\fint_{2B}\sfd^2_\Y(u(y),u(x))\,\d\mm(y)},
\]
hence dividing by $r$ and letting $r\downarrow0$ we conclude recalling \eqref{eq:conven}.
\end{proof}

Another direct consequence of \eqref{eq:convdacat3} and Theorem \ref{thm:lapcomp} is:
\begin{proposition}[Reverse Poincar\'e inequality]\label{prop:invP}
With the above notation and assumptions, the following holds. Let $B:=B_r(x)\subset U$ and  $\lambda\in(0,1)$. Then
\begin{equation}
\label{eq:invP}
\int_{B_{\lambda r}(x)}\e_2[u]^2+|\d u|^2\,\d\mm\leq \frac{c(d)}{r^2(1-\lambda)^2}\inf_{\o\in\Y}\int_{ B_r(x)}\sfd_\Y^2(u(x),\o)\,\d\mm,
\end{equation}
for some constant $c( d)$ depending only on the essential dimension $d$ of $\X$ (and thus only on $N$).
\end{proposition}
\begin{proof} With an approximation argument we can assume that ${ B}\subset\subset U$. Hence  $u\restr{ B}$ is (essentially) bounded by Proposition \ref{prop:boundlinfty} and clearly harmonic. In particular $u( B)$ is essentially contained in a closed ball $\Y'$ of $\Y$. As it is well-known, and trivial consequence of \eqref{eq:convdacat3}, closed balls in $\CAT(0)$ spaces are convex, and thus  $\CAT(0)$ as well. Then for  $\o\in\Y$ let  $\sfd_\o:\Y'\to\R$ be given by $\sfd_\o(q):=\sfd_\Y(q,\o)$ and notice that by \eqref{eq:convdacat3} $\sfd_\o^2$ is 2-convex on $\Y'$, thus  Theorem \ref{thm:lapcomp} (with $ B$ in place of $U$)  gives $\e_2^2[u]\mm\leq c(d) \bd(\sfd^2_\o\circ u)$ on $ B$.  Therefore \eqref{eq:dfu} gives  $|\d u|^2\mm\leq c(d) \bd(\sfd^2_\o\circ u)$ on $ B$ and for    $\varphi:=(1-\tfrac2{r(1-\lambda)}\sfd(\cdot,B_{\lambda r}(x)))^+$ and $\alpha>0$ we have
\[
\begin{split}
\int_\X\varphi^2|\d u|^2\,\d\mm&\stackrel{\eqref{eq:considerazioni}}\leq  -c(d)\int_\X\d(\varphi^2)\cdot\d(\sfd_\o^2\circ u)\,\d\mm\\
&\stackrel{\phantom{\eqref{eq:considerazioni}}}=-c(d)\int_\X\varphi\,\sfd_\o\circ u\,\d\varphi\cdot\d(\sfd_\o\circ u)\,\d\mm\\
&\stackrel{\,\eqref{eq:defdun}\,}\leq c(d)\int_\X\alpha \varphi^2|\d u|^2+\alpha^{-1}\sfd^2_\o\circ u\,|\d\varphi|^2\,\d\mm.
\end{split}
\]
Picking $\alpha<\frac1{2c(d)}$, using  the trivial bound $|\d\varphi|\leq \frac2{r(\lambda-1)}\nchi_{ B}$ and  \eqref{eq:dfu} we get the conclusion.  
\end{proof}

We conclude the section with the following general result:
\begin{proposition}[A Rademacher-type result]\label{prop:rad}
With the above notation and assumptions, the following holds. Let $v\in W^{1,2}(U,\Y)=\KS(U,\Y)$ be locally Lipschitz. Then
\begin{equation}
\label{eq:radelip}
\lip(v)=|\d v|\qquad\mm-a.e.\ on\ U.
\end{equation}
\end{proposition}
\begin{proof}
By \cite[Prop. 2.5, Lemma 3.4]{GT20} and with the terminology therein we have  $\lip(v)(x)=|||{\sf md}_x(v)|||$ for $\mm$-a.e.\ $x\in U$, thus taking into account  \cite[Thm 4.12]{GT20} we need only to verify that
\begin{equation}
\label{eq:ovvio}
|||{\sf md}_\cdot(v)|||=\esssup{\sf md}_\cdot(v)(\mathscr I({\sf v}))\qquad\text{and}\qquad |\d v|=\esssup|\d v({\sf v})|,
\end{equation}
where in both cases the $\mm$-essential supremum is taken among ${\sf v}\in L^0(T\X)$ with $|{\sf v}|\leq 1$ $\mm$-a.e.. The first in \eqref{eq:ovvio} is obvious (see also \cite[Thm 4.9]{GT20} and  \cite{GP16} for the definition of the isomorphism $\mathscr I$ between the `abstract' $L^0(T\X)$ and `concrete' $L^0(T_{\rm GH}\X)$ tangent modules), while the second comes from the properties  of the abstract differential $\d v$ (see \cite[Prop. 3.5]{GPS18}).
\end{proof}
\begin{remark}\label{re:rad}{\rm This last statement has little to do with the specific geometry of $\RCD$ and $\CAT$ spaces and  the choice $p=2$. The same proof works  for any $p\in(1,\infty)$ and $\Y$ complete as soon as $\X$ is strongly rectifiable in the sense of \cite[Definition 2.18]{GT20} (see also the original definition \cite[Definition 3.1]{GP16}), uniformly locally doubling and supporting a (weak, local 1-1) Poincar\'e inequality. 

This line of thought and \cite[Proposition 2.5]{GT20} and its proof also show that, in this case, for $v$ Lipschitz  the conclusions of  \cite[Proposition 3.6]{GT20} hold in a stronger form: one can replace  the approximate $\lims$ in \cite[Definition 3.3]{GT20} with a $\lims$.  Compare with \cite[Definition 2.7]{ZZZ19}.
}\fr\end{remark}

\subsection{Lipschitz continuity and Zhang-Zhong-Zhu inequality}\label{se:zzz}

In this section  we shall work under the following assumptions:
\begin{itemize}
\item[-] $(\X,\sfd,\mm)$ is an $\RCD(K,N)$ space of essential dimension $d\in\N$.
\item[-] $U\subset\X$ is open, bounded and such that $\mm(\X\setminus U)>0$.
\item[-] $(\Y,\sfd_\Y)$  is a $\CAT(0)$ space.
\item[-] $u\in\ks(U;\Y)\subset L^2(U,\Y)$ is harmonic and $\cont(u)\subset U$ is given by Proposition \ref{prop:boundlinfty}. For convenience, we shall also fix  Borel representatives of $u$ and $\e_2[u]$.
\item[-]  $R>0$ is a fixed parameter, that we shall think of as `maximal radius'.  
\item[-] $\hat x\in U$ and $r\in(0,R)$ are fixed so that  the balls  $B:=B_r(\hat x)$,   $2B:=B_{2r}(\hat x)$ are contained in $U$. Also, we define $B'\supset B''\supset B$ as $B':=B_{3r/2}(\hat x)$ and $B'':=B_{4r/3}(\hat x)$.
\end{itemize} 
Start noticing that from the bound \eqref{eq:boundest}  and the definition \eqref{eq:defdu} we get
\[
\sfd_u(x,y)\leq \bar C\,\norm \qquad\forall x,y\in B'\qquad\text{ where }\qquad \norm:=\inf_{\o\in\Y}\sqrt{\fint_{2B}\sfd_\Y^2(u(\cdot),\o)\,\d\mm}
\]
for some $\bar C=\bar C(K^-R^2,N)$ that shall be kept fixed from here on.  We then define $f:\X^2\to \R$ as
\begin{equation}
\label{eq:deff}
f(x,y):=\left\{
\begin{array}{ll}
-\sfd_u(x,y),&\qquad\text{ if }x,y\in B',\\
- \bar C\,\norm,&\qquad\text{ otherwise}
\end{array}
\right.
\end{equation}
and notice that
\begin{equation}
\label{eq:fbdasotto}
- \bar C\,\norm\leq f(x,y)\leq 0\qquad\forall x,y\in\X.
\end{equation}
The properties of $\sfd_u$ (in particular upper semicontinuity and \eqref{eq:dut}) ensure that $f$ has the properties stated at the beginning of Section \ref{se:HLnew}. We then define $f_t$ as in formula \eqref{eq:defft}, i.e.\ we put
\[
f_t(x):=\inf_{y\in\X}f(x,y)+\frac{\sfd^2(x,y)}{2t}.
\]
Observe that  \eqref{eq:fbdasotto} gives
\begin{equation}
\label{eq:bft22}
-\bar C\,\norm\leq f_t\leq 0\qquad\ on\ \X,\ \forall t>0.
\end{equation}
Recall also that the functions $D^\pm_t(x):\X\to\R^+$ have been defined in \eqref{eq:defdpm} and put
\begin{equation}
\label{eq:defdtns}
D_t(x):=\left\{\begin{array}{ll}
D^-_t(x),&\qquad\text{ if }K\geq 0,\\
D^+_t(x),&\qquad\text{ if }K< 0.
\end{array}\right.
\end{equation}
Strongly inspired by the ideas in \cite{ZZ18},  \cite{ZZZ19}, we see that recalling the limiting property \eqref{eq:dualtilt} and \eqref{eq:liptilt}, we will be in  a good position to prove  the Zhang-Zhong-Zhu inequality if we show that
\begin{equation}
\label{eq:perzzz}
\bd f_t\restr{B''}\leq -K\frac{D_t^2}{t}\mm\qquad\forall  0<t\ll1.
\end{equation}
The Laplacian comparison estimate for the distance and the stability of upper Laplacian bounds under `inf' easily gives that $\bd f_t\restr{B''}\leq C_t$ for some `bad' constant $C_t$, hence by Lemma  \ref{le:lapcomb2} we see that \eqref{eq:perzzz} will follow if we prove that
\[
\tilde\Delta f_t\leq -K\frac{D_t^2}{t}\qquad\mm-a.e.\ on\ B''\  \forall  0<t\ll1.
\]
We are actually able to prove this only at points $x$ for which a minimizer for $f_t(x)$ is in some a priori given full-measure set of `nice' points (those for which the conclusion of Lemma \ref{le:tdF} holds) but it is not clear a priori whether this occurs for $\mm$-a.e.\ $x$. Here is where the variational principle enters into play, in line with the use of Jensen's lemma in the theory of viscosity solutions. We perturb our functions by adding $\frac{\sfd^2(\cdot,z)}n$ for some well chosen point $z\in\X$ and $n\gg1$ (this does not affect upper Laplacian bounds too much), to be sure that minimizers belong to the given set:
\begin{lemma}[Perturbation lemma]\label{le:partle}
With the same assumptions and notations introduced above, the following holds. Put $\bar T:=\frac{r^2}{72(2+4r^2+\bar C\,\norm)}$ and  let $E\subset B'$ be Borel such that $\mm(B'\setminus E)=0$. 

Then for every $t\in(0,\bar T)$ there is $\bar z_t\in B$ such that: for $\mm$-a.e.\ $x\in B''$ and every $n\in\N$, $n>0$ 
\[
\text{there is a minimizer $T_t(x)$ of }\qquad \X\ni y\quad\mapsto\quad g_{t,n}(x,y,\bar z_t):=f(x,y)+\frac{\sfd^2(x,y)}{2t}+\frac{\sfd^2(y,\bar z_t)}n
\]
and such minimizer belongs to $E\cap B'$.
\end{lemma}
\begin{proof} Let $x\in B''$ and notice that, by the choice of $\bar T$ and direct computation, for $t\in(0,\bar T)$  we have $f(x,\cdot)+\frac{\sfd^2(x,\cdot)}{2t}>c$ outside $B'$ for $c:=4r^2+1$. Thus by  item $(i)$ in Proposition \ref{prop:lapdu}, Theorem \ref{thm:lapbd}   and   Lemma \ref{le:ext} we know that  $\tilde f(x,\cdot):=(f(x,\cdot)+\frac{\sfd^2(x,\cdot)}{2t})\wedge c$ satisfies  $\bd \tilde f(x,\cdot)\leq \frac {C(K^-R^2,N)}t\mm$ and, picking $A:=\cont(u)\cup(\X\setminus U)$, that  $\tilde f(x,\cdot)\in \acm(\X)$.  

Now for any $z\in B$ define   $\tilde g_{t,n}(x,y,z):=\tilde f(x,y)+\tfrac{\sfd^2(y,z)}n$. We claim that
\begin{equation}
\label{eq:claimg}
\begin{split}
&\text{$\forall n\in\N,\ n>0,\  x\in B'',\ z\in B,\ t\in(0,\bar T)$ we have:}\\
&\text{if $y$ minimizes $\tilde g_{t,n}(x,\cdot,z)$ then it minimizes $ g_{t,n}(x,\cdot,z)$ and it belongs to $B'$.}
\end{split}
\end{equation}
To see this,  let $y$ be a minimizer for $\tilde g_{t,n}(x,\cdot,z)$ and notice that 
\[
(-\bar C\,\norm+\tfrac{\sfd^2(x,y)}{2t})\wedge c\leq\tilde g_{t,n}(x,y,z) \leq \tilde g_{t,n}(x,x,z)  \leq f(x,x)+\tfrac1n\sfd^2(x,z)\leq {\rm diam}(B'')^2\leq 4r^2.
\]
In particular, $\tilde f(x,y)\leq 4r^2<c$, hence $\tilde f(x,y)=f(x,y)+\frac{\sfd^2(x,y)}{2t}$ and thus   $\tilde g_{t,n}(x,y,z)=  g_{t,n}(x,y,z)$. Since $\tilde g_{t,n}\leq  g_{t,n}$, we see  that $y$ minimizes  $ g_{t,n}(x,\cdot,z)$. The above also implies   $\sfd^2(x,y)\leq 2t(4r^2+\bar C\,\norm)\leq\frac{r^2}{36}$. Since $x\in B''$ this implies $y\in B'$, so that \eqref{eq:claimg} is proved.

Now for every $x\in B''$ we apply Theorem \ref{prop:varregf} to $\tilde f(x,\cdot)$ to find a Borel map $T(x,\cdot):\X\to \X$ so that for $\mm$-a.e.\ $z\in\X$ the point $T(x,z)$ is the only minimizer of $\tilde g_{t,n}(x,\cdot,z)$ and 
\begin{equation}
\label{eq:abscont}
T(x,\cdot)_*\mm\ll\mm\qquad\forall x\in\X.
\end{equation}
By compactness we see that the multivalued map $(x,z)\mapsto K(x,z):=\{\text{minimizers of $\tilde g_{t,n}(x,\cdot,z)$}\}$ satisfies the assumption of the Borel selection theorem \cite[Theorem 6.9.3]{Bogachev07}, thus  there is $\tilde T:\X^2\to\X$ Borel so that $\tilde T(x,z)\in K(x,z)$ for any $x,z$. The uniqueness part of Theorem \ref{prop:varregf} for $\RCD(K,N)$ spaces forces $T(x,\cdot)=\tilde T(x,\cdot)$ $\mm$-a.e., hence \eqref{eq:abscont} gives $\tilde T_*(\mm\times\mm)\ll\mm$.
%
%
%
Since $\mm(B'\setminus E)=0$ by assumption, we have  $(\mm\times\mm)\big((B''\times B)\cap \tilde T^{-1}(B'\setminus E)\big)=0$, thus from Fubini's theorem and  \eqref{eq:claimg} we see that   for $\mm$-a.e.\ $z\in B$ we have:  for $\mm$-a.e.\ $x\in B''$ the point $\tilde T(x,z)$ is in $E\subset B'$ and minimizes $g_{t,n}(x,\cdot,z)$. Since  $n$ ranges over $\N$, 
this suffices to conclude.
\end{proof}
For $\bar T>0$, $t\in(0,\bar T)$ and $\bar z_t\in B$  as in  Lemma \ref{le:partle} above, we define  $f_{t,n}:\X\to\R$ as
\begin{equation}
\label{eq:defftn}
f_{t,n}(x):=\inf_{y\in\X}f(x,y)+\frac{\sfd^2(x,y)}{2t}+\frac1n\sfd^2(y,\bar z_t).
\end{equation}
A first, simple but useful, property of the $f_{t,n}$'s is:
\begin{proposition} With the same assumptions and notations introduced above the following holds. For $f_{t,n}$ defined as in \eqref{eq:defftn}   we have
\begin{equation}
\label{eq:subopt}
\bd f_{t,n}\restr{B''}\leq \frac {C(K^-R^2,N)}t\mm\qquad\forall t\in(0,\bar T).
\end{equation}
\end{proposition}
\begin{proof} By Lemma \ref{le:partle} and Definitions \eqref{eq:defdu}, \eqref{eq:deff} we know that for $\mm$-a.e.\ $x\in B''$ we have
\[
f_{t,n}(x)=\inf_{y\in B'\cap\cont(u)}f(x,y)+\frac{\sfd^2(x,y)}{2t}+\frac1n\sfd^2(y,\bar z).
\]
On the other hand, by Theorem \ref{thm:lapbd} and  Proposition \ref{prop:lapdu}  we know that 
\[
\bd\Big(f(\cdot,y)+\frac{\sfd^2(\cdot,y)}{2t}+\frac1n\sfd^2(y,\bar z)\Big)\restr{B''}\leq \frac {C(K^-R^2,N)}t\mm\qquad\text{ for every $y\in B'$},
\]
so that the conclusion  follows from   Lemma \ref{le:inflap2} (using the second part with $E:=\cont(u)$).
\end{proof}
To achieve \eqref{eq:perzzz} we now want to  find suitable upper bounds for $\tilde\Delta f_{t,n}$. To this aim we  follow \cite{ZZ18} and introduce, for fixed $\bar x,\bar y\in B'$, the auxiliary function $F_{\bar x,\bar y}:\X\to\R$ defined as
\[
F_{\bar x,\bar y}(z):=\left\{
\begin{array}{ll}
\frac1{\sfd_u(\bar x,\bar y)}\Big(\sfd_u^2(z,\bar x)-\sfd_{u,p}^2(z)+\frac14\sfd_u^2(\bar x,\bar y)\Big),&\qquad\text{ if }\sfd_u(\bar x,\bar y)>0\text{ and }z\in B',\\
0,&\qquad\text{ otherwise,}
\end{array}\right.
\]
where $p\in\Y$ is the midpoint of $u(\bar x),u(\bar y)$ and $\sfd_{u,p}:B'\to\R$ is defined as
\[
\begin{split}
\sfd_{u,p}(x):=&\limi_{x'\to x\atop x'\in\cont(u)}\sfd_\Y(u(x'), p)\qquad\forall x\in B'.
\end{split}
\]
The relevance of these functions is due to the next two results (analogue of \cite[Proof of Lemma 6.7, Lemma 6.4]{ZZ18}). Notice that the geometry of the target space here enters in a crucial way: in the first of these lemmas we use the kind of `parallelogram inequality' 
\begin{equation}
\label{eq:parineq}
(|ps|-|qr|)|qr|\geq \big(|pm_{qr}|^2 -|pq|^2-|m_{qr}q|^2\big)+\big(|sm_{qr}|^2 -|sr|^2-|m_{qr}s|^2\big),
\end{equation}
valid for any $p,q,r,s\in\Y$, where  $m_{qr}$ is the midpoint of $q,r$ and for brevity we wrote $|pq|$ in place of $\sfd_\Y(p,q)$ (see  \cite[Lemma 5.2]{ZZ18} for the proof), while in the second we use   property  \eqref{eq:convdacat3}.
\begin{lemma}
With the same assumptions and notations introduced above, we have
 \begin{equation}
\label{eq:fsplit}
f(x,y)\leq f(\bar x,\bar y)+F_{\bar x,\bar y}(x)+F_{\bar y,\bar x}(y),\qquad\forall \bar x,\bar y\in\cont(u),\ \forall x,y\in \X,
\end{equation}
with equality for $x=\bar x$ and $y=\bar y$.
\end{lemma}
\begin{proof}
If $\sfd_u(\bar x,\bar y)=0$ or both $x$ and $y$ are not in $B'$ the claim is obvious. Then we observe that rearranging the terms in \eqref{eq:parineq} we easily get that   \eqref{eq:fsplit} holds for any $x,y,\bar x,\bar y\in\cont(u)$ with $\sfd_u(\bar x,\bar y)>0$ and $x,y\in B'$. The conclusion for $x,y\in B'$ arbitrary follows from the definition of $\sfd_u(x,y)$ and $\sfd_{u,p}(x)$. It remains to deal with the case  $\sfd_u(\bar x,\bar y)>0$  and $x\in B'$, $y\notin B'$: if this happens \eqref{eq:fsplit} reads as
\begin{equation}
\label{eq:perc}
-\bar C\,\norm\leq f(\bar x,\bar y)+F_{\bar x,\bar y}(x).
\end{equation}
To prove this notice that we know the validity of \eqref{eq:fsplit} for $x,\bar x,\bar y$ as in \eqref{eq:perc} and the choice $y:=\bar y$. With this choice we have $F_{\bar y,\bar x}(y)=0$ and thus \eqref{eq:perc} follows from \eqref{eq:fbdasotto}.
\end{proof}
\begin{lemma}\label{le:tdF}
With the same assumptions and notations introduced above, the following holds. 

There is $E\subset B'$ Borel with $\mm(B'\setminus E)=0$ such that
\[
\tilde\Delta F_{x,y}(x)\leq 0,\qquad\forall x\in E,\ y\in B'.
\]
\end{lemma}
\begin{proof} For any $p\in\Y$  the identity $\sfd_{u,p}(\cdot)=\sfd_\Y(u(\cdot),p))$ holds on $\cont(u)$, hence $\mm$-a.e.\ on $B'$. Then   \eqref{eq:convdacat3} and Theorem \ref{thm:lapcomp} give $\bd(-\sfd_{u,p}^2)\restr{B'}\leq- 2(d+2)\e_2[u]^2\mm$ and thus  Lemma \ref{le:lapcomb2} implies
\[
\tilde\Delta (-\sfd_{u,p}^2)(x)\leq -2(d+2)\e_2[u]^2(x)\qquad\forall  x\in B'\cap\cont(u)\text{ Lebesgue point of $\e_2[u]^2$ and  $\forall p\in \Y$.} 
\]
Similarly, for $x\in\cont(u)$ the identity $\sfd_u(x,y)=\sfd_\Y(u(x),u(y))$ holds for $\mm$-a.e.\ $y\in U$, thus for any such $x$ for which the limiting property  \eqref{eq:convenvar}  holds, since $\sfd_u(x,x)=0$ we have
\[
\tilde\Delta\sfd_u^2(x,\cdot)(x)= \lims_{t\downarrow0}\frac{\h_t(\sfd_u^2(x,\cdot))(x)}{t} = \lims_{t\downarrow0}\frac1t\int \sfd_\Y(u(x),u(y))\,\d\h_t\delta_x(y)\stackrel{\eqref{eq:convenvar} }=2(d+2)\e_2[u]^2(x).
\]

The conclusion  follows  from the subadditivity of $\tilde\Delta$ (which is obvious by definition).
\end{proof}

We then have the following crucial estimate (recall also the definition \eqref{eq:defdtns}):
\begin{proposition}[Laplacian bound for the $f_t$'s]\label{prop:lapft}
With the same assumptions and notations introduced above and for  $\bar T$ as given by Lemma \ref{le:partle}, we have  
\[
\bd f_t\restr{B''}\leq -K\frac{D_t^2}{t}\mm\qquad\forall  t\in(0,\bar T).
\]
\end{proposition}
\begin{proof} We apply Lemma \ref{le:partle} with $E\subset B'$ chosen as the set of points for which the conclusions of Lemma \ref{le:tdF} hold and then define the functions $f_{t,n}$ accordingly.  By the very definitions of $f_t$ and $f_{t,n}$ and keeping also in mind that for $x\in B''$ a minimizer for $f_{t,n}$  exists in $B'$ (by Lemma \ref{le:partle}), we have that $f_{t,n}\to f_t$ uniformly on $B''$ as $n\to\infty$. 

Hence taking into account the `bad' distributional bound \eqref{eq:subopt}, Lemma \ref{le:lapcomb2} and the stability property \eqref{eq:stablubloc},   to conclude it is sufficient to find $C(K^-R^2,N)>0$  such that
\begin{equation}
\label{eq:optb}
\tilde\Delta f_{t,n}\leq \frac{C(K^-R^2,N)}n-K\frac{D_t^2}{t}\qquad\mm-a.e.\   on\  B''\qquad\forall t\in(0,\bar T),\ n\in\N.
\end{equation}
%
By Lemma \ref{le:partle}  we know  that for $\mm$-a.e.\ $\bar x\in B''\cap E$ there is $\bar y\in E$ minimizer for $f_{t,n}(\bar x)$ (recall the definition \ref{eq:defftn}). Fix such $\bar x,\bar y$ and notice that  for every $x\in B''$ we have
\begin{equation}
\label{eq:bella}
\begin{split}
f_{t,n}(x)&=\inf_{y\in\Y}\Big( f(x,y)+\frac{\sfd^2(x,y)}{2t}+\frac{\sfd^2(y,\bar z)}n\Big),\\
\text{(by \eqref{eq:fsplit})}\qquad &\leq   f(\bar x,\bar y)+F_{\bar x,\bar y}(x)+\inf_{y\in \Y}\Big(F_{\bar y,\bar x}(y)+\frac{\sfd^2(x,y)}{2t}+\frac{\sfd^2(y,\bar z)}n\Big)\\
&= f(\bar x,\bar y)+F_{\bar x,\bar y}(x)+Q_t\big(F_{\bar y,\bar x}(\cdot)+\tfrac{\sfd^2(\cdot,\bar z)}n\big)(x)
\end{split}
\end{equation}
with equality for $x=\bar x$. By the monotonicity of the heat flow (or the representation formula \eqref{eq:reprht})  it follows that for every $s>0$ we have
\[
\h_sf_{t,n}(\bar x)\leq f(\bar x,\bar y)+\h_sF_{\bar x,\bar y}(\bar x)+\h_s\big(Q_t\big(F_{\bar y,\bar x}(\cdot)+\tfrac{\sfd^2(\cdot,\bar z)}n\big)\big)(\bar x)
\]
hence subtracting \eqref{eq:bella} written for the equality case $x=\bar x$, after passing to the limit  we obtain:
\[
\tilde\Delta f_{t,n}(\bar x)\leq \tilde\Delta F_{\bar x,\bar y}(\bar x)+\tilde\Delta \big(Q_t\big(F_{\bar y,\bar x}(\cdot)+\tfrac{\sfd^2(\cdot,\bar z)}n\big)\big)(\bar x).
\]
Now observe that Lemma \ref{le:tdF}  and the fact that $\bar x\in E$ give $ \tilde\Delta F_{\bar x,\bar y}(\bar x)\leq 0$. Also, by inspecting the equality case $x=\bar x$ in \eqref{eq:bella} we see that the inf in the definition of $Q_t\big(F_{\bar y,\bar x}(\cdot)+\tfrac{\sfd^2(\cdot,\bar z)}n\big)(\bar x)$ is realized in $\bar y$, therefore  Lemma \ref{le:lappunt}, the subadditivity of $\tilde\Delta$ and the definition \eqref{eq:defdtns} give
\[
\begin{split}
\tilde\Delta \big(Q_t\big(F_{\bar y,\bar x}(\cdot)+\tfrac{\sfd^2(\cdot,\bar z)}n\big)\big)(\bar x)
&\leq \tilde\Delta  F_{\bar y,\bar x}(\bar y)+\tfrac1n \tilde\Delta {\sfd^2(\cdot,\bar z_t)}(\bar y)-K\frac{D_t^2(\bar x)}{t}.
\end{split}
\]
Now observe that since  $\bar y\in E$ we have $ \tilde\Delta  F_{\bar y,\bar x}(\bar y)\leq 0$ by Lemma \ref{le:tdF}. To bound the second term we apply Lemma \ref{le:lapcomb2}  with $g\equiv C$ and $ C=C(K^-R^2,N)$ given by Theorem \ref{thm:lapbd}  (recall that $\bar x.\bar y,\bar z_t\in B'\subset B_R(\hat x)$): we obtain  $\tilde\Delta \sfd^2(\cdot,\bar z_t)\leq C$ on $B'$. The conclusion \eqref{eq:optb} follows.
\end{proof}
We are now ready to prove the main result of the paper: 
\begin{theorem}[Zhang-Zhong-Zhu inequality and Lipschitz continuity of harmonic maps]\label{thm:main}
With the same assumptions and notations introduced above, the following holds.  

The map $u:U\to\Y$ satisfies the Zhang-Zhong-Zhu inequality
\begin{equation}
\label{eq:ZZ}
\bd\frac{|\d u|^2}2\geq K|\d u|^2\mm\qquad on\ U
\end{equation}
 and  has a representative, still denoted by $u$, that is locally Lipschitz continuous and satisfies 
\begin{equation}
\label{eq:main}
\Lip(u\restr B)\leq \frac {C(K^-R^2,N)}r\inf_{\o\in\Y}\,\sqrt{\fint_{2B}\sfd_\Y^2(u(\cdot),\o)\,\d\mm}.
\end{equation}
\end{theorem}
\begin{proof}  \ \\
\st{Rough estimates} Let $\tfrac12 B:=B_{\tfrac r2}(\hat x)$, $\tfrac14 B:=B_{\tfrac r4}(\hat x)$. We claim that for some   $C, T>0$ we have
\begin{equation}
\label{eq:roughclaim}
\big\|\tfrac{f_t}t\big\|_{L^\infty(\frac12B)}+\big\|\tfrac{D_t^2}{t^2}\big\|_{L^\infty(\frac14B)}\leq C \qquad\forall t\in(0, T).
\end{equation}
To see this, start noticing that $|\partial^-f_t|\leq\lip(f_t)$ holds everywhere, so by \eqref{eq:radelip} for $\Y=\R$ (or the analogous result in \cite{Cheeger00}) we get $|\partial^-f_t|\leq|\d f_t|$ $\mm$-a.e.. Then  Theorem \ref{thm:hlvar} and \eqref{eq:tilte2} give
\begin{equation}
\label{eq:l1contr}
\|f_t\|_{L^1( B)}=\|f_t-f_0\|_{L^1( B)}\leq \int_0^t\int _{ B}|\d f_t|^2+C\,\e_2^2[u]\,\d\mm\,\d t\leq tC+\int_0^t\int _{ B}|\d f_t|^2\,\d\mm\,\d t
\end{equation}
for every $t\in(0,\bar T)$. Now let  $\varphi:=(1-\tfrac{1}{3r}\sfd(\cdot, B))^+$ and notice that
\begin{equation}
\label{eq:anchdop}
\int\varphi^2|\d f_t|^2\,\d\mm=\int\d f_t\cdot\d(f_t\varphi^2)-f_t\varphi\d f_t\cdot\d\varphi\,\d\mm \leq\int\d f_t\cdot\d(f_t\varphi^2)+\tfrac12\varphi^2|\d f_t|^2\,\d\mm+C,
\end{equation}
where in the last inequality we used \eqref{eq:bft22} and the trivial bound $|\d\varphi|\leq \tfrac{1}{3r}\nchi_{2 B}$. Now observe that Proposition \ref{prop:lapft} and the bounds  \eqref{eq:stimadqt2}, \eqref{eq:fbdasotto}   give 
\begin{equation}
\label{eq:lapgrezzo}
\bd (-f_t)\restr {B''}\geq -(K^-\bar C\,\norm)\mm\qquad\forall  t\in(0,\bar T)
\end{equation}
and since $\varphi=0$ outside $B''$ and $f_t\varphi^2\leq 0$, by \eqref{eq:bft22} and \eqref{eq:anchdop} we get  the uniform bound
$\int_{ B}|\d f_t|^2\,\d\mm\leq C$ for any  $t\in (0,\bar T).$ We use this estimate in \eqref{eq:l1contr} to deduce $\|f_t\|_{L^1(B)}\leq tC$ and then apply Theorem \ref{thm:ellrt} using again   \eqref{eq:lapgrezzo} to get the first in \eqref{eq:roughclaim}.

For the second, we notice that arguing as for \eqref{eq:claimg} using the uniform bound \eqref{eq:bft22} we see that, possibly picking $\bar T $ smaller, for   $x\in \tfrac14B$ any minimizer $y_t$ for $f_t(x)$ belongs to $\tfrac12B$. Then the claim follows from what already proved and  the estimate \eqref{eq:stimadqt2}.
 
\st{Core argument} By the continuity in $t$ of $f_t(x)$ (item $(i)$ of Proposition \ref{prop:paseHL}) we see that $\lims_{t\downarrow0}\frac{-f_t}{t}=\inf_{n\in\N}\sup_{t\in \Q\cap(0,\frac1n)}\frac{-f_t}{t}$ pointwise on $\X$. Thus letting, for every $n\in\N$,  $\{t_{n,i}\}_{i\in\N}$ be an enumeration of $\Q\cap(0,\frac1n)$, by Proposition \ref{prop:dualtilt} we have that $\sup_i\frac{-f_{t_{n,i}}}{t_{n,i}}$ converges  to $\tfrac12\tilt(f)^2$ pointwise. Then by the first estimate in \eqref{eq:roughclaim} we see that the convergence is also in $L^1(\tfrac12B)$, hence   by a diagonalization argument we can find $N_n\in\N$ so that $\tilde f_n:=\frac{f_{t_{n,1}}}{t_{n,1}}\wedge \cdots\wedge \frac{f_{t_{n,N_n}}}{t_{n,N_n}}$ converge to $-\tfrac12\tilt(f)^2$ in $L^1(\tfrac12B)$, and then up to pass to a non-relabeled subsequence also pointwise $\mm$-a.e..

Let $A_{n,i}:=\{\frac{f_{t_{n,i}}}{t_{n,i}}=\tilde f_{n}\}\setminus \cup_{j<i}A_{n,j}\subset\X$, notice that for every $n\in\N$ the sets $(A_{n,i})_{i=1,\ldots,N_n}$ form a Borel partition of $\X$ and then put $\tilde g_n:=\sum_i\nchi_{A_{i,n}}\frac{D_{t_{n,i}}^2}{t_{n,i}^2}$. By the last claim in Proposition \ref{prop:dualtilt} and the construction we see that $\tilde g_n\to \tilt(f)^2$ $\mm$-a.e.\ on $B$, so that by the second estimate in \eqref{eq:roughclaim} the convergence is also in $L^1(\frac14B)$.

Now notice that Proposition \ref{prop:lapft} and the last claim in Lemma \ref{le:lapcomb2} (and  induction) give
\begin{equation}
\label{eq:perCacc}
\bd \tilde f_n\leq -K\tilde g_n\mm,\qquad\forall n\in\N,\qquad on \ \tfrac14B,
\end{equation}
so that the $L^1$ convergences, the uniform bounds \eqref{eq:roughclaim} and   the stability property \eqref{eq:stablubloc} imply
\begin{equation}
\label{eq:ZZlip}
\tfrac12\bd\lip^2(u)\geq  K\lip^2(u)\mm
\end{equation}
on $\frac14B$, having also  used \eqref{eq:liptilt}. By the arbitrariness of $B$ and the locality property \eqref{eq:localappr} we infer that \eqref{eq:ZZlip} holds on $U$. We now claim that $\lip^2(u)\in W^{1,2}_{loc}(U)$ and to this aim, by the arbitrariness of $B$  it is sufficient to prove that for any $V\subset\subset\frac14 B$ open we have  $\lip^2(u)\in W^{1,2}(\bar V,\sfd,\mm\restr V)$ (this is a direct consequence of the locality of the concepts of `being Sobolev' and `minimal weak upper gradients', see for instance \cite[Theorem 4.19]{AmbrosioGigliSavare11-2}). Thus let $\varphi\in\Lip_{\bs}(\X)$ be identically 1 on $V$ with support in $\frac14B$ and notice that \eqref{eq:perCacc} and the same computations in \eqref{eq:anchdop} give
\[
\begin{split}
\int\varphi^2|\d \tilde f_n|^2\,\d\mm&\leq \int -K\tilde g_n\varphi^2\tilde f_n +\tfrac12\varphi^2|\d \tilde f_n|^2+2\tilde f_n^2|\d\varphi|^2\,\d\mm.
\end{split}
\]
Since $\sup_n\|\tilde f_n\|_{L^\infty(\frac14B)},\|\tilde g_n\|_{L^\infty(\frac14B)}<\infty$ by \eqref{eq:roughclaim}, this is sufficient to derive a uniform bound on $\int\varphi^2|\d \tilde f_n|^2\,\d\mm$ and thus on $\int_{V}|\d \tilde f_n|^2\,\d\mm$ so that the lower semicontinuity of the Cheeger energy on the space $(\bar V,\sfd,\mm\restr V)$ gives our claim. 

We can thus apply Theorem \ref{thm:ellrt} to $f:=\lip^2(u)\in W^{1,2}_{loc}(U)$ using \eqref{eq:ZZlip} and conclude that
\begin{equation}
\label{eq:lipu}
\|\lip(u)\|_{L^\infty(B'')}\leq C\sqrt{\fint_{B'}\lip^2(u)\,\d\mm}\stackrel{\eqref{eq:tilte2}}\leq C\sqrt{\fint_{B'}\e_2^2[u]\,\d\mm}\stackrel{\eqref{eq:invP}}\leq  \frac {\sf C}r\,\norm,
\end{equation}
for some ${\sf C}={\sf C}(K^-R^2,N)$. Now let $(p_n)$ be a countable set dense in the essential image of $u$, put $u_n:=\sfd_\Y(u(\cdot),p_n)$ and notice that - trivially -  we have $\lip(u_n)\leq \lip(u)$ everywhere and  by item $(ii)$ in Proposition \ref{prop:baseKS}  we know that $u_n\in W^{1,2}(U)$. Hence \eqref{eq:lipu} and the simple Lemma \ref{le:lipdf} below ensure that $|\d u_n|\leq \frac {\sf C}r \,\norm$ $\mm$-a.e.\ on $B''$. We can then apply the local Sobolev-to-Lipschitz property (see \cite[Proposition 1.10]{GV21}) and deduce that $u_n$ has a representative $\tilde u_n$ that is locally Lipschitz on $B''$ with local Lipschitz constant uniformly bounded by $\frac {\sf C}r\,\norm$. Since any two points $x,y\in B$ can be joined by a curve lying on $B''$ with length bounded by $3\sfd(x,y)$ (if a geodesic $\gamma$ from $x$ to $y$ is so that $\gamma_t\notin B''$ for some $t$, its length, and thus $\sfd(x,y)$, must be at least  $\tfrac{2r}3$, while a curve passing through the center has length $\leq 2r$), we conclude that $\Lip(\tilde u_n\restr{B})\leq \frac {\sf C}r\,\norm$.

Since this is true for any $n\in\N$, it follows  that for some $N\subset B$ Borel negligible we have 
\[
|\sfd_\Y(u(x),p_n)-\sfd_\Y(u(y),p_n)|\leq \frac Cr\sfd(x,y)\,\norm\qquad\forall x,y\in B\setminus N.
\]
Taking the supremum in $n$ we conclude that the restriction of $u$ to $B\setminus N$ is $\frac {\sf C}r\,\norm$-Lipschitz, so that $u:U\to\Y$ has a locally Lipschitz representative satisfying \eqref{eq:main}, as desired. 

Finally, \eqref{eq:ZZ} follows from \eqref{eq:ZZlip} taking into account the identity \eqref{eq:radelip}.
\end{proof}
\begin{lemma}\label{le:lipdf}
Let $(\X',\sfd',\mm')$ be a complete separable metric space equipped with a non-negative and non-zero Radon measure finite on bounded sets. Let $U\subset \X'$ be open, $f\in W^{1,2}(U)$ and assume that $\lip(f)\in L^2(U)$. Then $|\d f|\leq \lip(f)$ $\mm$-a.e.\ on $U$.
\end{lemma}
\begin{proof}
Multiplying $f$ by a sequence of Lipschitz cut-off functions with support in $U$ we can, and will, reduce to the case $U=\X'$. Observe that for any curve $\gamma$ and $t\in[0,1]$ for which the metric speed $|\dot\gamma_t|$ of $\gamma$ exists and $\lip f(\gamma_t)$ is finite we have $\lims_{h\to 0}\frac{|f(\gamma_{t+h})-f(\gamma_t)|}{|h|}\leq \lip f(\gamma_t)|\dot\gamma_t|$. Now let $\ppi$ be a test plan and notice  that (by    \cite[Theorem 2.1.21]{GP19}) for $\ppi$-a.e.\ $\gamma$ we have $f\circ\gamma\in W^{1,1}(0,1)$, thus  by what just proved it is easy to see that its distributional derivative is bounded in modulus by  $\lip f(\gamma_t)|\dot\gamma_t|$. By    \cite[Theorem 2.1.21]{GP19} again this is sufficient to  conclude.
\end{proof}

\def\cprime{$'$} \def\cprime{$'$}

\end{document}